\documentclass[11pt,a4paper]{article}
\usepackage[utf8]{inputenc}
\usepackage[T1]{fontenc}
\usepackage[a4paper]{geometry}
\usepackage{lmodern}
\usepackage{graphicx}
\usepackage{enumitem}
\usepackage{amsfonts}
\usepackage{amssymb}
\usepackage{amsmath}
\usepackage{subcaption}
\usepackage{amsthm}
\usepackage[dvipsnames]{xcolor}
\usepackage{tcolorbox}
\usepackage[vcentermath]{youngtab}
\usepackage{listings}
\usepackage{caption}
\usepackage[all]{xy}
\usepackage{tikz}
\usepackage{faktor}
\usepackage{ytableau}
\usetikzlibrary{arrows}
\newtheorem{The}{Theorem}[section]
\newtheorem{Pro}[The]{Proposition}
\newtheorem{Def}[The]{Definition}
\newtheorem{Cor}[The]{Corollary}
\newtheorem{Lem}[The]{Lemma}
\newtheorem{Rem}[The]{Remark}
\newtheorem{Exe}[The]{Example}

\usepackage{hyperref}
\usepackage{fancyhdr}
\usepackage{geometry}
\usepackage{lastpage}

\newcommand{\diag}[3]{ \foreach \t in {1,...,#3} {\draw[thick] (#1+\t,#2-1) rectangle (#1+\t-1,#2);} }

\newcommand{\mdiagp}[3]{ 
\foreach \t in {1,...,#3} {\filldraw[thick, fill=lightgray, draw=black] (#1+\t,#2-1) rectangle (#1+\t-1,#2);}}

\newcommand{\mdiag}[3]{ 
\foreach \t in {1,...,3} {\filldraw[thick, fill=lightgray, draw=black] (#1+\t,#2-1) rectangle (#1+\t-1,#2);}
\foreach \t in {4,...,#3} {\draw[thick] (#1+\t,#2-1) rectangle (#1+\t-1,#2);}}

\newcommand{\diagg}[4]{ \foreach \t in {1,...,#3} {\draw[thick] (#1+\t,#2-1) rectangle (#1+\t-1,#2);} \foreach \t in {1,...,#4} {\draw[thick] (#1+\t,#2-1) rectangle (#1+\t-1,#2-2);} }

\newcommand{\mdiaggp}[4]{ \foreach \t in {1,...,#3} {\filldraw[thick, fill=lightgray, draw=black] (#1+\t,#2-1) rectangle (#1+\t-1,#2);}
\foreach \t in {1,...,#4} {\draw[thick] (#1+\t,#2-1) rectangle (#1+\t-1,#2-2);} }

\newcommand{\mdiagg}[4]{ \foreach \t in {1,...,3} {\filldraw[thick, fill=lightgray, draw=black] (#1+\t,#2-1) rectangle (#1+\t-1,#2);}
\foreach \t in {4,...,#3} {\draw[thick] (#1+\t,#2-1) rectangle (#1+\t-1,#2);}
\foreach \t in {1,...,#4} {\draw[thick] (#1+\t,#2-1) rectangle (#1+\t-1,#2-2);} }

\newcommand{\diaggg}[5]{ \foreach \t in {1,...,#3} {\draw[thick] (#1+\t,#2-1) rectangle (#1+\t-1,#2);} \foreach \t in {1,...,#4} {\draw[thick] (#1+\t,#2-1) rectangle (#1+\t-1,#2-2);}
                         \foreach \t in {1,...,#5} {\draw[thick] (#1+\t,#2-2) rectangle (#1+\t-1,#2-3);} }
                     
\newcommand{\mdiagggp}[5]{ \foreach \t in {1,...,#3} {\filldraw[thick, fill=lightgray, draw=black] (#1+\t,#2-1) rectangle (#1+\t-1,#2);}
\foreach \t in {1,...,#4} {\draw[thick] (#1+\t,#2-1) rectangle (#1+\t-1,#2-2);}
                         \foreach \t in {1,...,#5} {\draw[thick] (#1+\t,#2-2) rectangle (#1+\t-1,#2-3);} }                        
                         
\newcommand{\mdiaggg}[5]{ \foreach \t in {1,...,3} {\filldraw[thick, fill=lightgray, draw=black] (#1+\t,#2-1) rectangle (#1+\t-1,#2);}
\foreach \t in {4,...,#3} {\draw[thick] (#1+\t,#2-1) rectangle (#1+\t-1,#2);}
\foreach \t in {1,...,#4} {\draw[thick] (#1+\t,#2-1) rectangle (#1+\t-1,#2-2);}
                         \foreach \t in {1,...,#5} {\draw[thick] (#1+\t,#2-2) rectangle (#1+\t-1,#2-3);} }

\newcommand{\mdiaggggp}[6]{ \foreach \t in {1,...,#3} {\filldraw[thick, fill=lightgray, draw=black] (#1+\t,#2-1) rectangle (#1+\t-1,#2);}
\foreach \t in {1,...,#4} {\draw[thick] (#1+\t,#2-1) rectangle (#1+\t-1,#2-2);}
                         \foreach \t in {1,...,#5} {\draw[thick] (#1+\t,#2-2) rectangle (#1+\t-1,#2-3);} \foreach \t in {1,...,#6} {\draw[thick] (#1+\t,#2-3) rectangle (#1+\t-1,#2-4);} }

\def\eseq{-10} % yshift to align tikzpictures in equations                                                      

\newcommand{\strand}[2]{
	\fill (#1,#2) circle (0.2);
	\draw[thick] (#1,#2) -- ++(0,-4);
	\fill (#1,#2-4) circle (0.2);
} % straight strand

\newcommand{\slab}[2]{
	\node at #1 {\scriptsize #2};
} % strand label

\newcommand{\scirc}[2]{
	\fill (#1,#2) circle (0.2);
} % circle for strands

\newcommand{\ocross}[2]{
\scirc{#1}{#2}
\scirc{#1}{#2-4}
\scirc{#1+4}{#2}
\scirc{#1+4}{#2-4}
\draw[thick] (#1+4,#2)..controls +(0,-2) and +(0,+2) .. (#1,#2-4);
\draw[thick] (#1,#2)..controls +(0,-2) and +(0,+2) .. (#1+4,#2-4);
} % overcrossing

\newcommand{\ellk}[2]{
	\draw[white,fill=lightgray] (#1-1,#2-2) rectangle (#1+1,#2+2);
	\fill (#1,#2+2) ellipse (1.4cm and 0.2cm);
	\fill (#1,#2-2) ellipse (1.4cm and 0.2cm);
	\draw[thick] (#1-1,#2+2) -- (#1-1,#2-2);
	\slab{(#1,#2)}{$k$}
} % ellipse with k strands

\newcommand{\ellp}[2]{
	\draw[white,fill=lightgray] (#1-1,#2-2) rectangle (#1+1,#2+2);
	\fill (#1,#2+2) ellipse (1.4cm and 0.2cm);
	\fill (#1,#2-2) ellipse (1.4cm and 0.2cm);
	\draw[thick] (#1-1,#2+2) -- (#1-1,#2-2);
	\slab{(#1,#2)}{$\scriptstyle{k-1}$}
} % ellipse with k strands

\newcommand{\ellstrand}[2]{
	\draw[thick] (#1+1,#2+2) -- (#1+1,#2-2);
} % straight right strand for ellipse

\newcommand{\ellocross}[2]{
	\draw[thick] (#1+4,#2+2)..controls +(0,-2) and +(0,+2) .. (#1+1,#2-2);
	\draw[thick] (#1+1,#2+2)..controls +(0,-2) and +(0,+2) .. (#1+4,#2-2);
	\fill (#1+4,#2+2) circle (0.2);
	\fill (#1+4,#2-2) circle (0.2);
} % overcrossing right strand for ellipse

\newcommand{\ellU}[2]{
	\draw[lightgray,fill=lightgray] (#1-1,#2-2) rectangle (#1-0.5,#2+2);
	\fill (#1,#2+2) ellipse (1.4cm and 0.2cm);
	\fill (#1,#2-2) ellipse (1.4cm and 0.2cm);
	\draw[thick] (#1-1,#2+2) -- (#1-1,#2-2);
	\draw[thick] (#1-0.5,#2+2) -- (#1-0.5,#2-2);
	
	\fill (#1+12,#2+2) circle (0.2);
	\fill (#1+12,#2-2) circle (0.2);
	\fill (#1+8,#2+2) circle (0.2);
	\fill (#1+8,#2-2) circle (0.2);
	
	\draw[thick] (#1+1,#2-2)..controls +(0,2.5) and +(0,-5) .. (#1+12,#2+2);
	\draw[thick] (#1+0.5,#2-2)..controls +(0,3.5) and +(0,-5) .. (#1+8,#2+2);  
	\draw[thick] (#1,#2-2)..controls +(0,3) and +(0,-3)  .. (#1+4,#2+2);
	
	\draw[thick] (#1+1,#2+2)..controls +(0,-1) and +(-1,0) .. (#1+2.5,#2+1.2) -- (#1+9,#2+1.2) ..controls +(2.5,0) and +(0,2.5) .. (#1+12,#2-2);
	\draw[thick] (#1+0.5,#2+2)..controls +(0,-1.2) and +(-1.2,0) .. (#1+2.5,#2+0.6) -- (#1+5,#2+0.6) ..controls +(2.5,0) and +(0,2.5) .. (#1+8,#2-2);
	\draw[thick] (#1,#2+2)..controls +(0,-3) and +(0,3) .. (#1+4,#2-2);
	\fill (#1+4,#2+2) circle (0.2);
	\fill (#1+4,#2-2) circle (0.2);
}

\definecolor{mygray}{gray}{0.9}                    
\title{\bf Fused permutations algebras and degenerate affine Hecke algebras}
\numberwithin{equation}{section}
\usepackage{authblk}
\renewcommand*{\Affilfont}{\normalsize\small}
\author{DEMESMAY Yoann}

\affil{Laboratoire de math\'ematiques de Reims UMR 9008, Universit\'e de Reims Champagne-Ardenne,\newline\vspace{.9em} 
Moulin de la Housse BP 1039, 51100 Reims, France.}

{
	\makeatletter
	\renewcommand\AB@affilsepx{: \protect\Affilfont}
	\makeatother
	\affil{E-mail address:}
	\makeatletter
	\renewcommand\AB@affilsepx{, \protect\Affilfont}
	\makeatother
	\affil{yoann.demesmay-michaud@univ-reims.fr}
}
\begin{document}
\maketitle
\abstract{This paper gives an algebraic presentation of an algebra called the fused permutations algebra in the one-boundary case. It is obtained through a detailed study of the degenerate cyclotomic Hecke algebra. In particular, we prove that the fused permutations algebra is a quotient of the degenerate cyclotomic affine Hecke algebra, and we also describe a basis combinatorially in terms of signed permutations with avoiding patterns. In order to understand this quotient, we study the primitive idempotents of this degenerate cyclotomic affine Hecke algebra.}
\section{Introduction}
The symmetric group $\mathfrak{S}_n$ appears in the Schur–Weyl duality \cite{GW} describing the centralisers of tensor powers of the vector representation of the linear group $GL_N$. Considering $V$ a vector space of dimension $N$, there is a surjective morphism:
\begin{equation*}
\mathbb{C}\mathfrak{S}_n \longrightarrow End_{GL_N}(V^{\otimes n}).
\end{equation*}

The symmetric group algebra $\mathbb{C}\mathfrak{S}_n$ does not depend on $N$ and plays its role for $GL_N$ for any $N$. The dependence on $N$ of the centraliser appears in the description of the kernel of the above map. Indeed, for a given $N$, the centraliser of $GL_N$ is isomorphic to the quotient, when $n\geq N+1$, of $\mathbb{C}\mathfrak{S}_n$ by the antisymmetriser on $N +1$ points defined by:
\begin{equation*}
\sum_{\delta\in \mathfrak{S}_{N+1}}(-1)^{\ell(\delta)}\delta\in \mathbb{C}\mathfrak{S}_{N+1}\subset \mathbb{C}\mathfrak{S}_n.
\end{equation*}
where $\ell$ is the length function on the Coxeter generators, see \cite{GP}. This antisymmetriser is a minimal central idempotent of $\mathbb{C}\mathfrak{S}_{N+1}$ (the quotient is trivial if $n \leq N$) and generates the kernel for any $n \geq N + 1$. Hence we know an explicit algebraic description of the $GL_N$ centraliser on $V^{\otimes n}$.

We would like to generalize this result to a new algebra called the fused permutations algebra. This last one was introduced in \cite{CP} for this purpose. For $k\in \mathbb{Z}_{\geq 0}$, we have a surjective morphism:
\begin{equation*}
H_{k,n}  \longrightarrow End_{GL_N}(S^k(V)\otimes V^{\otimes n}),
\end{equation*}
where $S^k(V)$ is the $k$-th symmetric power and $H_{k,n}$ is called the fused permutations algebra. Again, the algebra $H_{k,n}$ does not depend on $N$ and for large $N$ is exactly the centraliser. The dependence of the centralisers on $N$ appears in the kernel of the surjective map, of which an explicit description is conjectured in \cite{CP} and proved in some cases, including the ones we will study in this paper. Therefore, what remains to be done is to find an explicit algebraic description of the fused permutations algebra $H_{k,n}$.

The fused permutations algebra can be obtained via the fused Hecke algebra which is a flat deformation of the latter. This algebra has been studied in \cite{ZP} by linking it to the cyclotomic Hecke algebra, a quotient of the affine Hecke algebra of type $A$.

The first main goal of this paper is to give an algebraic presentation of the fused permutations algebra $H_{k,n}$ and a basis in terms of type $B$ combinatorial objects.

The second main purpose of this paper, which comes from the first one, is the study of primitive idempotents of the cylotomic degenerate affine Hecke algebra. Indeed, the construction of a complete set of primitive orthogonal idempotents of this algebra is well known (see \cite{HM}), and we "linearize" some particular cases of this construction. In other words, we give a decomposition of these particular idempotents on a specific basis of the algebra.

We now describe more precisely the algebras involved in the paper.
\paragraph{The algebra $H_{k,n}$.}
The algebra of fused permutations $H_{k,n}$ can be seen in different ways. We present through the second section two ways. One consists of seeing this algebra with a basis indexed by diagrams with a certain multiplication. The other one, more algebraic, consists of seeing it as a certain subalgebra of the algebra of symmetric group on $k+n$ elements. More precisely, we consider the algebra $P_k\mathbb{C}\mathfrak{S}_{k+n}P_k$ where $P_k$ is a minimal central idempotent of $\mathbb{C}\mathfrak{S}_k$, called symmetriser on $k$ points, identified as an element of $\mathbb{C}\mathfrak{S}_{k+n}$ via the inclusion $\mathbb{C}\mathfrak{S}_k\subset \mathbb{C}\mathfrak{S}_{k+n}$. These two constructions are isomorphic as algebras. We have a good understanding of the representation theory. Namely, the algebra is semisimple and the irreducible representations are indexed by partitions $\lambda \vdash k+n$ having its first component greater than $k$.
\paragraph{The algebra $\mathcal{\hat{H}}_n^{(\kappa_1,\kappa_2)}$.} We consider the cyclotomic degenerate affine Hecke algebra $\mathcal{\hat{H}}_n^{(\kappa_1,\kappa_2)}$ defined over $\mathbb{C}$. The parameters $\kappa_1$ and $\kappa_2$ correspond to the eigenvalues of the polynomial generator $x_1$.
The algebra $\mathcal{\hat{H}}_n^{(\kappa_1,\kappa_2)}$ has a standard basis indexed by the signed permutations and we also have a good understanding of the representation theory. Namely, the semisiplicity of the algebra is characterized by the values $\kappa_1,\kappa_2$ and $n$. Furthermore, in the semisimple case, the irreducible representations are indexed by bipartitions of $n$. Among these irreducible representations, four of them are of dimension $1$ and they correspond to the following bipartitions:
\begin{equation*}\begin{array}{lllllll}
(\Box\dots\Box\,,\,\emptyset), &\quad& (\begin{array}{l}\Box\\[-0.8em] \vdots\\[-0.4em] \Box\end{array}\,,\,\emptyset), & \quad &
(\emptyset\,,\,\Box\dots\Box), &\quad& (\emptyset\,,\,\begin{array}{l}\Box\\[-0.8em] \vdots\\[-0.4em] \Box\end{array}).
\\
\ \ (1,\kappa_1) &&(-1,\kappa_1) && \ \ (1,\kappa_2) && (-1,\kappa_2)
\end{array}
\end{equation*}
Each of these one-dimensional representations corresponds to a choice of eigenvalues for the generators, as indicated above. Moreover, explicit expressions for the minimal central idempotents corresponding to these representations are studied in the third section. These idempotents are denoted $F_n^{(\alpha_1,\alpha_2)}$, where $(\alpha_1,\alpha_2)$ are the corresponding eigenvalues, and will be crucial to all our constructions.

\paragraph{The algebra $\mathcal{A}_n^{\kappa}$.} The algebra $\mathcal{A}_n^\kappa$ is obtained from $\mathcal{\hat{H}}_n^{(\kappa_1,\kappa_2)}$ by quotienting out the idempotent $F_2^{(-1,\kappa_2)}$. This quotient is one of the main object of study of the third section. Quite
naturally from its definition, the irreducible representations of the algebra $\mathcal{A}_n^\kappa$ are indexed by bipartitions with a one-row partition as the second component. Our first main result is that the algebra $\mathcal{A}_n^\kappa$ has a basis in terms of signed permutations with avoiding patterns.

\paragraph{The algebra $\mathcal{A}_n^{\kappa,(k)}$.} The last section is mainly devoted to the algebraic description of the fused permutation algebra $H_{k,n}$ . For this purpose, we define the algebra $\mathcal{A}_n^{\kappa,(k)}$ as a quotient of $\mathcal{A}_n^\kappa$ by $F_{k+1}^{(1,\kappa_1)}$. The main result of the section is that, when we specialize the indeterminates $\kappa_1$ and $\kappa_2$ by giving them specific values, the algebras $\mathcal{A}_n^{\kappa,(k)}$ and $H_{k,n}$ are isomorphic. This leads to a presentation of the fused permutation algebra in terms of generators and relations. Again, a basis of $\mathcal{A}_n^{\kappa,(k)}$ in terms of signed permutations with avoiding patterns is provided .

\bigskip

Thus, all algebras involved in the paper are obtained by quotienting by some of the idempotents mentionned above, as summarised in the following picture:
\begin{center}
\begin{tikzpicture}[scale=1]
\node at (-4.5,-2.2) {$\hat{\mathcal{H}}_n^{(\kappa_1,\kappa_2)}$};

\draw[->] (-3.8,-2.2) -- (-2,-2.2);

\node at (5,-2.2) {\footnotesize $\left(\begin{array}{l}\kappa_1=0\\ \kappa_2=k+1\end{array}\right).$};
\node at (-2.5,-1.8) {\footnotesize $F_2^{(-1,\kappa_2)}$};

\node at (-1.5,-2.2) {$\mathcal{A}_n^{\kappa}$};
\node at (2,-2.2) {$\mathcal{A}_n^{\kappa,(k)}\cong H_{k,n}$};

\draw[->,densely dashed] (-1,-2.2) -- (0.8,-2.2);
\node at (-0.1,-1.8) {\footnotesize $F_{k+1}^{(1,\kappa_1)}$};
\end{tikzpicture}
\end{center}
The full line represents a genuine quotient, while dashed line represents a quotient combined with a specialisation of the parameters $\kappa_1,\kappa_2$ as indicated in the diagram.

\paragraph{Acknowledgments:} I would like to thank my advisor Loïc Poulain d'Andecy for introducing me to these subjects, for many
helpful discussions, opinions,  support and who guided me all the way long.
\tableofcontents
\section{Fused permutations algebra}\label{s2}
This section is based on \cite{CP} to which we refer for more details.
\subsection{Combinatorial definition}
A \textbf{multiset} is the data of a pair $(E,m)$ where $E$ is a set and $m:E\longrightarrow \mathbb{Z}_{\geq 1}$ is a function.

In other words, a multiset is a set where we take into account possible repetitions. The order does not matter, as in a set. We will use a notation with double accolades for a multiset.

For example, $\lbrace\lbrace 4,4,3,4,1,2,1\rbrace\rbrace$ is a multiset containing elements of $\lbrace 1,2,3,4\rbrace$. The function is $m(1)=2$, $m(2)=1$, $m(3)=1$, $m(4)=3$.

Let $n \in \mathbb{Z}_{\geq 0}$ and $k\geq 1$.

\begin{Def}
The set $S_{k,n}$ consists of all sequences $(S_1,\ldots ,S_{n+1})$ satisfying the following conditions:
\begin{enumerate}
\item[•]
$S_1$ is a multiset containing $k$ elements of $\lbrace 1,\ldots ,n+1\rbrace$.
\item[•] The sets $\lbrace S_i\rbrace_{i\in \lbrace 2,\ldots ,n+1\rbrace}$ are singletons of elements of $\lbrace 1,\ldots ,n+1\rbrace$.
\item[•] The element $1$ appears $k$ times among $S_1,\ldots ,S_{n+1}$, while $2,\ldots ,n+1$ appear each a single time among $S_1,\ldots ,S_{n+1}$.
\end{enumerate}
An element of $S_{k,n}$ is called a \textbf{fused permutation}.
\end{Def}
For example, for $n=k=2$, $(\lbrace \lbrace 1,3\rbrace\rbrace, \lbrace\lbrace 1\rbrace\rbrace, \lbrace\lbrace 2\rbrace\rbrace)\in S_{2,2}$. 

\paragraph{Diagrammatic representation}
We can view these objects in a more geometric way.

We fix two horizontal lines of \( n+1 \) points, one above the other. We connect the points on the upper line to those on the lower line, imposing the following conditions:

There are \( k \) edges starting from the first point on top, and likewise, \( k \) edges arriving at the first point on bottom.

For \( a = 2, \ldots, n+1 \), there is exactly one edge leaving the \( a \)-th point and one edge arriving at the \( a \)-th point.

We observe that, in total, there are \( k + n \) edges.

Such a diagram corresponds to a fused permutation as follows. The multiset $S_1$ corresponds to the points at the bottom reached from the first point. Likewise, the singletons $S_a$, $a\geq 2$, correspond to the point connected to the $a$-th point on top.

\begin{Exe}
\begin{itemize}
\item[•] Let $n\in \mathbb{Z}_{\geq 0}$ and $k=1$. A fused permutation can easily be identified with a permutation on \( n+1 \) elements due to the imposed condition. Thus, \( S_{k,n} \) is indeed a generalization of \( \mathfrak{S}_{n+1} \).
\item[•] For $n=k=2$, we have $Card(S_{k,n})=7$. For example:
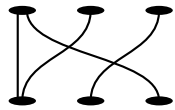
\begin{center}
\begin{tikzpicture}[scale=0.3]
\fill (21,2) ellipse (0.6cm and 0.2cm);\fill (21,-2) ellipse (0.6cm and 0.2cm);
\draw[thick] (20.8,2) -- (20.8,-2);\draw[thick] (21.2,2)..controls +(0,-2) and +(0,+2) .. (27,-2);  
\fill (24,2) ellipse (0.6cm and 0.2cm);\fill (24,-2) ellipse (0.6cm and 0.2cm);
\draw[thick] (24,2)..controls +(0,-2) and +(0,+2) .. (21,-2); 
%\draw[thick] (26.8,2)..controls +(0,-2) and +(0,+2) .. (24.2,-2);
\fill (27,2) ellipse (0.6cm and 0.2cm);\fill (27,-2) ellipse (0.6cm and 0.2cm);
%\draw[thick] (24.2,2)..controls +(0,-2) and +(0,+2) .. (26.8,-2);
\draw[thick] (27,2)..controls +(0,-2) and +(0,+2)..(24,-2);
\node at (38,0) {\text{corresponds to} $(\{1,3\},\{1\},\{2\}).$};
\end{tikzpicture}
\end{center}
\item[•] For $n=2$ and $k=3$, we have for example:
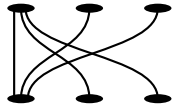
\begin{center}
\begin{tikzpicture}[scale=0.3]
\fill (21,2) ellipse (0.6cm and 0.2cm);\fill (21,-2) ellipse (0.6cm and 0.2cm);
\draw[thick] (20.7,2) -- (20.7,-2);\draw[thick] (21.2,2)..controls +(0,-2) and +(0,+2) .. (27,-2);  
\draw[thick] (21,2)..controls +(0,-2) and +(0,+2) .. (24,-2);
\fill (24,2) ellipse (0.6cm and 0.2cm);\fill (24,-2) ellipse (0.6cm and 0.2cm);
\draw[thick] (24,2)..controls +(0,-2) and +(0,+2) .. (21,-2); 
%\draw[thick] (26.8,2)..controls +(0,-2) and +(0,+2) .. (24.2,-2);
\fill (27,2) ellipse (0.6cm and 0.2cm);\fill (27,-2) ellipse (0.6cm and 0.2cm);
%\draw[thick] (24.2,2)..controls +(0,-2) and +(0,+2) .. (26.8,-2);
\draw[thick] (27,2)..controls +(0,-2) and +(0,+2)..(21.3,-2);
\node at (38,0) {\text{corresponds to} $(\{1,2,3\},\{1\},\{1\}).$};
\end{tikzpicture}
\end{center}
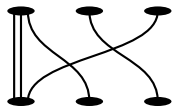
\begin{center}
\begin{tikzpicture}[scale=0.3]
\fill (21,2) ellipse (0.6cm and 0.2cm);\fill (21,-2) ellipse (0.6cm and 0.2cm);
\draw[thick] (20.7,2) -- (20.7,-2);
\draw[thick] (21,2) -- (21,-2);
\draw[thick] (21.3,2)..controls +(0,-2) and +(0,+2) .. (24,-2);
\fill (24,2) ellipse (0.6cm and 0.2cm);\fill (24,-2) ellipse (0.6cm and 0.2cm);
\draw[thick] (24,2)..controls +(0,-2) and +(0,+2) .. (27,-2); 
%\draw[thick] (26.8,2)..controls +(0,-2) and +(0,+2) .. (24.2,-2);
\fill (27,2) ellipse (0.6cm and 0.2cm);\fill (27,-2) ellipse (0.6cm and 0.2cm);
%\draw[thick] (24.2,2)..controls +(0,-2) and +(0,+2) .. (26.8,-2);
\draw[thick] (27,2)..controls +(0,-2) and +(0,+2)..(21.3,-2);
\node at (38,0) {\text{corresponds to} $(\{1,1,2\},\{3\},\{1\}).$};
\end{tikzpicture}
\end{center}
\end{itemize}
\end{Exe}

\paragraph{Multiplication of fused permutations}
We will now define an (associative) algebra based on the set \( S_{k,n} \), with the underlying vector space being the vector space over \( \mathbb{C} \), where the basis is indexed by \( S_{k,n} \). It remains to define a multiplication (associative and unital) on this basis in order to equip it with the structure of an algebra.
Let \( d, d' \in S_{k,n} \) be identified with their corresponding diagrams. We proceed in three steps:

\begin{itemize}
    \item[•] \underline{Concatenation:} We place the diagram of \( d \) above the diagram of \( d' \), identifying the bottom line of \( d \) with the top line of \( d' \).
    
    \item[•] \underline{Elimination of middle points:} There are \( k \) edges arriving and \( k \) edges leaving from the first point in the middle line. We then remove the first point of this middle line and sum over all possibilities of connecting the \( k \) arriving edges to the \( k \) leaving edges from this point (so there are \( k! \) possibilities).
    
    For the other points in the middle, one edge arrives and one edge leaves from each point: we connect these two edges (there is only one possibility per point).
    
    \item[•] \underline{Normalization:} We divide the result by the number of possibilities, which is \( k! \).
\end{itemize}
For example, let $n=1$ and $k=2$. Let us see a detailed example of a product of two elements of \( S_{2,1} \):

\begin{center}
\begin{tikzpicture}[scale=0.26]
\fill (1,2) ellipse (0.6cm and 0.2cm);\fill (1,-2) ellipse (0.6cm and 0.2cm);
\draw[thick] (0.8,2) -- (0.8,-2);\draw[thick] (1.2,2)..controls +(0,-2) and +(0,+2) .. (3.8,-2);  
\fill (4,2) ellipse (0.6cm and 0.2cm);\fill (4,-2) ellipse (0.6cm and 0.2cm);
\draw[thick] (3.8,2)..controls +(0,-2) and +(0,+2) .. (1.2,-2);
\node at (6,0) {$.$};
\fill (8,2) ellipse (0.6cm and 0.2cm);\fill (8,-2) ellipse (0.6cm and 0.2cm);
\draw[thick] (7.8,2) -- (7.8,-2);\draw[thick] (8.2,2)..controls +(0,-2) and +(0,+2) .. (10.8,-2);  
\fill (11,2) ellipse (0.6cm and 0.2cm);\fill (11,-2) ellipse (0.6cm and 0.2cm);
\draw[thick] (10.8,2)..controls +(0,-2) and +(0,+2) .. (8.2,-2); 
\node at (13,0) {$=$};
\fill (15,4) ellipse (0.6cm and 0.2cm);\fill (15,0) ellipse (0.6cm and 0.2cm);
\draw[thick] (14.8,4) -- (14.8,0);\draw[thick] (15.2,4)..controls +(0,-2) and +(0,+2) .. (17.8,0);  
\fill (18,4) ellipse (0.6cm and 0.2cm);\fill (18,0) ellipse (0.6cm and 0.2cm);
\draw[thick] (17.8,4)..controls +(0,-2) and +(0,+2) .. (15.2,0);

\fill (15,0) ellipse (0.6cm and 0.2cm);\fill (15,-4) ellipse (0.6cm and 0.2cm);
\draw[thick] (14.8,0) -- (14.8,-4);\draw[thick] (15.2,0)..controls +(0,-2) and +(0,+2) .. (17.8,-4);  
\fill (18,0) ellipse (0.6cm and 0.2cm);\fill (18,-4) ellipse (0.6cm and 0.2cm);
\draw[thick] (17.8,0)..controls +(0,-2) and +(0,+2) .. (15.2,-4); 

\node at (20.5,0) {$=\frac{1}{2}\Bigl($};
\fill (22.5,4) ellipse (0.6cm and 0.2cm);
\draw[thick] (22.3,4) -- (22.3,0);\draw[thick] (22.7,4)..controls +(0,-2) and +(0,+2) .. (25.3,0);  
\fill (25.5,4) ellipse (0.6cm and 0.2cm);
\draw[thick] (25.3,4)..controls +(0,-2) and +(0,+2) .. (22.7,0); 

\fill (22.5,-4) ellipse (0.6cm and 0.2cm);
\draw[thick] (22.3,0) -- (22.3,-4);\draw[thick] (22.7,0)..controls +(0,-2) and +(0,+2) .. (25.3,-4);  
\fill (25.5,-4) ellipse (0.6cm and 0.2cm);
\draw[thick] (25.3,0)..controls +(0,-2) and +(0,+2) .. (22.7,-4); 

\node at (27.5,0) {$+$};
\fill (29.5,4) ellipse (0.6cm and 0.2cm);
\draw[thick] (29.3,4) -- (29.3,0.5);\draw[thick] (29.7,4)..controls +(0,-2) and +(0,+2) .. (32.3,0);  
\fill (32.5,4) ellipse (0.6cm and 0.2cm);
\draw[thick] (32.3,4)..controls +(0,-2) and +(0,+2) .. (29.7,0.5); 

\draw[thick] (29.3,0.5)..controls +(0,-0.2) and +(0,+0.2) .. (29.7,-0.5);
\draw[thick] (29.7,0.5)..controls +(0,-0.2) and +(0,+0.2) .. (29.3,-0.5);

\fill (29.5,-4) ellipse (0.6cm and 0.2cm);
\draw[thick] (29.3,-0.5) -- (29.3,-4);\draw[thick] (29.7,-0.5)..controls +(0,-2) and +(0,+2) .. (32.3,-4);  
\fill (32.5,-4) ellipse (0.6cm and 0.2cm);
\draw[thick] (32.3,0)..controls +(0,-2) and +(0,+2) .. (29.7,-4); 

\node at (34,0) {$\Bigr)$};

\node at (36,0) {$=\frac{1}{2}$};
\fill (38,2) ellipse (0.6cm and 0.2cm);\fill (38,-2) ellipse (0.6cm and 0.2cm);
\draw[thick] (37.8,2) -- (37.8,-2);\draw[thick] (38.2,2) -- (38.2,-2);
\fill (41,2) ellipse (0.6cm and 0.2cm);\fill (41,-2) ellipse (0.6cm and 0.2cm);
\draw[thick] (41,2) -- (41,-2);

\node at (43,0) {$+\frac{1}{2}$};

\fill (45,2) ellipse (0.6cm and 0.2cm);\fill (45,-2) ellipse (0.6cm and 0.2cm);
\draw[thick] (44.8,2) -- (44.8,-2);\draw[thick] (45.2,2)..controls +(0,-2) and +(0,+2) .. (48,-2);  
\fill (48,2) ellipse (0.6cm and 0.2cm);\fill (48,-2) ellipse (0.6cm and 0.2cm);
\draw[thick] (48,2)..controls +(0,-2) and +(0,+2) .. (45.2,-2);
\node at (49,0){$.$};
\end{tikzpicture}
\end{center}
The algebra \( H_{k,n}\) defined above is called the \textbf{fused permutations algebra}.

It is clearly associative and it is a unital algebra: indeed, \[ 1_{H_{k,n}} := (\{ \underbrace{1, \ldots, 1}_{k \text{ times}} \}, \{ 2 \}, \ldots, \{ n+1 \}) \] is the unit of $H_{k,n}$ and its corresponding diagram is:

\begin{align}
 	&\mathbf{1} :=
 	\begin{tikzpicture}[scale=0.3,baseline={([yshift=\eseq]current bounding box.center)}]	
		\ellk{0}{0}
		\ellstrand{0}{0}
		\slab{(4,3)}{$1$}
		\strand{4}{2}
		\node at (6,0) {$\dots$};
		\slab{(8,3)}{$n$}
		\strand{8}{2}
		\node at (10   ,0) {,};
	\end{tikzpicture}
\end{align}
where the shaded area corresponds to $k$ vertical parallel edges.
\begin{Pro}
The algebra \( H_{k,n} \) has dimension:
\begin{equation}\label{eq: dim H_{k,n}}
\dim(H_{k,n}) = \sum_{i=0}^{\min(k,n)} \binom{n}{i}^2 (n-i)!.
\end{equation}
\end{Pro}
\begin{proof}
A basis of the algebra \( H_{k,n} \) is given by the diagrams, so it is sufficient to count the number of distinct diagrams.

Let \( i \in \{ 0, \ldots, \min(k,n) \} \) be the number of edges starting from the first point on top and arriving at the last \( n \) points. There must also be \( i \) edges connecting the first point at the bottom to the last \( n \) points at the top. This gives the choice of distributing these \( i \) edges among the \( n \) points, both at the top and the bottom. Finally, for the remaining \( n-i \) points, both at the top and the bottom, there is a choice for connecting them, namely \( (n-i)! \) choices. Thus, the number of diagrams for each \( i \), \( i = 0, \ldots, \min(k,n) \), leads to the result by summing over all \( i \).
\end{proof}
\subsection{Algebraic Definition}\label{ss22}
Recall that the algebra of the symmetric group \( \mathbb{C}\mathfrak{S}_n \) is a \( \mathbb{C} \)-algebra whose basis is indexed by the elements of \( \mathfrak{S}_n \). Moreover, there is a natural inclusion \( \mathbb{C}\mathfrak{S}_n \subset \mathbb{C}\mathfrak{S}_{n+1} \), where the elements of \( \mathfrak{S}_n \) (elements of the basis of \( \mathbb{C}\mathfrak{S}_n \)) are viewed in \( \mathfrak{S}_{n+1} \) as the same permutation on the first \( n \) elements, and where \( n+1 \) remains fixed.

We define an element, called the \textbf{symmetrizer} of \( \mathbb{C}\mathfrak{S}_k \), by:
\begin{equation}
P_{k}:=\frac{1}{k!}\sum_{\omega\in \mathfrak{S}_{k}}\omega \in \mathbb{C}\mathfrak{S}_{k}.
\end{equation}

This element is an idempotent of \( \mathbb{C}\mathfrak{S}_{k+n} \) (i.e., \( P_k^2 = P_k \)). Moreover, for \( i=1, \ldots, k-1 \), considering \( \delta_i \) as the transposition swapping \( i \) and \( i+1 \) and fixing the other points, we have \( \delta_i P_k = P_k \delta_i = P_k \).

\begin{Def}
We define the algebra:
\begin{equation}
\mathcal{H}_{k,n}:=P_k \cdot \mathbb{C}\mathfrak{S}_{k+n} \cdot P_k,
\end{equation}
where \( P_k \) is identified as an element of \( \mathbb{C}\mathfrak{S}_{k+n} \).
\end{Def}

\begin{Rem}
The algebra \( \mathcal{H}_{k,n} \) is a subalgebra of \( \mathbb{C}\mathfrak{S}_{k+n} \) with the unit \( P_k \) (not containing the unit of \( \mathbb{C}\mathfrak{S}_{k+n} \)).
\end{Rem}

\begin{Pro} \cite[Proposition 2.4]{CP}

The algebras \( \mathcal{H}_{k,n} \) and \( H_{k,n} \) are isomorphic.
\end{Pro}
We just give here a description of the isomorphism. For $\omega\in \mathfrak{S}_{k+n}$, we see $\omega$ as a diagram made of two rows of $k+n$ dots connected by edges according to the permutation $\omega$ (the $i$-th bottom dot is connected to the $\omega(i)$-th top dot). The multiplication in $\mathfrak{S}_{k+n}$ is then simply the concatenation of diagrams.

Then, in the diagram of $\omega\in \mathfrak{S}_{k+n}$, in each of the two rows of dots, we glue the $k$ first dots into an ellipse. We denote $[\omega]$ the corresponding fused permutation in $S_{k,n}$ with its diagrammatic representation.

Then, we have the following map:
\begin{equation} \label{def H}
\begin{array}{cccc}
\mathbf{H}: & \mathcal{H}_{k,n} & \longrightarrow & H_{k,n}\\
& P_k\omega P_k & \longmapsto & [\omega]
\end{array}
\end{equation}
which is an isomorphism of algebras.

For example, let $n=2$ and $k=3$. Given $\omega=\left(\begin{array}{ccccc}
1 & 2 & 3 & 4 & 5\\
1 & 3 & 4 & 5 & 2 
\end{array}\right)$, its corresponding diagram is:
\begin{center}
\begin{tikzpicture}[scale=0.3]
\fill (21,2) ellipse (0.2cm and 0.2cm);\fill (21,-2) ellipse (0.2cm and 0.2cm);
\node at (19,2){$\omega(i)$};
\node at (19,-2){$i$};
\draw[thick] (21,2) ..controls+(0,-2) and +(0,+2).. (21,-2);
\fill (23,2) ellipse (0.2cm and 0.2cm);\fill (23,-2) ellipse (0.2cm and 0.2cm);
\draw[thick] (23,2) ..controls+(0,-2) and +(0,+2)..  (29,-2);
\fill (25,2) ellipse (0.2cm and 0.2cm);\fill (25,-2) ellipse (0.2cm and 0.2cm);
\draw[thick] (25,2) ..controls+(0,-2) and +(0,+2)..  (23,-2);
\fill (27,2) ellipse (0.2cm and 0.2cm);\fill (27,-2) ellipse (0.2cm and 0.2cm);
\draw[thick] (27,2) ..controls+(0,-2) and +(0,+2)..  (25,-2);
\fill (29,2) ellipse (0.2cm and 0.2cm);\fill (29,-2) ellipse (0.2cm and 0.2cm);
\draw[thick] (29,2) ..controls+(0,-2) and +(0,+2)..  (27,-2);
\node at (33,0){$\in \mathfrak{S}_{3+2}$.};
\end{tikzpicture}
\end{center}

Gluing the first three dots in each of the two rows, we obtain:
\begin{center}
\begin{tikzpicture}[scale=0.3]
\node at (15,0){$\mathbf{H}(P_k\omega P_k)=[\omega]=$};
\fill (21,2) ellipse (0.6cm and 0.2cm);\fill (21,-2) ellipse (0.6cm and 0.2cm);
\draw[thick] (20.7,2) ..controls+(0,-2) and +(0,+2).. (20.7,-2);
\draw[thick] (21,2) ..controls+(0,-2) and +(0,+2).. (21,-2);
\fill (23,2) ellipse (0.6cm and 0.2cm);\fill (23,-2) ellipse (0.6cm and 0.2cm);
\draw[thick] (21.3,2) ..controls+(0,-2) and +(0,+2)..  (25,-2);
\draw[thick] (23,2) ..controls+(0,-2) and +(0,+2)..  (21.3,-2);
\draw[thick] (25,2) ..controls+(0,-2) and +(0,+2)..  (23,-2);
\fill (25,2) ellipse (0.6cm and 0.2cm);\fill (25,-2) ellipse (0.6cm and 0.2cm);
\node at (28,0){$\in H_{3,2}$.};
\end{tikzpicture}
\end{center}
We remark that, for fixed $k$, the algebras $H_{k,n}$ form a chain of algebras as $n$ varies:
\[\mathbb{C} = H_{k,0} \subset H_{k,1} \subset H_{k,2} \subset \cdots \subset H_{k,n} \subset H_{k,n+1} \subset \cdots.\]
The inclusion is given by the natural inclusion of diagrams (adding a vertical edge to the right).
\paragraph{Semisimple algebras.} We refer to \cite{L} and \cite{GW} for more details on this paragraph and the following.

A finite-dimensional algebra \( \mathfrak{A} \) over the field \( \mathbb{C} \) is said to be \textbf{semisimple} if every finite-dimensional module over \( \mathfrak{A} \) (i.e., every finite-dimensional representation of \( \mathfrak{A} \)) can be decomposed as a direct sum of simple modules (i.e., irreducible representations).

To understand the representation theory of a finite-dimensional semisimple algebra over \( \mathbb{C} \), it is therefore sufficient to study its irreducible representations. Moreover, the structure of the algebra is encoded in the set of its irreducible representations. Indeed, we have the following theorem, which is a special case of the Artin–Wedderburn theorem, see \cite[\textsection{3.3}]{GW}:
let \( \mathfrak{A} \) be a finite-dimensional semisimple algebra over \( \mathbb{C} \). Then there exists an isomorphism of algebras:
\begin{equation}\label{AW formula}
\mathfrak{A} \cong \bigoplus_k \mathrm{End}(V_k),
\end{equation}
where the direct sum runs over a set of pairwise non-isomorphic irreducible representations \( V_k \) of \( \mathfrak{A} \). We note $Irr(\mathfrak{A})$ the set of representatives of the equivalence classes (under module isomorphism) of irreducible representations of \( \mathfrak{A} \).  The isomorphism is given naturally by sending $\mathfrak{a} \in \mathfrak{A}$ to the endomorphism in $End(V_k)$ corresponding to the action of $\mathfrak{a}$ on $V_k$.

Thus, a finite-dimensional semisimple algebra is a \textbf{multi-matrix algebra}, meaning it decomposes as a direct sum of matrix algebras. According to the theorem, every element \( \mathfrak{a} \in \mathfrak{A} \) can be viewed, via the isomorphism above, as a block-diagonal matrix of size \( \sum_k \dim(V_k)\); each diagonal block corresponds to a term in the direct sum and is the image of \( \mathfrak{a} \) under the projection onto \( \mathrm{End}(V_k) \).

\paragraph{Bratteli diagram of a chain of semisimple algebras.}

A \textbf{chain of algebras} is the data of $(\mathfrak{A}_n,\iota_n)_{n\geq 0}$ where, for $n\geq 0$, $\mathfrak{A}_n$ is an algebra and $\iota_n : \mathfrak{A}_n \hookrightarrow \mathfrak{A}_{n+1}$ is an injective morphism of algebras.

Consider \(\{ \mathfrak{A}_n \}_{n \geq 0}\), a chain of finite-dimensional semisimple algebras.

Given \(V_i\) an irreducible representation of \(\mathfrak{A}_{n+1}\), the restriction of \(V_i\) to \(\iota_n(\mathfrak{A}_n)\), denoted \(Res_{\mathfrak{A}_n}(V_i)\), decomposes, by semisimplicity of the algebra, as a direct sum of irreducible representations:
\[
Res_{\mathfrak{A}_n}(V_i) \cong \bigoplus_{i \in I} m_i W_i
\]
where \(I\) denotes a subset of \(Irr(\mathfrak{A}_n)\), the \(\lbrace W_i\rbrace_{i \in I}\) are pairwise non-isomorphic irreducible representations of \(\mathfrak{A}_n\), and the \((m_i)_{i \in I}\) are called multiplicities.

The \textbf{Bratteli diagram} of the chain of algebras \(\{ \mathfrak{A}_n \}_{n \geq 0}\) is the graph defined by:
\begin{enumerate}
\item[•] A set of vertices partitioned into subsets indexed by \(n \geq 0\). We call \(n\) the \textbf{level} of the diagram. The vertices at level \(n\) are indexed by \(Irr(\mathfrak{A}_n)\).
\item[•] The edges connect only vertices between adjacent levels. Let \(V\) be an irreducible representation of \(\mathfrak{A}_n\) and \(V'\) an irreducible representation of \(\mathfrak{A}_{n+1}\). Then there are \(m\) edges connecting the vertices indexed by \(V\) and \(V'\) if and only if \(V\) appears in the decomposition \(Res_{\mathfrak{A}_n}(V')\) with multiplicity \(m\).
\end{enumerate}
Graphically, all vertices at a given level are placed on a horizontal line, with the vertices at level \(n+1\) placed below those at level \(n\).

\paragraph{Artin-Wedderburn decomposition of $H_{k,n}$.}
Let $n\in \mathbb{Z}_{\geq 0}$ and $k\geq 1$.

We know that the algebra $H_{k,n}$ is semisimple and its irreducible representations were described in \cite{CP}. They are indexed by partitions $\lambda=(\lambda_1,\dots ,\lambda_p) \vdash k+n$ such that $\lambda_1\geq k$ (that is, the first line of $\lambda$ contains at least $k$ boxes). For example, for $n=0$, there is a single irreducible representation, indexed by a line of $k$ boxes. We set $Irr(H_{k,n}):=\lbrace \lambda=(\lambda_1,\ldots ,\lambda_p)\vdash k+n \mid \lambda_1\geq k\rbrace$ and denote its irreducible representations $U_\lambda$, $\lambda\in Irr(H_{k,n})$.

Then we have the following Artin-Wedderburn decomposition for $H_{k,n}$ (see \cite[Theorem 6.5]{CP}):
\begin{equation} \label{AFP decomposition}
H_{k,n}\simeq \bigoplus_{\lambda\in Irr(H_{k,n})}End(U_\lambda).
\end{equation}

The branching rules are given by inclusion of partitions. For example, when $k = 3$, the beginning of the Bratteli diagram is:
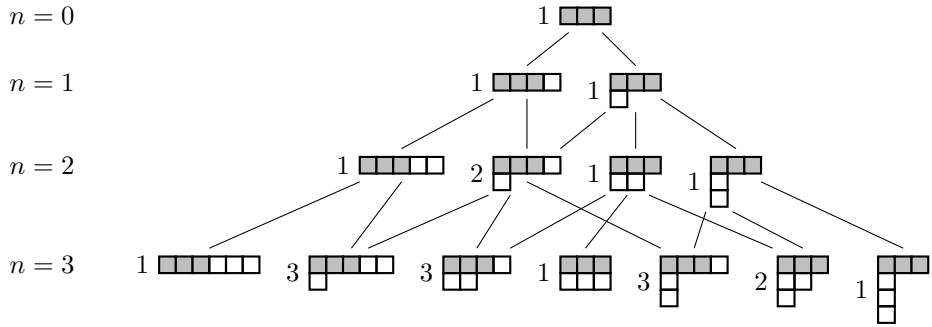
\begin{center}
 \begin{tikzpicture}[scale=0.22]
%\foreach \t in {1,2,...,5} {\draw[thick] (\t,1) rectangle (\t+1,0);}
%\foreach \t in {1,2,...,2} {\draw[thick] (\t,0) rectangle (\t+1,-1);}
%\node at (0.5,4) {$\emptyset$};
%\draw ( 0.5,3) -- (0.5, 1);
\mdiagp{-1}{0}{3};\node at (-2,-0.5) {$1$};
\draw (-0.5,-1.5) -- (-3,-3.5);\draw (1.5,-1.5) -- (3.5,-3.5);
\mdiag{-5}{-4}{4};\node at (-6,-4.5) {$1$};\mdiaggp{2}{-4}{3}{1};\node at
(1,-5) {$1$};

\draw (-5,-5.5) -- (-10.5,-8.5);\draw (-3,-5.5) -- (-3,-8.5);\draw
(1.7,-6.3) -- (-1,-8.5);\draw (3.5,-6.3) -- (3.5,-8.5);\draw (5,-5.5)
-- (9.5,-8.5);

\mdiag{-13}{-9}{5};\node at (-14,-9.5) {$1$};\mdiagg{-5}{-9}{4}{1};\node
at (-6,-10) {$2$};\mdiaggp{2}{-9}{3}{2};\node at (1,-10)
{$1$};\mdiagggp{8}{-9}{3}{1}{1};\node at (7,-10.5) {$1$};

\draw (-13,-10.5) -- (-22,-14.5);\draw (-10.5,-10.5) --
(-13.5,-14.5);\draw (-5.3,-11.3) -- (-12.5,-14.5);\draw (-4,-11.3) --
(-6,-14.5);\draw (-3,-10.5) -- (5,-14.5);
\draw (1.7,-11.3) -- (-4,-14.5);\draw (3,-11.3) -- (0.5,-14.5);\draw
(4.3,-11.3) -- (11.7,-14.5);\draw (7.7,-12.3) -- (7,-14.5);\draw
(9.3,-12.3) -- (13.5,-14.5);
\draw (11,-10.5) -- (19.5,-14.5);

\node at (-26,-15.5) {$1$};\mdiag{-25}{-15}{6};
\node at (-17,-16) {$3$};\mdiagg{-16}{-15}{5}{1};

\node at (-9,-16) {$3$};\mdiagg{-8}{-15}{4}{2};

\node at (-2,-16) {$1$};\mdiaggp{-1}{-15}{3}{3};

\node at (4,-16.5) {$3$};\mdiaggg{5}{-15}{4}{1}{1};\node at (11,-16.5)
{$2$};\mdiagggp{12}{-15}{3}{2}{1};\node at (17,-17)
{$1$};\mdiaggggp{18}{-15}{3}{1}{1}{1};

%\draw[thin, fill=gray,opacity=0.2] (3.5,-16)..controls +(0,12) and
%+(0,12) .. (22.5,-16) .. controls +(0,-12) and +(0,-12) ..
%(3.5,-16);\node at (18,-7) {$gl(2)$};
%\draw[thin, fill=gray,opacity=0.2] (16,-16)..controls +(0,5) and
%+(0,5) .. (21.5,-16) .. controls +(0,-5) and +(0,-5) .. (16,-16);\node
%at (19,-11.5) {$gl(3)$};

\node at (-32,-0.5) {$n=0$};\node at (-32,-4.5) {$n=1$};\node at
(-32,-9.5) {$n=2$};\node at (-32,-15.5) {$n=3$};
\end{tikzpicture}
\end{center}
We have shaded the three fixed boxes in the first row of each partition. Next to each partition is the dimension of the corresponding irreducible representation. We emphasize that the dimension is not the number of standard tableaux strictly speaking, but is the number of standard fillings of the non-shaded boxes by $1,\dots,n$.

We need to be more precise on the dimension of $U_\lambda$, $\lambda\in Irr(H_{k,n})$. 

Fix $p,l\in \mathbb{Z}_{\geq 0}$.

A sequence of non-negative integers $\nu = (\nu_1, \ldots , \nu_l)$ such that $\nu_1+\ldots+\nu_l =p$ is called a \textbf{composition} of $p$. We say that the \textbf{size} $\lvert \nu\rvert$ is equal to $p$. We make no difference between $\nu$ and the same sequence where we added some parts equal to $0$ at the end.

A Young tableau is called \textbf{semistandard} if the numbers are weakly ascending along rows and strictly ascending down columns of the Young diagram.

Let $n\in \mathbb{Z}_{\geq 0}$ and $k\geq 1$.

Let $\mathbf{t}$ be a Young tableau of size $k+n$ with, as entries, values in $\lbrace 1,2,\dots ,n+1\rbrace$. For $a=1,\ldots ,n+1$, let $\nu_a$ be the number of times the integer $a$ appears in the tableau $\mathbf{t}$. The sequence $\nu=(\nu_1,\ldots ,\nu_{n+1})$ forms a \textbf{composition} of $k+n$. We say that $\mathbf{t}$ is a tableau of \textbf{weight} $\nu$. 

Given $\lambda\vdash k+n$, we set:
\begin{equation*}
SSTab(\lambda,k)=\lbrace \text{semistandard Young tableaux of shape}~\lambda~\text{and weight}~(k,1,\ldots ,1)\rbrace.
\end{equation*}
Then we have the following equality (\cite[Theorem 6.5]{CP}):
\begin{equation*}
\dim(U_\lambda)=\lvert SSTab(\lambda,k)\rvert.
\end{equation*}

For instance, for $n=2$ and $k=3$, there are exactly two semistandard tableaux of weight $(3,1,1)$:
\[\begin{ytableau}
1 & 1 & 1 & 2\\
3
\end{ytableau}\qquad \text{and} \qquad \begin{ytableau}
1 & 1 & 1 & 3\\
2
\end{ytableau}\]
Thus, for $\lambda=(4,1)\vdash 3+2$, we have $\dim(U_\lambda)=2$.
\section{Cyclotomic degenerate affine Hecke algebra}
We refer to \cite{K} or \cite{HM} for this section for more details.
\subsection{Definitions and properties}

Let \( n \in \mathbb{Z}_{\geq 0}\). Fix a pair \( \kappa = (\kappa_1,\kappa_2) \in \mathbb{C}^2 \).

\begin{Def} \label{AHDC}
The \textbf{cyclotomic degenerate affine Hecke algebra} with cyclotomic parameter \( \kappa \) is the unital associative algebra \( \hat{\mathcal{H}}_n^{\kappa} \) generated by \( x_1,x_2,\ldots,x_n \) (Jucys–Murphy elements) and \( s_1,\ldots,s_{n-1} \) (Coxeter elements), subject to the following relations:
\begin{equation} \label{Relation1}
(x_1-\kappa_1)(x_1-\kappa_2)=0, 
\end{equation}
\begin{equation}\label{R1}
x_rs_{\ell}=s_{\ell}x_r \quad \text{for} \quad r\neq \ell,\ell+1,
\end{equation}
\begin{equation}\label{R_2}
x_rx_t = x_tx_r, \quad r,t=1,\ldots,n,
\end{equation}
\begin{equation}\label{eq:HC}
x_{r+1} = s_r x_r s_r + s_r, \quad r=1,\ldots,n-1, 
\end{equation}
\begin{equation}\label{S_1}
s_k^2=1 \quad \text{for} \quad k=1,\ldots,n-1,
\end{equation}
\begin{equation}\label{S_2}
s_i s_j = s_j s_i \quad \text{if} \quad |i-j| > 1,
\end{equation}
\begin{equation} \label{Relation2}
s_k s_{k+1} s_k = s_{k+1} s_k s_{k+1} \quad \text{for} \quad k=1,\ldots,n-2.
\end{equation}
\end{Def}
By convention, we have $\hat{\mathcal{H}}_0^{\kappa}=\mathbb{C}$.

One sees that relation (\ref{eq:HC}) shows that the algebra is generated by \( x_1, s_1, \ldots, s_{n-1} \), but we keep the generators \( x_2, \ldots, x_n \) for convenience. Another presentation of this algebra  will be given later with only these generators and less defining relations.

Obviously, the algebra does not depend on the order of the pair \(\kappa\).
\paragraph{Basic properties of $\hat{\mathcal{H}}_n^{\kappa}$.}
We start by giving another presentation of $\hat{\mathcal{H}}_n^{\kappa}$ which will be useful for the last section.
\begin{Pro} \cite[Proposition 7.8.1]{K}

The algebra \( \hat{\mathcal{H}}_n^{\kappa} \) is generated by \( x_1\) and \( s_1,\ldots,s_{n-1} \), subject only to the following relations:
\begin{equation} \label{eq: DR1}
(x_1-\kappa_1)(x_1-\kappa_2)=0, 
\end{equation}
\begin{equation} \label{eq: DR2}
x_1s_{\ell}=s_{\ell}x_1 \quad (2\leq \ell \leq n-1),
\end{equation}
\begin{equation} \label{eq: DR3}
x_1(s_1x_1s_1+s_1)=(s_1x_1s_1+s_1)x_1,
\end{equation}
\begin{equation} \label{eq: DR4}
s_k^2=1, \quad
s_k s_m = s_m s_k, \quad 
s_k s_{k+1} s_k = s_{k+1} s_k s_{k+1},
\end{equation}
for all admissible $k,m$ with $\lvert k-m\rvert >1$.
\end{Pro}
For \( \alpha := (\alpha_1, \alpha_2, \ldots, \alpha_n) \in (\mathbb{Z}_{\geq 0})^n \), define \( x^\alpha := x_1^{\alpha_1} x_2^{\alpha_2} \ldots x_n^{\alpha_n} \).

Recall that \( \mathfrak{S}_n \) is the symmetric group on \( n \) elements. For \( 1 \leq r < n \), let \( \delta_r := (r, r+1) \) be the transposition swapping \( r \) and \( r+1 \), and fixing the other points. Then \( \{\delta_1,\ldots,\delta_{n-1}\} \) is the standard set of Coxeter generators for \( \mathfrak{S}_n \). A \textbf{reduced expression} for an element \( \omega \in \mathfrak{S}_n \) is a minimal-length word \( \omega = \delta_{r_1} \ldots \delta_{r_k} \), with \( 1 \leq r_j < n \) and \( 1 \leq j \leq k \). If \( \omega = \delta_{r_1} \ldots \delta_{r_k} \) is reduced, define \( s_\omega := s_{r_1} \ldots s_{r_k} \). Then \( s_\omega \) is independent of the chosen reduced expression by Matsumoto’s lemma (cf. \cite{GP}, Theorem 1.2.2), since the braid relations are satisfied in \( \hat{\mathcal{H}}_n^{\kappa} \).

A remarkable result is the fact that the set:
\begin{equation} \label{base}
\left\{ x^\alpha s_\omega \mid \alpha \in \{0,1\}^n, \omega \in \mathfrak{S}_n \right\}
\end{equation}
is a basis for \( \hat{\mathcal{H}}_n^{\kappa} \) (cf. (\cite[Theorem 7.5.6]{K}).

In particular, \( \hat{\mathcal{H}}_n^{\kappa} \) has dimension:  
\begin{equation}
\label{dimension}
\dim(\hat{\mathcal{H}}_n^{\kappa})=2^n n!.
\end{equation}

For \(n\ge2\), we naturally identify \(\hat{\mathcal{H}}_{n-1}^\kappa\) with the subalgebra of \(\hat{\mathcal{H}}_n^\kappa\) generated by \(x_1,\dots,x_{n-1}\) and \(s_1,\dots,s_{n-2}\). This is possible thanks, for example, to the knowledge of the basis \eqref{base}.

Then, for fixed $k$, the algebras $\hat{\mathcal{H}}_n^\kappa$ form a chain of algebras as $n$ varies:
\[\mathbb{C} = \hat{\mathcal{H}}_0^{\kappa} \subset \hat{\mathcal{H}}_1^{\kappa} \subset \hat{\mathcal{H}}_2^{\kappa} \subset \cdots \subset \hat{\mathcal{H}}_n^{\kappa} \subset \hat{\mathcal{H}}_{n+1}^{\kappa} \subset \cdots.\]

Furthermore, we can see that
\( \mathbb{C}\mathfrak{S}_n \) is both a subalgebra and a quotient of \( \hat{\mathcal{H}}_n^{\kappa} \). Indeed, by (\ref{base}), the basis elements with \( \alpha = (0,\ldots,0) \) form a basis of a subalgebra isomorphic to \( \mathbb{C}\mathfrak{S}_n \).

Now define the map:
\begin{equation} \label{pi}
\begin{array}{cccc}
\pi: & \hat{\mathcal{H}}_n^{\kappa} & \longrightarrow & \mathbb{C}\mathfrak{S}_n \\
& s_i & \longmapsto & \delta_i \\
& x_1 & \longmapsto & 0
\end{array}
\end{equation}
for \( i = 1, \ldots, n-1 \).

This is a morphism of algebras (the relations are clearly satisfied in \( \mathbb{C}\mathfrak{S}_n \)). Surjectivity is clear since elementary transpositions generate the group algebra of the symmetric group, hence the quotient.

\begin{Rem}
The morphism \( \pi \) in (\ref{pi}) maps \( x_k \) to \( J_k := \sum_{i=1}^{k-1} (i,k) \), for \( k = 2, \ldots, n \), which are the Jucys–Murphy elements of \( \mathbb{C}\mathfrak{S}_n \), known for their elegant properties (see \cite{VO} or \cite{M}).
\end{Rem}

The algebra $\hat{\mathcal{H}}_n^{\kappa}$ is called a cyclotomic Hecke algebra of level $2$ since the relation \eqref{eq: DR1} is a quadratic characteristic relation for $x_1$. For what is recalled below about $\hat{\mathcal{H}}_n^{\kappa}$, see \emph{e.g.} \cite{K}.

We will see that the algebra $\hat{\mathcal{H}}_n^{\kappa}$ has a basis labelled by the elements of the Coxeter group of type $B_n$. We will abuse notations and denote $B_n$ this Coxeter group. Its elements can be viewed as signed permutations, that is, those permutations $\omega$ on the set $\{-n,-n+1,\dots,-1,1,2,\dots,n\}$ such that $\omega(-i)=-\omega(i)$ for all $i \in \{1,2,\dots,n\}$. We can represent the elements of $B_n$ by words $b=b_1b_2\dots b_n$ where each of the numbers $1,2,\dots,n$ appears once and is possibly barred (see \textit{e.g.}\ \cite{JS}). In this representation, $b_i$ is the image of $i$ by $\omega$, where the bar notation is understood as a negative sign. The group $B_n$ contains $n!2^n$ elements. 

We denote by $\tau_i$ the transposition of $i$ and $i+1$ in $B_n$ (which thus also transposes $-i$ and $-(i+1)$), and by $\tau_0$ the transposition of $-1$ and $1$. The group $B_n$ is generated by $\tau_i$ with $i=0,\dots,n-1$ subject only to the following relations (see \emph{e.g.} \cite{GP}):
\begin{equation}\label{R1 Bn}
	\tau_i^2=1,
\end{equation}
\begin{equation}\label{R2 Bn}
	\tau_k\tau_{k+1}\tau_k=\tau_{k+1}\tau_k\tau_{k+1},
\quad
	\tau_i\tau_j=\tau_j\tau_i,
\quad
	\tau_0\tau_1\tau_0\tau_1=\tau_1\tau_0\tau_1\tau_0,
\end{equation}
for all admissible $k=1,\dots ,n-1$, $i,j=0,\ldots ,n-1$ such that $\lvert i-j\rvert >1$.

The following product of sets (where the product of two sets $A\cdot B$ is understood as $A\cdot B=\lbrace ab\mid a\in A,b\in B\rbrace$) provides a normal form for the elements of the group $B_n$ (\cite[Part I]{JS}):
\begin{equation}\label{group set}
\left\{\begin{array}{c} 
1,\\ \tau_0
\end{array}\right\} 
\cdot 
\left\{\begin{array}{c} 
1,\\ \tau_1, \\ \tau_1\tau_0, \\ \tau_1\tau_0\tau_1 \end{array}\right\}
\cdot 
\left\{\begin{array}{c} 
1,\\ \tau_2, \\ \tau_2\tau_1, \\ \tau_2\tau_1\tau_0, \\ \tau_2\tau_1\tau_0\tau_1,\\ \tau_2\tau_1\tau_0\tau_1\tau_2\\ \end{array}\right\}
\cdot\ \ldots\ \cdot 
\left\{\begin{array}{c} 1,\\ \tau_{n-1}, \\ \vdots \\ \tau_{n-1}\dots \tau_1\tau_0,\\ \tau_{n-1}\dots \tau_1\tau_0\tau_1, \\ \vdots \\ \tau_{n-1}\dots \tau_1\tau_0\tau_1\dots \tau_{n-1} \end{array}\right\}. 
\end{equation} 
For an element $\omega\in B_n$ written in its normal form \eqref{group set} $\omega=\tau_{i_1}\tau_{i_2}\dots \tau_{i_k}$, $k\geq 0$, and denoting $s_0:=x_1$, we set $s_{\omega}=s_{i_1}\dots s_{i_k}$ in $\hat{\mathcal{H}}_n^{\kappa}$.

We introduce the following notation for \( 0\leq m \leq n \):
\[
[n,m] = s_n \dots s_{m+1} s_m,\quad [n,0]=s_n\dots s_1x_1, \quad\quad [n,-m] = s_n \dots s_1 x_1 s_1 \dots s_m.
\]
We consider the following product of sets:

\begin{equation}\label{basis set}
\left\{\begin{array}{c} 
1,\\ x_1
\end{array}\right\} 
\cdot 
\left\{\begin{array}{c} 
1,\\ s_1, \\ s_1x_1, \\ s_1x_1s_1 \end{array}\right\}
\cdot 
\left\{\begin{array}{c} 
1,\\ s_2, \\ s_2s_1, \\ s_2s_1x_1, \\ s_2s_1x_1s_1, \\ s_2s_1x_1s_1s_2 \end{array}\right\}
\cdot\ \ldots\ \cdot 
\left\{\begin{array}{c} 1,\\ s_{n-1}, \\ \vdots \\ s_{n-1}\dots s_1x_1,\\ s_{n-1}\dots s_1x_1s_1, \\ \vdots \\ s_{n-1}\dots s_1x_1s_1\dots s_{n-1} \end{array}\right\}. 
\end{equation} 
One sees that the elements of \eqref{basis set} are exactly the $s_\omega$ for $\omega\in B_n$ written in its normal form \eqref{group set}.

Then, the elements appearing in the above product are of the form:
\begin{equation}\label{Standard form}
[n_1,m_1]\dots [n_k,m_k], ~\text{with} ~ 0 \leq n_1 < n_2 < \dots < n_k \leq n-1, ~ \text{and} ~ |m_i| \leq n_i.
\end{equation}
\begin{Pro} \label{standard}
The elements of \eqref{Standard form} form a basis of the algebra \( \hat{\mathcal{H}}_n^\kappa \).
\end{Pro}

\begin{proof}
We show that these elements span the vector space \( \hat{\mathcal{H}}_n^\kappa \).

To do this, we need the following lemma:
\begin{Lem}
Every element of \( \hat{\mathcal{H}}_n^\kappa \) can be written as a linear combination of elements of the form \( \omega [n-1,j] \), where \( \omega \in \hat{\mathcal{H}}_{n-1}^\kappa \) and \( |j| \leq n-1 \).
\end{Lem}

\begin{proof}
We verify that multiplication of the elements $\omega [n-1,j]$ on the right by generators of the algebra $x_1,s_1,\ldots ,s_{n-1}$ gives a linear combination of elements of the same form.

Let $p\in \lbrace 1,\ldots n-1\rbrace$ and $j\in \mathbb{N}$ such that $j\leq n-1$.

We consider different cases:

\begin{enumerate}
\item[•] \underline{Multiplication by \( s_p \)}: We use braid and commutation relations.
\begin{enumerate}
\item \textit{Case \( p > j+1 \)}: We have:
\begin{equation*}
\omega [n-1,-j]s_p = \omega s_{n-1} \ldots s_{p}s_{p-1}s_{p} \ldots s_1 x_1 s_1 \ldots s_j = \underbrace{\omega s_{p-1}}_{\in \hat{\mathcal{H}}_{n-1}^\kappa} [n-1,-j],
\end{equation*}
and similarly, we have:
\begin{align*}
\omega [n-1,j]s_p &= \underbrace{\omega s_{p-1}}_{\in \hat{\mathcal{H}}_{n-1}^\kappa} [n-1,j].
\end{align*}

\item \textit{Case \( p = j+1 \)}: We have:
\[
\omega [n-1,-j]s_{j+1} = \omega [n-1,-(j+1)], \quad 
\omega [n-1,j]s_{j+1} = \underbrace{\omega s_j}_{\in \hat{\mathcal{H}}_{n-1}^\kappa} [n-1,j].
\]

\item \textit{Case \( p = j \)}: We obtain:
\[
\omega [n-1,-j]s_j = \omega [n-1,-(j-1)], \quad 
\omega [n-1,j]s_j = \omega [n-1,j+1].
\]

\item \textit{Case \( p < j \)}: We obtain:
\begin{align*}
\omega [n-1,-j]s_p &= \omega s_{n-1}\dots s_1 x_1 s_1\dots s_p s_{p+1}s_p \dots s_j \\
&= \omega s_{n-1}\dots s_{p+1} s_p s_{p+1}\dots s_1 x_1 s_1\dots s_j \\
&= \underbrace{\omega s_p}_{\in \hat{\mathcal{H}}_{n-1}^\kappa} [n-1,-j],
\end{align*}
and finally, we have:
\[
\omega [n-1,j]s_p = 
\begin{cases}
\omega [n-1,j-1] & \text{if } p = j-1, \\
\underbrace{\omega s_p}_{\in \hat{\mathcal{H}}_n^\kappa} [n-1,j] & \text{otherwise}.
\end{cases}
\]
\end{enumerate}

\item[•] \underline{Multiplication by \( x_1 \)}:
For \( j = 0 \), since $x_1^2=(\kappa_1+\kappa_2)x_1-\kappa_1\kappa_2$ from (\ref{eq: DR1}), we have: 
\begin{align*}
\omega [n-1,0]x_1 &= \omega [n-1,1]x_1^2\\
&= (\kappa_1+\kappa_2)\omega[n-1,0]-\kappa_1\kappa_2\omega [n-1,1].
\end{align*}

Using the relations of the algebra, one sees:
\begin{equation}\label{(x_1s_1)^2}
(s_1 x_1)^2 = (x_1 s_1)^2 + x_1 s_1 - s_1 x_1.
\end{equation}
Then, from (\ref{(x_1s_1)^2}), we have, for $j\geq 1$:
\begin{align*}
\omega[n-1,-j]x_1&=\omega(s_{n-1}\dots s_2s_1x_1s_1x_1s_2\dots s_j)\\&=\omega (s_{n-1}\ldots s_2x_1s_1x_1s_1\ldots s_j+s_{n-1}\ldots s_2x_1s_1\ldots s_j\\ &-s_{n-1}\ldots s_1x_1s_2\ldots s_j).
\end{align*}
We can rewrite the elements on the right side of the last equality as:
\[
\omega s_{n-1}\ldots s_2 x_1 s_1 x_1 s_1 \ldots s_j = \underbrace{\omega x_1}_{\in \hat{\mathcal{H}}_{n-1}^\kappa} [n-1,-j],
\]
\[
\omega s_{n-1}\ldots s_2 x_1 s_1 \ldots s_j = \underbrace{\omega x_1 s_1 \ldots s_{j-1}}_{\in \hat{\mathcal{H}}_{n-1}^\kappa} [n-1,1],
\]
\[
\omega s_{n-1} \ldots s_1 x_1 s_2 \ldots s_j = \underbrace{\omega s_1\ldots s_{j-1}}_{\in \hat{\mathcal{H}}_{n-1}^{\kappa}} [n-1,0].
\]

Lastly,
\[
\omega [n-1,j]x_1 =
\begin{cases}
\omega [n-1,0] & \text{if } j=1, \\
\underbrace{\omega x_1}_{\in \hat{\mathcal{H}}_{n-1}^\kappa} [n-1,j] & \text{otherwise}.
\end{cases}
\]
\end{enumerate}
\end{proof}

Using the lemma and a straightforward induction, the elements described span \( \hat{\mathcal{H}}_n^\kappa \). Moreover, for each \( k \in \{1,\ldots,n\} \), the set 
\[
\left\{\begin{array}{c} 
1,\\ s_{k-1}, \\ \vdots \\ s_{k-1} \dots s_1 x_1,\\ s_{k-1}\dots s_1 x_1 s_1, \\ \vdots \\ s_{k-1}\dots s_1 x_1 s_1 \dots s_{k-1} 
\end{array}\right\}
\]
has cardinality \( 2k \), so the product over all \( k \in \{1,\ldots,n\} \) yields a set of cardinality \( 2^n n! \) which is equal to the dimension of $\hat{\mathcal{H}}_n^{\kappa}$ by (\ref{dimension}). Therefore, this set forms a basis for the space.
\end{proof}
Thus, the product of sets \eqref{basis set} forms a basis of the algebra $\hat{\mathcal{H}}_n^{\kappa}$ and then the algebra $\hat{\mathcal{H}}_n^{\kappa}$ has a basis labelled by elements of the Coxeter group of type $B_n$, namely the set $\lbrace s_\omega\rbrace_{\omega\in B_n}$.

\begin{Rem} \label{Rem braid}
The element $s_\omega\in \hat{\mathcal{H}}_n^{\kappa}$ depends on the reduced form of $\omega\in B_n$. Indeed, for instance, in $B_n$, $n\geq 2$, we have $\omega=(\tau_0\tau_1)^2=(\tau_1\tau_0)^2$ but, in $\hat{\mathcal{H}}_n^{\kappa}$, if $\omega$ is written in its standard form \eqref{group set}, then, using relation \eqref{eq: DR3}, we have: 
\[s_\omega=(x_1s_1)^2=(s_1x_1)^2+s_1x_1-x_1s_1,\]
which is not equal to $(s_1x_1)^2$.

Thus, the braid relation $(x_1s_1)^2=(s_1x_1)^2$ is not satisfied in $\hat{\mathcal{H}}_n^{\kappa}$ and then the assumptions of Matsumoto's lemma are not satisfied (cf. \cite[Theorem 1.2.2]{GP}).
\end{Rem}
\subsection{Representations of the algebra \( \hat{\mathcal{H}}_n^{\kappa} \)} 
\paragraph{Bipartitions and (standard) tableaux.}

A \textbf{partition} of a non-negative integer $d$ is a non-increasing sequence of positive integers $\lambda=(\lambda_1,\ldots ,\lambda_p)$ such that $\lvert \lambda\rvert =\lambda_1+\ldots +\lambda_p=d$. We do not distinguish between a partition and the same partition extended by adding zero components, i.e., $(\lambda_1,\ldots ,\lambda_p)=(\lambda_1,\ldots ,\lambda_p,0,\ldots ,0)$. Let $\mathcal{P}_d$ denote the set of all partitions of $d$.

We identify partitions with their \textbf{Young diagrams}: the Young diagram of $\lambda$ is a left-justified array of rows of boxes such that the $j$-th row (we count from top to bottom) contains  $\lambda_j$ boxes. The number of non-empty rows is the length $\ell(\lambda)$ of $\lambda$. By convention, the empty set $\emptyset$ is the only partition of $n=0$.

A \textbf{bipartition} of a non-negative integer $n$ is a pair $\boldsymbol{\lambda}=(\lambda^{(1)},\lambda^{(2)})$ of partitions such that $\lvert \lambda^{(1)}\rvert +\lvert \lambda^{(2)}\rvert=n$. We identify a bipartition $\boldsymbol{\lambda}$ with its \textbf{diagram}, which is the set of nodes
\[
[\![\boldsymbol{\lambda}]\!]:=\lbrace (l,r,c)\mid 1\leq c\leq \lambda_r^{(l)},~l=1,2\rbrace,
\]
that can be visualized as a pair of Young diagrams. For example, for $\boldsymbol{\lambda}=((3,1,1), (2,2,1))$, a bipartition of $10$, we have:
\[
[\![\boldsymbol{\lambda}]\!]=\left(
\begin{ytableau}
       ~ & ~ & ~\\
       ~ \\
       ~ 
\end{ytableau}\quad ,\quad\begin{ytableau}
~ & ~\\
~ & ~\\
~~
\end{ytableau}
\right)
\]

Here, $l$, $r$, and $c$ denote the \textbf{component}, \textbf{row}, and \textbf{column} respectively.

The set of bipartitions of $n$ is partially ordered via:
\begin{equation}
\boldsymbol{\lambda}\unrhd\boldsymbol{\mu}\Longleftrightarrow \sum_{k=1}^{l-1}\lvert \lambda^{(k)}\rvert+\sum_{j=1}^i\lambda_j^{(l)}\geq \sum_{k=1}^{l-1}\lvert \mu^{(k)}\rvert +\sum_{j=1}^i\mu_j^{(l)}, \quad l=1,2,\quad i\geq 1.
\end{equation}
If $\boldsymbol{\lambda}\unrhd\boldsymbol{\mu}$ and $\boldsymbol{\lambda}\neq\boldsymbol{\mu}$, we write $\boldsymbol{\lambda}\rhd\boldsymbol{\mu}$.

Let $Par_2(n)$ denote the set of bipartitions of $n$.

Fix $\boldsymbol{\lambda}\in Par_2(n)$. A \textbf{$\boldsymbol{\lambda}$-tableau} is a bijection $t: [\![\boldsymbol{\lambda}]\!] \longrightarrow \lbrace 1,2,\ldots ,n\rbrace$, identifying $[\![\boldsymbol{\lambda}]\!]$ with its set of nodes. For example,
\begin{equation} \label{tableaux}
\left(
\begin{ytableau}
       6 & 7 & 8\\
       9 \\
       10
\end{ytableau}\quad , \quad\begin{ytableau}
1 & 2\\
3 & 4\\
5
\end{ytableau}
\right)\quad \text{and}\quad \left(
\begin{ytableau}
       1 & 3 & 4\\
       6 \\
       9
\end{ytableau}\quad , \quad\begin{ytableau}
5 & 2\\
7 & 8\\
10
\end{ytableau}
\right)
\end{equation}
are two $\boldsymbol{\lambda}$-tableaux for $\boldsymbol{\lambda}=((3,1,1), (2,2,1))$.

A $\boldsymbol{\lambda}$-tableau is called \textbf{standard} if its entries strictly increase along each row and column in each component. For example, in (\ref{tableaux}), the first tableau is standard, but the second is not.

Let $Std(\boldsymbol{\lambda})$ denote the set of standard $\boldsymbol{\lambda}$-tableaux, and let $Std(Par_2(n)):=\bigcup_{\boldsymbol{\lambda}\in Par_2(n)}Std(\boldsymbol{\lambda})$.

If $\mathfrak{t}$ is a $\boldsymbol{\lambda}$-tableau, then let $\mathtt{Forme}(\mathfrak{t})=\boldsymbol{\lambda}$, and define $\mathfrak{t}_{\downarrow m}$ as the subtableau containing the integers $\lbrace 1,2,\ldots ,m\rbrace$. If $\mathfrak{t}$ is standard, then $\mathtt{Forme}(\mathfrak{t}_{\downarrow m})$ is a bipartition for all $m\geq 0$.

We extend the dominance order on bipartitions to standard tableaux:
\begin{equation}
\mathfrak{s}\unrhd \mathfrak{t} \Longleftrightarrow \mathtt{Forme}(\mathfrak{s}_{\downarrow m})\unrhd \mathtt{Forme}(\mathfrak{t}_{\downarrow m})\quad \forall~1\leq m\leq n.
\end{equation}
Similarly, $\mathfrak{s}\rhd \mathfrak{t}$ if $\mathfrak{s}\unrhd \mathfrak{t}$ and $\mathfrak{s} \neq \mathfrak{t}$.

For example,
\[
\mathfrak{s}=\left(\begin{ytableau}
1 & 3 & 4\\
2 & 5\\
6
\end{ytableau} \quad , \quad  \begin{ytableau}
7 & 8\\
9
\end{ytableau}\right) \quad \text{and} \quad \mathfrak{t}=\left(\begin{ytableau}
1 & 2 \\
3 & 4
\end{ytableau}\quad , \quad \begin{ytableau}
5 & 6\\
7 & 8\\
9
\end{ytableau}\right)
\]
are not comparable because:
\[
\mathtt{Forme}(\mathfrak{t}_{\downarrow 2})=\left(\begin{ytableau}
~ & 
\end{ytableau}\quad , \quad \emptyset \right)\rhd \left(\begin{ytableau}
~\\
~
\end{ytableau}\quad , \quad \emptyset \right)=\mathtt{Forme}(\mathfrak{s}_{\downarrow 2}),
\]
but
\[
\mathtt{Forme}(\mathfrak{s}_{\downarrow 4})=\left(\begin{ytableau}
~ & ~& ~\\
~
\end{ytableau}\quad , \quad \emptyset \right)\rhd \left(\begin{ytableau}
~ & ~\\
~ & ~
\end{ytableau}\quad , \quad \emptyset \right)=\mathtt{Forme}(\mathfrak{t}_{\downarrow 4}).
\]

Given $\kappa=(\kappa_1,\kappa_2)\in \mathbb{C}^2$, for a $\boldsymbol{\lambda}$-tableau $\mathfrak{s}$, we define its $i$-th \textbf{content} as $c_i^{\kappa}(\mathfrak{s})=c_i(\mathfrak{s})=\kappa_l+c-r$, where $t(l,c,r)=i$ for $i=1,\ldots ,n$. 

The \textbf{axial distance} between $i$ and $i+1$ in $\mathfrak{s}$ is defined as $\rho_i(\mathfrak{s})=c_i(\mathfrak{s})-c_{i+1}(\mathfrak{s})$.

\paragraph{Artin-Wedderburn decomposition of $\hat{\mathcal{H}}_n^{\kappa}$.}
We know, from \cite[Theorem 3.13]{HM}, that the algebra \( \hat{\mathcal{H}}_n^{\kappa} \) is semisimple if and only if: 
\begin{equation}\label{SS H_n}
\kappa_1-\kappa_2\in \mathbb{C}\setminus \lbrace 0,\pm 1,\pm 2,\dots ,\pm (n-1)\rbrace.
\end{equation}

In this case, the set of irreducible representations of $\hat{\mathcal{H}}_n^{\kappa}$ is indexed by the bipartitions of size $n$, and we will denote $W_{\boldsymbol{\lambda}}$ the irreducible representations indexed by $\boldsymbol{\lambda}\in Par_2(n)$ so that:
\[Irr(\hat{\mathcal{H}}_n^{\kappa})=\{W_{\boldsymbol{\lambda}}\ |\ \boldsymbol{\lambda}\in Par_2(n)\}\ .\]
In this case, we have the following Artin-Wedderburn decomposition :
\begin{equation}\label{AW decomposition}
\hat{\mathcal{H}}_n^{\kappa}\cong \bigoplus_{\boldsymbol{\lambda}\in Par_2(n)}End(W_{\boldsymbol{\lambda}}).
\end{equation}
Furthemore, we have the dimension of each of these irreducible representations:
\begin{equation}\label{dim W}
\dim(W_{\boldsymbol{\lambda}})=\lvert Std(\boldsymbol{\lambda})\rvert.
\end{equation}

Let us describe how the generators act on $W_{\boldsymbol{\lambda}}$.

Fix $\boldsymbol{\lambda}\in Par_2(n)$ and let \(W_{\boldsymbol{\lambda}}\) be the vector space with a basis \(\lbrace v_{\mathfrak{t}}\rbrace_{\mathfrak{t}\in Std(\boldsymbol{\lambda})}\) indexed by standard $\boldsymbol{\lambda}$-tableaux.

The following formulas for the generators of $\hat{\mathcal{H}}_n^{\kappa}$ define an irreducible representation of the algebra on the space $W_{\boldsymbol{\lambda}}$:
\begin{align} 
x_k(v_{\mathfrak{t}})&=c_k(\mathfrak{t})v_{\mathfrak{t}}, \quad k=1,\ldots ,n \label{JM action}\\
s_r(v_{\mathfrak{t}})&=\alpha_r(\mathfrak{t})v_{s_r(\mathfrak{t})}-\frac{1}{\rho_r(\mathfrak{t})}v_{\mathfrak{t}},\quad r=1,\ldots ,n-1,
\end{align}
where $s_r(\mathfrak{t})$ is the $\boldsymbol{\lambda}$-tableau (not necessarily standard) obtained from $\mathfrak{t}$ by exchanging $r$ and $r+1$ and leaving the others fixed, and for \( \mathfrak{t} \in \mathrm{Std}(\mathrm{Par}_2(n)) \), we define:
\[
\alpha_r(\mathfrak{t}) = \left\{
\begin{array}{lll}
1 & \text{if } s_r(\mathfrak{t}) \text{ is standard and } \mathfrak{t} \triangleright s_r(\mathfrak{t}), \\
\frac{(1 + \rho_r(\mathfrak{t}))(1 + \rho_r(s_r(\mathfrak{t})))}{\rho_r(\mathfrak{t})\rho_r(s_r(\mathfrak{t}))} & \text{if } s_r(\mathfrak{t}) \text{ is standard and } s_r(\mathfrak{t}) \triangleright \mathfrak{t}, \\
0 & \text{otherwise},
\end{array}
\right.
\]
for \( 1 \leq r < n \). If $s_r(\mathfrak{t})$ is not standard, then we set $v_{s_r(\mathfrak{t})}:=0$.

As \( \boldsymbol{\lambda} \) runs over the set of bipartitions of \( n \), the representations \( W_{\boldsymbol{\lambda}} \) are pairwise non-isomorphic and exhaust the set of irreducible representations of \( \hat{\mathcal{H}}_n^{\kappa} \) (see \cite{HM}, \textsection{3}).

\begin{Rem}
The basis \(\lbrace v_{\mathfrak{t}}\rbrace_{\mathfrak{t}\in Std(\boldsymbol{\lambda})}\) of the irreducible module $W_{\boldsymbol{\lambda}}$ depends on the set $\lbrace\alpha_r(\mathfrak{t})\rbrace_{\mathfrak{t}\in Std(\boldsymbol{\lambda)})}$, called the seminormal system of coefficients associated with this basis (see \cite{HM}, \textsection{3} for more details).
\end{Rem}
The branching rules expressing the restriction from $\hat{\mathcal{H}}_n^{\kappa}$ to $\hat{\mathcal{H}}_{n-1}^{\kappa}$ are given by inclusion of bipartitions (or more precisely, of their Young diagrams), as shown in the beginning of the Bratteli graph below.
For example, for \(k \geq 2\), we have the following Bratteli diagram of the chain $\lbrace \hat{\mathcal{H}}_n^{\kappa}\rbrace_{n\geq 0}$ for the first four levels (\emph{e.g.} \cite{K} or \cite{HM}):
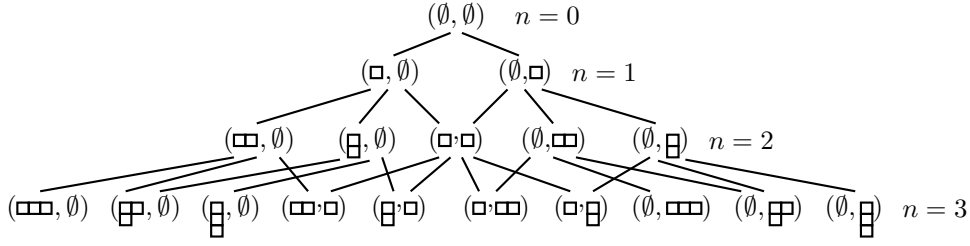
\begin{center}
\begin{tikzpicture}[scale=0.15]
% (the tikz code remains unchanged)
\node at (0,0) {$(\emptyset,\emptyset)$};
\node at (8,0) {$n=0$};

\draw[thick] (-0.5,-1.5) -- (-5.5,-3.5);
\draw[thick] (0.5,-1.5) -- (5.5,-3.5);

\node at (-8,-5) {$($};\diag{-7.5}{-4.5}{1};\node at (-4.7,-5) {$, \emptyset)$};
\node at (5,-5) {$(\emptyset ,$};\diag{6.5}{-4.5}{1};\node at (8,-5) {$)$};
\node at(13,-5) {$n=1$};

\draw[thick] (-7.5,-6.5) -- (-17.5,-9.5);
\draw[thick] (-6,-6.5) -- (-8.5,-9.5);
\draw[thick] (-4.5,-6.5) -- (-1.5,-9.5);
\draw[thick] (4.5,-6.5) -- (1.5,-9.5);
\draw[thick] (6,-6.5) -- (8.5,-9.5);
\draw[thick] (7.5,-6.5) -- (17.5,-9.5);

\node at (-20,-11) {$($};\diag{-19.5}{-10.5}{2};\node at (-15.7,-11) {$, \emptyset)$};
\node at (-10,-11) {$($};\diagg{-9.5}{-10.5}{1}{1};\node at (-6.7,-11) {$,\emptyset)$};
\node at (-2,-11) {$($};\diag{-1.5}{-10.5}{1};\node at (0,-11) {$,$};\diag{0.5}{-10.5}{1};\node at (2,-11) {$)$};
\node at (7,-11) {$(\emptyset ,$};\diag{8.5}{-10.5}{2};\node at (11,-11) {$)$};
\node at (16.7,-11) {$(\emptyset ,$};\diagg{18.5}{-10.5}{1}{1};\node at (20,-11) {$)$};
\node at (25,-11) {$n=2$};

\draw[thick] (-19.5,-12.5) -- (-36.5,-15.5);
\draw[thick] (-17.5,-12.5) -- (-29,-15.5);
\draw[thick] (-15.5,-12.5) -- (-13,-15.5);

\draw[thick] (-10.2,-12.7) -- (-26,-15.5);
\draw[thick] (-7.6,-12.8) -- (-19.5,-15.5);
\draw[thick] (-6.5,-12.5) -- (-5.5,-15.5);

\draw[thick] (-1.5,-12.5) -- (-11,-15.5);
\draw[thick] (-0.5,-12.5) -- (-4,-15.5);
\draw[thick] (0.5,-12.5) -- (2,-15.5);
\draw[thick] (1.5,-12.5) -- (10,-15.5);

\draw[thick] (6.5,-12.5) -- (3.5,-15.5);
\draw[thick] (8.5,-12.5) -- (17,-15.5);
\draw[thick] (10.5,-12.5) -- (25,-15.5);

\draw[thick] (17,-12.5) -- (12,-15.5);
\draw[thick] (17.7,-12.8) -- (27,-15.5);
\draw[thick] (20.1,-12.7) -- (35,-15.5);

\node at (-39,-17) {$($};\diag{-38.5}{-16.5}{3};\node at (-33.7,-17) {$,\emptyset)$};
\node at (-30,-17) {$($};\diagg{-29.5}{-16.5}{2}{1};\node at (-25.7,-17) {$,\emptyset)$};
\node at (-22,-17) {$($};\diaggg{-21.5}{-16.5}{1}{1}{1};\node at (-18.7,-17) {$,\emptyset)$};

\node at (-15,-17) {$($};\diag{-14.5}{-16.5}{2};\node at (-12,-17) {$,$};\diag{-11.5}{-16.5}{1};\node at (-10,-17) {$)$};
\node at (-7,-17) {$($};\diagg{-6.5}{-16.5}{1}{1};\node at (-5,-17) {$,$};\diag{-4.5}{-16.5}{1};\node at (-3,-17) {$)$};

\node at (1,-17) {$($};\diag{1.5}{-16.5}{1};\node at (3,-17) {$,$};\diag{3.5}{-16.5}{2};\node at (6,-17) {$)$};
\node at (9,-17) {$($};\diag{9.5}{-16.5}{1};\node at (11,-17) {$,$};\diagg{11.5}{-16.5}{1}{1};\node at (13,-17) {$)$};

\node at (16.7,-17) {$(\emptyset ,$};\diag{18.5}{-16.5}{3};\node at (22,-17) {$)$};
\node at (25.7,-17) {$(\emptyset ,$};\diagg{27.5}{-16.5}{2}{1};\node at (30,-17) {$)$};
\node at (33.7,-17) {$(\emptyset ,$};\diaggg{35.5}{-16.5}{1}{1}{1};\node at (37,-17) {$)$};
\node at (42,-17) {$n=3$};
\end{tikzpicture}
\end{center}
\subsection{Idempotents of the algebra \( \hat{\mathcal{H}}_n^{\kappa} \)} 
\paragraph{Complete set of central orthogonal idempotents.}

We refer to \cite[Chapter 17, \textsection{4}]{L} for more details on the idempotents of a semisimple algebra.

Let \( \mathfrak{A} \) be a unital semisimple algebra over \(\mathbb{C}\). Recall that \( e \in \mathfrak{A} \) is an \textbf{idempotent} if and only if \( e^2 = e \). If \( e_1, e_2 \in \mathfrak{A} \) are idempotents, then we say that they are \textbf{orthogonal} if and only if 
\[
e_1 e_2 = e_2 e_1 = 0.
\]
Furthermore, an idempotent \( e \in \mathfrak{A} \) is called \textbf{primitive} if \( e \) cannot be written as a sum 
\[
e = e_1 + e_2
\]
where both \( e_1, e_2 \in \mathfrak{A} \) are nonzero idempotents. An idempotent is called \textbf{trivial} if it is either \( 0 \) or \( 1 \).

Assume that we have a decomposition such that
\[
\mathfrak{A} = M_1 \oplus M_2
\]
as (left) modules and \( 1 = m_1 + m_2 \) for some \( m_1 \in M_1 \), \( m_2 \in M_2 \). Then \( m_1 \), \( m_2 \) are orthogonal idempotents in \( \mathfrak{A} \), and
\[
M_1 = \mathfrak{A}m_1, \quad M_2 = \mathfrak{A}m_2.
\]

Furthermore, \( M_i \) is irreducible (as a module) if and only if \( m_i \) is primitive. This can be easily generalized to any finite number of summands.

If \( \mathfrak{A} \) is additionally finite-dimensional, then
\[
\mathfrak{A} = P_1 \oplus \cdots \oplus P_n
\]
for some irreducible modules \( P_i \). It follows from the preceding that
\[
P_i = \mathfrak{A}e_i
\]
for some \( e_i \in \mathfrak{A} \) and \( \{e_1, \ldots, e_n\} \) is a set of pairwise orthogonal, primitive idempotents.

A finite set of idempotents is called \textbf{complete} if they sum to $1$.

The set \( \{e_1, \ldots, e_n\} \) is called a \textbf{complete set of primitive orthogonal idempotents} of \( \mathfrak{A}\).

Let $e$ be a non-zero central idempotent (\emph{i.e.} $e^2=e$ and $e$ lies in the center of $\mathfrak{A}$). We say $e$ is \textbf{centrally primitive} if we cannot write e as a sum of two non-zero orthogonal central idempotents. A central idempotent $e$ can be centrally primitive without being primitive, \emph{i.e.} $e$ can be written as a sum of two non-zero orthogonal idempotents, but neither of these is central.

A \textbf{complete set of centrally primitive orthogonal idempotents} in $\mathfrak{A}$ is a set of elements $\lbrace e_1,e_2,\dots ,e_p\rbrace$, $p\geq 0$ such that that they are centrally primitive, they sum to $1$ and they satisfy $e_{i}e_{j} = \delta_{i,j} e_{i}$.

Furthermore, a complete set of centrally primitive orthogonal idempotents in $\mathfrak{A}$ is unique in the sense that if $1 = e_1 + \ldots + e_r = f_1 + \ldots + f_s$ where each of \(e_1, \ldots, e_r\) and \(f_1, \ldots, f_s\) is a set of centrally primitive central idempotents which are orthogonal, then $r=s$ and there is a permutation $\sigma$ of $\lbrace 1,…,r\rbrace$ such that $e_i = f_{\sigma(i)}$ for all $i$.

In the particular case where $\mathfrak{A}$ is semisimple, we can describe the complete set of centrally primitive orthogonal idempotents via the Artin-Wedderburn decomposition (cf. (\ref{AW formula})).

Indeed, let $S$ be an indexing set for a complete set of pairwise non-isomorphic irreducible representations of $\mathfrak{A}$.

Let $\lambda \in S$. We define $E_{\lambda}$ as the element of $\mathfrak{A}$ corresponding under the Artin–Wedderburn decomposition of $\mathfrak{A}$ (cf. (\ref{AW formula})) to $Id_{V_{\lambda}}$ in the component corresponding to $\lambda$ and $0$ in all other components. 
The set $\lbrace E_{\lambda}\rbrace_{\lambda\in S}$ is the complete set of centrally primitive orthogonal idempotents of $\mathfrak{A}$.
In any representation $W$ of $\mathfrak{A}$, the action of $E_{\lambda}$ projects onto the isotopic component of $W$ corresponding to $\lambda$. More precisely, if the decomposition of $W$ into irreducible is

\[
W \cong \bigoplus_{\lambda' \in S} V_{\lambda '}^{\oplus m_{\lambda'}},
\]
then the action of \( E_\lambda \) is the projection onto \( V_{\lambda}^{\oplus m_\lambda} \) corresponding to this decomposition, that is:
\[
E_\lambda= \mathrm{Id}_{V_{\lambda}}^{\oplus m_\lambda} \quad \text{on} \quad V_{\lambda}^{\oplus m_{\lambda}}\quad \text{and} \quad E_\lambda \left( V_{\lambda '}^{\oplus m_{\lambda'}} \right) = 0 \quad \text{for } \lambda' \ne \lambda.
\]

\paragraph{Ideals and quotients.}

From the decomposition (\ref{AW formula}), one immediately sees that the ideals (and equivalently, the quotients) of \( \mathfrak{A} \) correspond to subsets \( S' \subset S \) in the following way:
\[
I_{S'} := \bigoplus_{\lambda \in S'} \mathrm{End}(V_\lambda), \quad \text{and} \quad \mathfrak{A}/ I_{S'} \cong \bigoplus_{\lambda \in S \setminus S'} \mathrm{End}(V_\lambda).
\]
One set of generators of the ideal \( I_{S'} \) consists of the elements \( E_\lambda \) with \( \lambda \in S' \).

\paragraph{Case of $\hat{\mathcal{H}}_n^{\kappa}$.}
Fix $n\in \mathbb{Z}_{\geq 1}$.
For the rest of the section, we assume that we are in the semisimple case for the algebra $\hat{\mathcal{H}}_n^{\kappa}$, \emph{i.e.} we consider  the condition \eqref{SS H_n} on the complex numbers $\kappa_1$ and $\kappa_2$, namely:
\[\kappa_1-\kappa_2\in \mathbb{C}\setminus \lbrace 0,\pm 1,\dots ,\pm (n-1)\rbrace.\]

We denote
\[
\mathcal{C}_n := \{\, c_i(\mathfrak{t}) \mid \mathfrak{t}\in Std(Par_2(n)),\;1\le i\le n\}
\]
the set of all contents appearing in a standard tableau of size \(n\).

For \(\mathfrak{t}\in Std(Par_2(n))\), we define the element of \(\hat{\mathcal{H}}_n^{\kappa}\):
\begin{equation}\label{idempotent}
F_{\mathfrak{t}}
=
\prod_{k=1}^n \prod_{\substack{c\in\mathcal{C}_n\\c\ne c_k(\mathfrak{t})}}
\frac{x_k - c}{c_k(\mathfrak{t}) - c}.
\end{equation}

From the action \eqref{JM action} of the Jucys-Murphy elements $x_i$, $i=1,\ldots ,n$, on a basis $\lbrace v_{\mathfrak{t}}\rbrace_{\mathfrak{t}\in Std(\boldsymbol{\lambda})}$ of the irreducible module $W_{\boldsymbol{\lambda}}$ of $\hat{\mathcal{H}}_n^{\kappa}$ for $\boldsymbol{\lambda}\in Par_2(n)$, we can view via the Artin-Wedderburn decomposition \eqref{AW decomposition} $x_i$ as a diagonal matrix with diagonal entries which are elements of $\mathcal{C}_n$.

Then, the definition \eqref{idempotent} provides via the Artin-Wedderburn decomposition, for each \(\mathfrak{t}\in Std(Par_2(n))\), an elementary matrix $E_i$, where $E_i$ is the diagonal matrix of size $\sum_{\boldsymbol{\lambda}\in Par_2(n)}\dim (W_{\boldsymbol{\lambda}})$ with the coefficient $1$ on the $i$-th line, 0 orthewise. The index $i$ corresponds to the standard tableau $\mathfrak{t}$.
Moreover, two different standard tableaux provide two different elementary diagonal matrices.

Thus, the isomorphism \eqref{AW decomposition} induces the following bijection of sets:
\begin{equation*}
\lbrace F_\mathfrak{t} \mid \mathfrak{t}\in Std(Par_2(n)) \rbrace \longleftrightarrow \left\lbrace E_i \mid i=1,2,\dots, \sum_{\boldsymbol{\lambda}\in Par_2(n)}\dim (W_{\boldsymbol{\lambda}})\right\rbrace .
\end{equation*}
The set $\lbrace F_{\mathfrak{t}} \mid \mathfrak{t}\in Std(Par_2(n))\rbrace$ is then a complete set of primitive orthogonal idempotents of $\hat{\mathcal{H}}_n^{\kappa}$.

The previous discussion allows us to deduce the following useful properties of Jucys-Murphy elements of $\hat{\mathcal{H}}_n^{\kappa}$.

We recall that for \(n\ge2\), we naturally identify \(\hat{\mathcal{H}}_{n-1}^\kappa\) with the subalgebra of \(\hat{\mathcal{H}}_n^\kappa\) generated by \(x_1,\dots,x_{n-1}\) and \(s_1,\dots,s_{n-2}\).

If \(\mathfrak{s}\in Std(Par_2(n-1))\), then in \(\hat{\mathcal{H}}_n^\kappa\):
\[
F_{\mathfrak{s}} = \sum_{\mathfrak{s}\to\mathfrak{t}} F_{\mathfrak{t}},
\]
where \(\mathfrak{s}\to\mathfrak{t}\) means \(\mathfrak{t}\) is obtained from $\mathfrak{s}$ by adding the box labeled \(n\).

Furthermore, by the Jucys–Murphy action one has for all \(\mathfrak{t}\in Std(Par_2(n))\) and \(r=2,\dots,n\):
\[
x_rF_{\mathfrak{t}} = F_{\mathfrak{t}}x_r = c_r(\mathfrak{t})F_{\mathfrak{t}}.
\]
In particular,
\[
x_n = \sum_{\boldsymbol{\lambda}\in Par_2(n)}
\sum_{\mathrm{Forme}(\mathfrak{t})=\boldsymbol{\lambda}}
c_n(\mathfrak{t})\,F_{\mathfrak{t}}=\sum_{\mathfrak{t}\in Std(Par_2(n))}c_n(\mathfrak{t})F_{\mathfrak{t}}.
\]
By convention, \(F_{(\emptyset,\emptyset)}:=1\), and the \(F_{\mathfrak{t}}\) can be built recursively: if \(\mathfrak{s}\) is obtained by removing the box containing \(n\) from \(\mathfrak{t}\), then
\begin{equation}\label{rec formula}
F_{\mathfrak{t}} = 
F_{\mathfrak{s}} \cdot \frac{(x_n - \alpha_1)\cdots(x_n - \alpha_k)}
{(c-\alpha_1)\cdots(c-\alpha_k)},
\end{equation}
where \(\alpha_1,\dots,\alpha_k\) are the contents of all addable boxes of \(\mathrm{Forme}(\mathfrak{s})\) except the one of \(\mathfrak{t}\), whose content is \(c\).

For \(\boldsymbol{\lambda}\in Par_2(n)\), let \(\mathfrak{t}^{\boldsymbol{\lambda}}\) be the standard tableau whose entries \(1,\dots,n\) fill the diagram in Western reading order.  
For example, when \(\boldsymbol{\lambda}=((3,2,1),(2,2))\),
\[
\mathfrak{t}^{\boldsymbol{\lambda}} = 
\left(
\begin{ytableau}
1 & 2 & 3\\
4 & 5\\
6
\end{ytableau}\quad ,\quad
\begin{ytableau}
7 & 8\\
9 & 10
\end{ytableau}
\right).
\]

The elements defined in (\ref{idempotent}) for the standard tableaux of the form $\mathfrak{t}^{\boldsymbol{\lambda}}$, $\boldsymbol{\lambda}\in Par_2(n)$, are particularly interesting and have been studied in \cite[Theorem 3.13]{HM}. 

Indeed, a complete set of pairwise non-isomorphic irreducible representations of $\hat{\mathcal{H}}_n^{\kappa}$ is given by the $\lbrace \hat{\mathcal{H}}_n^{\kappa}F_{\mathfrak{t}^{\boldsymbol{\lambda}}}\rbrace_{\boldsymbol{\lambda}\in Par_2(n)}$.

We deduce from (\ref{AW decomposition}) that we have the following decomposition:
\begin{equation}
\hat{\mathcal{H}}_n^{\kappa}\cong \bigoplus_{\boldsymbol{\lambda}\in Par_2(n)}End(\hat{\mathcal{H}}_n^{\kappa}F_{\mathfrak{t}^{\boldsymbol{\lambda}}}),
\end{equation}
where \(\dim(\hat{\mathcal{H}}_n^{\kappa}F_{\mathfrak{t}^{\boldsymbol{\lambda}}})=\lvert Std(\boldsymbol{\lambda})\rvert\) and the module isomorphism \(W_{\boldsymbol{\lambda}}\cong \hat{\mathcal{H}}_n^{\kappa}F_{\mathfrak{t}^{\boldsymbol{\lambda}}}\).

We define the four one-dimensional representations of $\hat{\mathcal{H}}^{(\kappa_1,\kappa_2)}_n$ associated with the following bipartitions:
\begin{equation}\label{1dimrep}\begin{array}{lllllll}
(\Box\dots\Box\,,\,\emptyset) &\quad& (\begin{array}{l}\Box\\[-0.8em] \vdots\\[-0.4em] \Box\end{array}\,,\,\emptyset) & \quad &
(\emptyset\,,\,\Box\dots\Box) &\quad& (\emptyset\,,\,\begin{array}{l}\Box\\[-0.8em] \vdots\\[-0.4em] \Box\end{array})
\\
\ \ x_1\mapsto \kappa_1 && x_1\mapsto \kappa_1 && \ \ x_1\mapsto \kappa_2 && x_1\mapsto \kappa_2 \\
\ \ s_i\mapsto 1        && s_i\mapsto -1 && \ \ s_i\mapsto 1 && s_i\mapsto -1\\
\end{array}
\end{equation}
Let $\lambda\vdash n$.
A \textbf{standard tableau} of shape $\lambda$  is a bijective filling of the boxes of $\lambda$ by numbers $1,\dots,n$ such that the entries strictly increase along any row and down 
 any column of the diagram. We denote by $d_{\lambda}$ the number of standard tableaux of shape $\lambda$. From the representation theory of the symmetric group, or from the Robinson--Schensted correspondence (see \emph{e.g.} \cite{SR}), we have:
 \begin{equation}\label{dimfactn}
 \sum_{\lambda\vdash n}d_{\lambda}^2=n!\ .
 \end{equation}

Recall that a standard tableau of shape $(\delta,\mu)$, for $(\delta,\mu)$ a bipartition of $n$, is a bijective filling of the boxes of $\delta$ and $\mu$ by the numbers $1,\dots,n$ such that the entries strictly increase along any row and down 
 any column of the two diagrams. The number of standard tableaux of shape $(\delta,\mu)$ is easily seen to be:
 \begin{equation}\label{dimbipart}
 d_{\delta,\mu}=\binom{n}{|\delta|}d_{\delta} d_{\mu}\ .
 \end{equation}

Thus, we recover from the equality \eqref{dimbipart} the four primitive central idempotents associated with the one-dimensionnal representations defined in (\ref{1dimrep}).

We denote the partition $(1,\ldots ,1)\in \mathcal{P}_n$ by $(1^n)$.

We are particularly interested in \(F_{\mathfrak{t}^{\boldsymbol{\lambda}}}\) for  
\[
\boldsymbol{\lambda}\in \{\,((n),\emptyset),\,((1^n),\emptyset),\,(\emptyset,(1^n)),\,(\emptyset,(n))\}.
\]

For these bipartitions, we fix:
\begin{equation} \label{PCI}
F_n^{(\alpha_1,\alpha_2)}:=F_{\mathfrak{t}^{\boldsymbol{\lambda}}},
\end{equation}
where $\alpha_1$ (resp. $\alpha_2$) is the image of $s_i$ for $i=1,\ldots ,n$ (resp. the image of $x_1$) by the one-dimensional representation associated with $\boldsymbol{\lambda}$ defined in (\ref{1dimrep}). In particular, we have that $\alpha_1\in \lbrace -1,1\rbrace$ and $\alpha_2\in \lbrace \kappa_1,\kappa_2\rbrace$.

We recall that the following products form a basis of \(\hat{\mathcal{H}}_n^\kappa\) (cf.\ Proposition \ref{standard}):
\begin{equation*}
\left\{\begin{array}{c} 
1,\\ x_1
\end{array}\right\} 
\cdot 
\left\{\begin{array}{c} 
1,\\ s_1, \\ s_1x_1, \\ s_1x_1s_1 \end{array}\right\}
\cdot 
\left\{\begin{array}{c} 
1,\\ s_2, \\ s_2s_1, \\ s_2s_1x_1, \\ s_2s_1x_1s_1, \\ s_2s_1x_1s_1s_2 \end{array}\right\}
\cdot\ \ldots\ \cdot 
\left\{\begin{array}{c} 1,\\ s_{n-1}, \\ \vdots \\ s_{n-1}\dots s_1x_1,\\ s_{n-1}\dots s_1x_1s_1, \\ \vdots \\ s_{n-1}\dots s_1x_1s_1\dots s_{n-1} \end{array}\right\}. 
\end{equation*} 
We index these sets by \(0,\dots,n-1\). For example, \(s_2s_1x_1\) belongs to the second set, and \(s_{n-1}\) to the \((n-1)\)‑th set.

We define \(\ell(\omega)\) as the \textbf{length} of \(\omega\), and \(\ell_0(\omega)\) as the number of occurrences of the generator \(x_1\) in a reduced expression.

Given an element $\omega \in B_n$, we define $m_i(\omega) = 1$ if $x_1$ appears in the $i$-th set when $\omega$ is written in its standard form, and $m_i(\omega)=0$ otherwise.

In what follows, it will be convenient to consider the indices of $\kappa_1$, $\kappa_2$ modulo $2$, so that $\kappa_{b+1} = \kappa_1$ when $b = 2$.

We recall that we are in the conditions such that the algebra $\hat{\mathcal{H}}^{\kappa}_n$ is semisimple, namely:
\begin{equation*}
\kappa_1-\kappa_2\in \mathbb{C}\setminus \lbrace 0,\pm 1,\pm 2,\dots ,\pm (n-1)\rbrace.
\end{equation*}
\begin{Pro} \label{idempotent 1}
For $b=1,2$ and $\alpha=\pm 1$, we have, in \(\hat{\mathcal{H}}_n^\kappa\):
\begin{equation} \label{idempotent formula}
F_n^{(\alpha,\kappa_b)}=D_n(\alpha,\kappa_b)
\sum_{\omega\in B_n}
\alpha^{\ell(\omega)}\left(\prod_{i=0}^{n-1}(i-\alpha \kappa_{b+1})^{1-m_i(\omega)}\right)s_\omega,
\end{equation}
where:
\begin{equation*}
	D_n(\alpha,\kappa_b)=\left(n!\prod_{i=0}^{n-1}(i+\alpha(\kappa_b-\kappa_{b+1}))\right)^{-1}.
\end{equation*}
\end{Pro}
\begin{proof}
We prove the result for $F_n^{(1,\kappa_1)}$ first.
 
The proof proceeds by induction on \(n\ge1\). For \(n=1\), we have
\[
F_1^{(1,\kappa_1)} = \frac{x_1-\kappa_2}{\kappa_1-\kappa_2}.
\]
Indeed, we have $\mathcal{C}_1=\lbrace \kappa_1,\kappa_2\rbrace$ and then we deduce the result from the definition (\ref{idempotent}).

This coincides with \eqref{idempotent formula} for $n=1$.

Assume the statement holds for \(n\). By the recurrence relation \eqref{rec formula}, we have:
\[
F_{n+1}^{(1,\kappa_1)}
=
F_n^{(1,\kappa_1)}
\frac{(x_{n+1}-(\kappa_1-1))(x_{n+1}-\kappa_2)}{(\kappa_1+n-(\kappa_1-1))(\kappa_1+n-\kappa_2)}=F_n^{(1,\kappa_1)}
\frac{(x_{n+1}-\kappa_1+1)(x_{n+1}-\kappa_2)}{(n+1)(n+\kappa_1-\kappa_2)}.
\]

We have:
\begin{equation} \label{rel lemme}
(x_{n+1}-\kappa_1+1)(x_{n+1}-\kappa_2)=x_{n+1}^2+(1-\kappa_1-\kappa_2)x_{n+1}+\kappa_2(\kappa_1-1).
\end{equation}
\begin{Lem} \label{lem idempotent}
For $i, j = 1, \dots, n$, with $i \neq j$, we have:
\begin{equation*}
[n,-n][n,j]s_{j+1}\dots s_n = s_j\dots s_{n-1}[n,-(j-1)],
\end{equation*}
\begin{equation*}
[n,j]s_{j+1}\dots s_n[n,-n] = s_{j-1}\dots s_1x_1s_1\dots s_{n-1}n,j],
\end{equation*}
\begin{equation*}
(s_n\dots s_i\dots s_n)(s_n\dots s_j\dots s_n) = \left\{
\begin{array}{cl}
s_j\ldots s_{n-1}s_i\dots s_{j-2}[n,i], & \text{if } i<j, \\
s_{i-1}\ldots s_j\dots s_{i-1}s_i\ldots s_{n-1}[n,i],  &\text{if } i>j.
\end{array}
\right.
\end{equation*}
\end{Lem}
\begin{proof}
The proof is straightforward using the following equalities:
\begin{equation*}
[n,j]s_k=s_{k-1}[n,j]\quad \text{for}\quad k\in \lbrace j+1,\ldots ,n\rbrace,
\end{equation*}
\begin{equation*}
s_{\ell}^2=1\quad \text{and}\quad [n,-n]s_{\ell}=s_{\ell}[n,-n]\quad \text{for}\quad \ell\in \lbrace 1,\ldots ,n-1\rbrace.
\end{equation*}
\end{proof}
Recalling that $x_{n+1} = [n,-n] + \sum_{j=1}^n [n,j]s_{j+1} \ldots s_n$, we deduce the following equality in $\hat{\mathcal{H}}_{n+1}^{\kappa}$:
\begin{align*}
x_{n+1}^2 &= (\kappa_1+\kappa_2)[n,-n] -\kappa_1\kappa_2 + \sum_{j=1}^n [n,-n][n,j]s_{j+1} \ldots s_n \\ &+ \sum_{j=1}^n [n,j]s_{j+1} \ldots s_n [n,-n] 
+ \left( \sum_{j=1}^n [n,j]s_{j+1} \ldots s_n \right)^2
\\ &= (\kappa_1+\kappa_2)[n,-n] -\kappa_1\kappa_2 + \sum_{j=1}^n [n,-n][n,j]s_{j+1} \ldots s_n \\ &+ \sum_{j=1}^n [n,j]s_{j+1} \ldots s_n [n,-n]
+ n + \sum_{i\neq j}[n,i]s_{i+1}\dots s_n[n,j]s_{j+1}\dots s_n.
\end{align*}
We recall that we have the following property of $F_n^{(1,\kappa_1)}$:
\begin{equation}\label{help 3}
F_n^{(1,\kappa_1)}\,\omega
=
\begin{cases}
F_n^{(1,\kappa_1)} & \text{if}\quad \omega=s_i,\; i=1,\dots,n-1,\\
\kappa_1F_n^{(1,\kappa_1)} & \text{if}\quad\omega = x_1.
\end{cases}
\end{equation}
Using Lemma \ref{lem idempotent} and equality \eqref{help 3}, we have:
\begin{align*}
	F_n^{(1,\kappa_1)}\sum_{j=1}^n[n,-n][n,j]s_{j+1}\dots s_n&=F_n^{(1,\kappa_1)}\sum_{j=1}^n[n,-(j-1)]=F_n^{(1,\kappa_1)}\sum_{j=0}^{n-1}[n,-j],\\
	F_n^{(1,\kappa_1)}\sum_{j=1}^n[n,j]s_{j+1}\dots s_n[n,-n]&=F_n^{(1,\kappa_1)}\sum_{j=1}^ns_{j-1}\ldots s_1x_1s_1\ldots s_{n-1}[n,j]\\&=\kappa_1F_n^{(1,\kappa_1)}\sum_{j=1}^n[n,j],\\
	F_n^{(1,\kappa_1)}\sum_{i\neq j}^n[n,i]s_{i+1}\dots s_n[n,j]s_{j+1}\dots s_n&=F_n^{(1,\kappa_1)}\sum_{i\neq j}^n(s_n\dots s_i\dots s_n)(s_n\dots s_j\dots s_n)\\ 
	&=(n-1)F_n^{(1,\kappa_1)}\sum_{j=1}^n[n,j].
\end{align*}
Then we have:
\begin{multline}\label{help 1}
	F_n^{(1,\kappa_1)}x_{n+1}^2=F_n^{(1,\kappa_1)}\left[(\kappa_1+\kappa_2)[n,-n]-\kappa_1\kappa_2+\sum_{j=0}^{n-1}[n,-j]+n\right.\\
	\left. +(n-1+\kappa_1)\sum_{j=1}^n[n,j]\right].
\end{multline}
Now, using braid relations, we have, for $j=1,\dots ,n-1$:
\begin{equation*}
	s_n\dots s_{j+1}s_js_{j+1}\dots s_n=s_js_{j+1}\dots s_{n-1}[n,j].
\end{equation*}
Then we deduce the following equality, using the equality \eqref{help 3}:
\begin{equation*}
	F_n^{(1,\kappa_1)}s_n\dots s_{j+1}s_js_{j+1}\dots s_n=F_n^{(1,\kappa_1)}s_js_{j+1}\dots s_{n-1}[n,j]=F_n^{(1,\kappa_1)}[n,j],
\end{equation*}
and then, by definition of $x_{n+1}$:
\begin{equation} \label{help 2}
	F_n^{(1,\kappa_1)}x_{n+1}=F_n^{(1,\kappa_1)}\left([n,-n]+\sum_{j=1}^n[n,j]\right).
\end{equation}
Finally, from \eqref{rel lemme}, combining \eqref{help 1} and \eqref{help 2}, we deduce the following equality:
\begin{align*}
&F_n^{(1,\kappa_1)}
(x_{n+1}^2+(1-\kappa_1-\kappa_2)x_{n+1}+\kappa_2(\kappa_1-1))\\
=&
F_n^{(1,\kappa_1)}\left( \sum_{j=0}^n[n,-j]+(n-\kappa_2)\sum_{j=1}^{n+1}[n,j]\right),
\end{align*}
where by convention $[n,j]:=1$ if $j=n+1$.

By the induction hypothesis, we obtain:
\begin{align*}
F_{n+1}^{(1,\kappa_1)}&=D_n(1,\kappa_1)
\sum_{\omega\in B_n}
\left(\prod_{i=0}^{n-1}(i-\kappa_2)^{1-m_i(\omega)}\right)s_\omega \frac{1}{(n+1)(\kappa_1-\kappa_2+n)}\\ 
&\left( \sum_{j=0}^n[n,-j]+(n-\kappa_2)\sum_{j=1}^{n+1}[n,j]\right)\\
&=D_{n+1}(1,\kappa_1)\left(\sum_{\omega\in B_n}
\left(\prod_{i=0}^{n-1}(i-\kappa_2)^{1-m_i(\omega)}\right)s_\omega\right) \left( \sum_{j=0}^n[n,-j]\right.\\
&\left.+(n-\kappa_2)\sum_{j=1}^{n+1}[n,j]\right)\\
&=D_{n+1}(1,\kappa_1)
\sum_{\nu\in B_{n+1}}
\left(\prod_{i=0}^{n}(i-\kappa_2)^{1-m_i(\omega)}\right)s_\nu.
\end{align*}
This completes the proof for $F_n^{(1,\kappa_1)}$.

For the idempotent $F_n^{(-1,\kappa_1)}$, we consider the following map:
\[\begin{array}{cccc}
x_1 & \longmapsto & -x_1 &\\
s_i & \longmapsto & -s_i, & i=1,2,\dots ,n-1.
\end{array}\]

A direct check of the relations \eqref{eq: DR1}-\eqref{eq: DR4} shows that this map extends to an algebra morphism between $\hat{\mathcal{H}}_n^{-\kappa}$ and $\hat{\mathcal{H}}_n^{\kappa}$.

Moreover, the idempotent $F_n^{(1,-\kappa_1)}$ of $\hat{\mathcal{H}}_n^{-\kappa}$ is sent through this map to the idempotent $F_n^{(-1,\kappa_1)}$ of $\hat{\mathcal{H}}_n^\kappa$ since they are characterized, respectively, by:
\begin{align*}
	&s_iF_n^{(1,-\kappa_1)}=F_n^{(1,-\kappa_1)}s_i=F_n^{(1,-\kappa_1)},\\
    &x_1F_n^{(1,-\kappa_1)}=F_n^{(1,-\kappa_1)}x_1=-\kappa_1F_n^{(1,-\kappa_1)},\\
	&(F_n^{(1,-\kappa_1)})^2=F_n^{(1,-\kappa_1)},
\end{align*}
and 
\begin{align*}
	&s_iF_n^{(-1,\kappa_1)}=F_n^{(-1,\kappa_1)}s_i=-F_n^{(\alpha_1,\kappa_1)},\\
    &x_1F_n^{(-1,\kappa_1)}=F_n^{(-1,\kappa_1)}x_1=\kappa_1 F_n^{(-1,\kappa_1)}, \\
	&(F_n^{(-1,\kappa_1)})^2=F_n^{(-1,\kappa_1)},
\end{align*}
for $i=1,2,\dots ,n-1$.

Therefore, it shows that the idempotent $F_n^{(-1,\kappa_1)}$ of $\hat{\mathcal{H}}_n^\kappa$ is obtained from $F_n^{(1,\kappa_1)}$ by replacing each generator $x_1,s_1,s_2,\dots ,s_{n-1}$ by minus itself, and also replacing $\kappa_1,\kappa_2$ by $-\kappa_1,-\kappa_2$. This proves Formula \eqref{idempotent formula} for $\alpha=-1$.

Finally, it is obvious that for idempotents $F_n^{(\alpha,\kappa_2)}$ the formulas are obtained from the ones for $F_n^{(\alpha,\kappa_1)}$ by exchanging $\kappa_1$ and $\kappa_2$.
\end{proof}
\begin{Rem}
We deduce from the last proposition the following recurrence relation:
\begin{multline}\label{Rec idempotent}
F_{n+1}^{(\alpha,\kappa_b)}=\frac{1}{(n+1)(\alpha(\kappa_b-\kappa_{b+1})+n)}F_n^{(\alpha,\kappa_b)}\left( \sum_{j=0}^n \alpha^{n+j+1}[n,-j]\right.\\ \left. +(n-\alpha\kappa_{b+1})\sum_{j=1}^{n+1}\alpha^{n-j+1}[n,j]\right).
\end{multline}
\end{Rem}
\begin{Exe} \label{Exe F_2}
Let $n=2$. Recalling the (standard) basis of $\hat{\mathcal{H}}_2^{\kappa}$: 
\begin{equation*}
\lbrace 1,s_1,s_1x_1,s_1x_1s_1,x_1,x_1s_1,x_1s_1x_1,(x_1s_1)^2\rbrace,
\end{equation*} we have:
\begin{multline*}
F_2^{(1,\kappa_1)}=\frac{1}{2(\kappa_1-\kappa_2)(1+\kappa_1-\kappa_2)}(\kappa_2(1-\kappa_2)(1+s_1)-\kappa_2(s_1x_1+s_1x_1s_1)\\ +(1-\kappa_2)(x_1+x_1s_1)+x_1s_1x_1+(x_1s_1)^2),
\end{multline*}
\begin{multline*}
F_2^{(-1,\kappa_1)}=\frac{1}{2(\kappa_2-\kappa_1)(1+\kappa_2-\kappa_1)}(\kappa_2(1+\kappa_2)(1-s_1)+\kappa_2(s_1x_1-s_1x_1s_1)\\+(1+\kappa_2)(x_1s_1-x_1)+(x_1s_1)^2-x_1s_1x_1).
\end{multline*}
We obtain the idempotent $F_2^{(1,\kappa_2)}$ (respectively $F_2^{(-1,\kappa_2)}$) from $F_2^{(1,\kappa_1)}$ (respectively $F_2^{(-1,\kappa_1)}$) by replacing $\kappa_1$ with $\kappa_2$.
\end{Exe}
\paragraph{Particular case $(\kappa_1,\kappa_2)=(0,k+1)$.}
Let $k\geq 1$.

We are now interested in the particular case where we specialize the parameters $(\kappa_1,\kappa_2)$ to:
\begin{equation}\label{specialisation}
\kappa_1=0, \quad \kappa_2=k+1.
\end{equation}

The motivation for this specialization comes from the link that we will establish later between $\hat{\mathcal{H}}_n^{(0,k+1)}$ and the fused permutations algebra.

From now on, we assume for the remainder of this section that \( n \leq k+1 \) so that we are in the semisimple case (see \eqref{SS H_n}).

\begin{Cor}
In \(\hat{\mathcal{H}}_n^{(0,k+1)}\), we have:
\begin{equation}
F_n^{(1,0)}
=
n!^{-1}
\sum_{\omega\in B_n}
\left(\prod_{i=0}^{n-1}
(i-k-1)^{-m_i(\omega)}\right)s_\omega,
\end{equation}
\begin{equation}
F_n^{(1,k+1)}
= \left(n\prod_{i=0}^{n-1}(k+1+i)\right)^{-1}\sum_{\omega\in B_n}\left(\prod_{i=0}^{n-1}i^{-m_i(\omega)}\right)s_\omega,
\end{equation}
\begin{equation}
F_n^{(-1,0)}
=
n!^{-1}
\sum_{\omega\in B_n}(-1)^{\ell(\omega)}
\left(\prod_{i=0}^{n-1}
(i+k+1)^{-m_i(\omega)}\right)s_\omega,
\end{equation}
\begin{equation}
F_n^{(-1,k+1)}
= \left(n\prod_{i=0}^{n-1}(i-k-1)\right)^{-1}\sum_{\omega\in B_n}(-1)^{\ell(\omega)}\left(\prod_{i=0}^{n-1}i^{-m_i(\omega)}\right)s_\omega.
\end{equation}
\end{Cor}

\begin{proof}
The proof is straightforward using Proposition \ref{idempotent 1} with the specialization \eqref{specialisation} of the parameters.
\end{proof}

\begin{Rem}
From the expressions above, we obtain the following recurrence relations:
\begin{equation}\label{Rec1}
F_{n+1}^{(1,0)}
=
\frac{1}{n+1}
F_n^{(1,0)}
\left[
\frac{1}{(n-(k+1))}\sum_{j=0}^n[n,-j]
+
\sum_{j=1}^{n+1}[n,j]
\right],
\end{equation}
\begin{equation}
F_{n+1}^{(1,k+1)} =
\frac{1}{(n+1)(k+1+n)}
F_n^{(1,k+1)}
\left[
\sum_{j=0}^n[n,-j]
+ n\sum_{j=1}^{n+1}[n,j]
\right],
\end{equation}
\begin{equation}
F_{n+1}^{(-1,0)}
=
\frac{F_n^{(-1,0)}}{n+1}
\left[
\frac{1}{(n+k+1)}\sum_{j=0}^n(-1)^{n+j+1}[n,-j]
+
\sum_{j=1}^{n+1}(-1)^{n-j+1}[n,j]
\right],
\end{equation}
\begin{multline}
F_{n+1}^{(-1,k+1)} =
\frac{F_n^{(-1,k+1)}}{(n+1)(n-k-1)}
\left[
\sum_{j=0}^n(-1)^{n+j+1}[n,-j]\right.\\
\left.+ n\sum_{j=1}^{n+1}(-1)^{n-j+1}[n,j]
\right],
\end{multline}
where $F_n^{(\pm 1,\kappa_b)}$, $b=1,2$, are considered as elements of $\hat{\mathcal{H}}_{n+1}^{\kappa}$.
\end{Rem}
\begin{Exe}
The central idempotents $F_2^{(\pm 1,\kappa_b)}$, with $b=1,2$, for $n=2$ are:
\begin{multline} \label{F_20}
F_2^{(1,0)}=\frac{1}{2}\left(1+s_1-\frac{1}{k}\left(s_1x_1+s_1x_1s_1\right)-\frac{1}{k+1}\left(x_1+x_1s_1\right)\right.\\
\left.+\frac{1}{k(k+1)}\left(x_1s_1x_1+(x_1s_1)^2\right)\right),
\end{multline}
\begin{multline}
F_2^{(-1,0)}=\frac{1}{2}\left(\left(1-s_1\right)+\frac{1}{(k+1)}\left(x_1s_1-x_1\right)+\frac{1}{(k+2)}\left(s_1x_1-s_1x_1s_1\right)\right.+\\
\left. \frac{1}{(k+1)(k+2)}\left((x_1s_1)^2-x_1s_1x_1\right)\right),
\end{multline}
\begin{equation}
F_2^{(1,k+1)}=\frac{1}{2(k+1)(k+2)}\left(x_1+x_1s_1+x_1s_1x_1+(x_1s_1)^2\right),
\end{equation}
\begin{equation}
F_2^{(-1,k+1)}=\frac{1}{2k(k+1)}\left(x_1s_1+(x_1s_1)^2-x_1-x_1s_1x_1\right).
\end{equation}
\end{Exe}
\section{Connection with the fused permutations algebra}
\subsection{$H_{k,n}$ as a quotient of $\hat{\mathcal{H}}_n^{\kappa}$}
The goal of this section is to establish a connection between the algebra of fused permutations for a given $k\geq 1$ and the degenerate affine cyclotomic Hecke algebra.

We recall the definitions of $H_{k,n}$ (see Section \ref{s2}) and $\hat{\mathcal{H}}_n^{\kappa}$ (see Definition \ref{AHDC}).
\subsubsection{Via the representation theory}\label{sssec411}
We fix $n \in \mathbb{Z}_{\geq 0}$ and $k \geq 1$ and we assume that $n\leq k$. In particular, the algebra $\hat{\mathcal{H}}_n^{(0,k+1)}$ is semisimple (see \eqref{SS H_n}).

Recall that given a subset $\Gamma$ of the parametrising set $Par_2(n)$ of irreducible representations of $\hat{\mathcal{H}}_n^{\kappa}$, we have an associated subalgebra. Namely, after applying the Wedderburn decomposition and with the notation of \eqref{AW decomposition}, the subalgebra is $\bigoplus_{\gamma\in \Gamma}End(W_\gamma)$. 
\begin{Pro}\label{quo}(\textbf{$H_{k,n}$ as a quotient of $\hat{\mathcal{H}}_n^{(0,k+1)}$.})

Let $I_{k,n}$ be the ideal of $\hat{\mathcal{H}}_n^{(0,k+1)}$ corresponding to the following subset of partitions:
\begin{equation*}
L_{k,n}:=\lbrace \boldsymbol{\lambda}=(\lambda^{(1)},\lambda^{(2)})\in Par_2(n)\mid \ell(\lambda^{(2)})\geq 2\rbrace.
\end{equation*}
Then the quotient $\hat{\mathcal{H}}_n^{(0,k+1)}/I_{k,n}$ is isomorphic to $H_{k,n}$.
\end{Pro}
\begin{proof}
The quotient $\hat{\mathcal{H}}_n^{(0,k+1)}/I_{k,n}$ is isomorphic to the subalgebra of $\hat{\mathcal{H}}_n^{(0,k+1)}$ corresponding to the subset of $Par_2(n)$ complementary to $L_{k,n}$. This subset is:
\[
D_{k,n}:=\lbrace \boldsymbol{\lambda}=(\lambda^{(1)},\lambda^{(2)})\in Par_2(n)\mid \ell(\lambda^{(2)})<2\rbrace.
\]
Thus, it suffices to give a bijection between the sets $Irr(H_{k,n})$ and $D_{k,n}$ respecting the dimensions of the corresponding irreducible representations.

Let $\lambda=(\lambda_1,\dots ,\lambda_p)\in Irr(H_{k,n})$, $p\in \mathbb{Z}_{\geq 0}$. Thus, we have $\lambda=(k+a,\lambda_2,\dots ,\lambda_p)\vdash k+n$ with $a\in \mathbb{Z}_{\geq 0}$.

We set:
\begin{equation*}
\begin{array}{cccc}
\phi_n: & Irr(H_{k,n}) & \longrightarrow & D_{k,n}\\
& \lambda=(k+a,\lambda_2,\dots ,\lambda_p) & \longmapsto & \phi_n(\lambda):=((\lambda_2,\ldots ,\lambda_p),(a))
\end{array}.
\end{equation*}
First, $\phi_n(\lambda)$ takes values in $D_{k,n}$. Indeed $\phi_n(\lambda)$ is of the correct size since $\lvert \lambda\rvert=k+n$ and then $\lvert \phi_n(\lambda)\rvert=\lvert (\lambda_2,\ldots ,\lambda_p)\rvert+a=n$. 

A similar argument shows that the following map:
\begin{equation*}
\begin{array}{cccc}
\psi_n:& D_{k,n}&\longrightarrow & Irr(H_{k,n})\\
& ((\lambda_2,\ldots ,\lambda_p),(a)) &\longmapsto & (k+a,\lambda_2,\ldots ,\lambda_p)
\end{array}
\end{equation*}
provides the inverse map (recall that $n\leq k$ so that we have $\lambda_2\leq k+a$).

Now let $\lambda\in Irr(H_{k,n})$ and set $\lambda'=\phi_n(\lambda)$. It remains to show that the cardinality of $SSTab(\lambda,k)$ is equal to the cardinality of $Std(\lambda')$. 

We consider the map induced by $\phi_n$ (still denoted by $\phi_n$) defined by:
\begin{equation*}
\begin{array}{cccc}
\phi_n: & SSTab(\lambda,k) & \longrightarrow & Std(\lambda')\\
& \mathbf{T} & \longmapsto & \phi_n(\mathbf{T})
\end{array},
\end{equation*}
where $\phi_n(\mathbf{T})$ is the tableau obtained from $\mathbf{T}$ this way: first, decrease all entries of $\mathbf{T}$ by $1$. Second, keep all lines of $\mathbf{T}$ except the first one: this creates a tableau $t_1$ of shape $(\lambda_2,\dots ,\lambda_p)$. Third, keep the boxes of the first line after the first $k$ boxes: this creates a second tableau $t_2$ of shape $(a)$. Finally, we set $\phi_n(\mathbf{T}):=(t_1,t_2)$.

This map is well defined thanks to the conditions on $\mathbf{T}$ (see the example below).

We also consider the map deduced from $\psi_n$ (still denoted $\psi_n$):
\begin{equation*}
\begin{array}{cccc}
\psi_n:& Std(\lambda')&\longrightarrow & SSTab(\lambda,k)\\
& \mathfrak{t} &\longmapsto & \psi_n(\mathfrak{t})
\end{array}
\end{equation*}
where $\psi_n(\mathfrak{t})$ is defined this way: first, increase all entries of $\mathfrak{t}$ by $1$. Second, we consider the tableau $\psi_n(\mathfrak{t})$ of shape $(k+a,\lambda_2,\dots ,\lambda_p)$ where the entries of the first $k$ boxes of the first line is $1$ and the other ones are those of the corresponding boxes of $\mathfrak{t}$.
It remains to verify that the latter is well defined and then provides the inverse of $\phi_n$.

Considering $n\leq k$, in particular we have $n+k\leq 2k$. Then, if $\lambda\in Irr(H_{k,n})$ is of the form $(k+a,\lambda_2,\ldots ,\lambda_p)$, the previous condition imposes that $\lambda_i\leq k$, $i=2,\ldots ,p$. Then, the condition to be semistandard for $\psi_n(\mathbf{t})$, $\mathbf{t}\in Std(Par_2(n))$ is equivalent for $\mathbf{t}$ to be standard. Thus $\psi_n(\mathbf{t})$ takes values in $SSTab(\lambda,k)$.

Hence $\psi_n$ is well defined. Finally $\phi_n$ is a bijection and then:
\begin{equation*}
\lvert SSTab(\lambda,k)\rvert=\lvert Std(\lambda')\rvert.
\end{equation*}
Hence the result follows.
\end{proof}
\begin{Exe}
Let $k=3$ and $n=5$. Let $\mathbf{T}\in SSTab((5,2,1),3)$ the following semistandard tableau:
\[\mathbf{T}:=\begin{ytableau}
1 & 1 & 1 & 2 & 3\\
4 & 5\\
6
\end{ytableau}\quad.\]
Then we have, from the construction of the last proof, the following tableaux $t_1$ and $t_2$:
\[t_1=\begin{ytableau}
3 & 4\\
5
\end{ytableau}\quad \text{,}\quad t_2=\begin{ytableau}
1 & 2
\end{ytableau}\quad .\]
Then \[\phi_5(\mathbf{T})=\left(~\begin{ytableau}
3 & 4\\
5
\end{ytableau}\quad, \quad \begin{ytableau}
1 & 2
\end{ytableau}~\right).\]
\end{Exe}
\begin{Rem}
From the last proof, keeping the same notations, we have shown that, given $\lambda\vdash k+n$, $\lvert SSTab(\lambda,k)\rvert=\lvert Std(\lambda')\rvert$ if $n\leq k$.
\end{Rem}

For $k=2,3$, the Bratteli diagram for the chain of algebras $\lbrace \hat{\mathcal{H}}_n^{(0,k+1)}\rbrace_{n\geq 0}$ begins as:

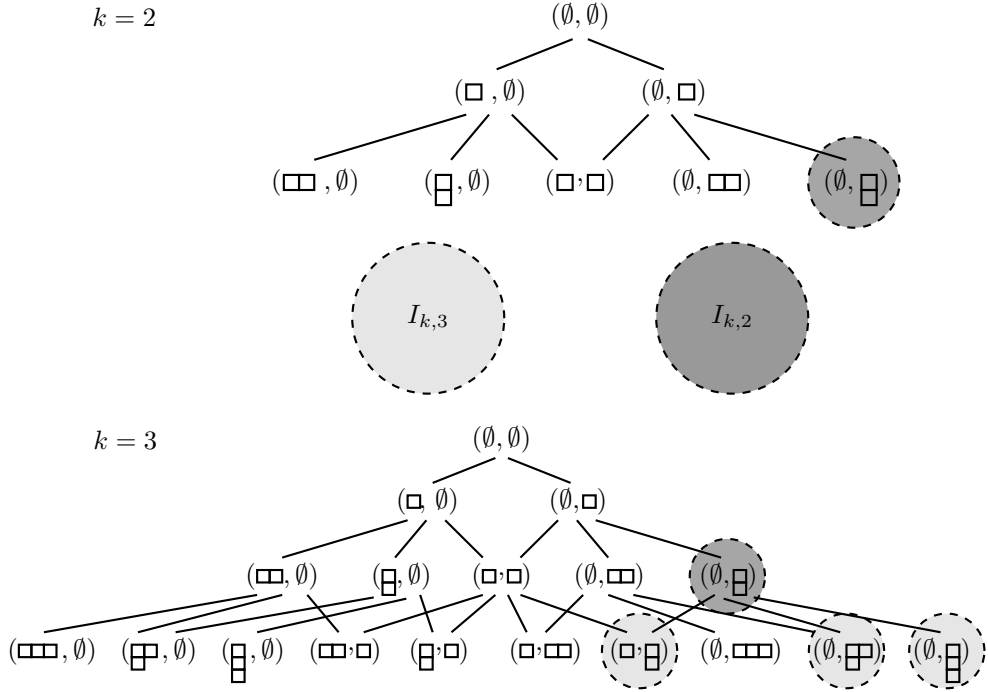
\begin{center}
\begin{tikzpicture}[scale=0.2]
\node at (-30,0) {$k=2$};
\draw [thick, dashed, fill=gray!70] (18,-11) circle (3);

\node at (0,0) {$(\emptyset,\emptyset)$};

\draw[thick] (-0.5,-1.5) -- (-5.5,-3.5);
\draw[thick] (0.5,-1.5) -- (5.5,-3.5);

\node at (-8,-5) {$($};\diag{-7.5}{-4.5}{1};\node at (-4.7,-5) {$,\emptyset)$};
\node at (5,-5) {$(\emptyset ,$};\diag{6.5}{-4.5}{1};\node at (8,-5) {$)$};

\draw[thick] (-7.5,-6.5) -- (-17.5,-9.5);
\draw[thick] (-6,-6.5) -- (-8.5,-9.5);
\draw[thick] (-4.5,-6.5) -- (-1.5,-9.5);
\draw[thick] (4.5,-6.5) -- (1.5,-9.5);
\draw[thick] (6,-6.5) -- (8.5,-9.5);
\draw[thick] (7.5,-6.5) -- (17.5,-9.5);

\node at (-20,-11) {$($};\diag{-19.5}{-10.5}{2};\node at (-15.7,-11) {$, \emptyset)$};
\node at (-10,-11) {$($};\diagg{-9.5}{-10.5}{1}{1};\node at (-7,-11) {$, \emptyset)$};
\node at (-2,-11) {$($};\diag{-1.5}{-10.5}{1};\node at (0,-11) {$,$};\diag{0.5}{-10.5}{1};\node at (2,-11) {$)$};
\node at (7,-11) {$(\emptyset ,$};\diag{8.5}{-10.5}{2};\node at (11,-11) {$)$};
\node at (17,-11) {$(\emptyset ,$};\diagg{18.5}{-10.5}{1}{1};\node at (20,-11) {$)$};

\draw [thick, dashed, fill=gray!20] (-10,-20) circle(5);
\draw [thick, dashed, fill=gray!80] (10,-20) circle(5);
\node at (-10,-20) {$I_{k,3}$};
\node at (10,-20) {$I_{k,2}$};
\end{tikzpicture}
\end{center}

\begin{center}
\begin{tikzpicture}[scale=0.165]
\node at (-30,0) {$k=3$};
\draw [thick, dashed, fill=gray!70] (18,-11) circle (3);
\draw [thick, dashed, fill=gray!20] (11,-17) circle(3);
\draw [thick, dashed, fill=gray!20] (27.5,-17) circle(3);
\draw [thick, dashed, fill=gray!20] (35.5,-17) circle(3);
\node at (0,0) {$(\emptyset,\emptyset)$};

\draw[thick] (-0.5,-1.5) -- (-5.5,-3.5);
\draw[thick] (0.5,-1.5) -- (5.5,-3.5);

\node at (-8,-5) {$($};\diag{-7.5}{-4.5}{1};\node at (-5,-5) {$,\,\emptyset)$};
\node at (5,-5) {$(\emptyset ,$};\diag{6.5}{-4.5}{1};\node at (8,-5) {$)$};

\draw[thick] (-7.5,-6.5) -- (-17.5,-9.5);
\draw[thick] (-6,-6.5) -- (-8.5,-9.5);
\draw[thick] (-4.5,-6.5) -- (-1.5,-9.5);
\draw[thick] (4.5,-6.5) -- (1.5,-9.5);
\draw[thick] (6,-6.5) -- (8.5,-9.5);
\draw[thick] (7.5,-6.5) -- (17.5,-9.5);

\node at (-20,-11) {$($};\diag{-19.5}{-10.5}{2};\node at (-16,-11) {$,\emptyset)$};
\node at (-10,-11) {$($};\diagg{-9.5}{-10.5}{1}{1};\node at (-7,-11) {$,\emptyset)$};
\node at (-2,-11) {$($};\diag{-1.5}{-10.5}{1};\node at (0,-11) {$,$};\diag{0.5}{-10.5}{1};\node at (2,-11) {$)$};
\node at (7,-11) {$(\emptyset ,$};\diag{8.5}{-10.5}{2};\node at (11,-11) {$)$};
\node at (17,-11) {$(\emptyset ,$};\diagg{18.5}{-10.5}{1}{1};\node at (20,-11) {$)$};

\draw[thick] (-19.5,-12.5) -- (-36.5,-15.5);
\draw[thick] (-17.5,-12.5) -- (-29,-15.5);
\draw[thick] (-15.5,-12.5) -- (-13,-15.5);

\draw[thick] (-10.2,-12.7) -- (-26,-15.5);
\draw[thick] (-7.6,-12.8) -- (-19.5,-15.5);
\draw[thick] (-6.5,-12.5) -- (-5.5,-15.5);

\draw[thick] (-1.5,-12.5) -- (-11,-15.5);
\draw[thick] (-0.5,-12.5) -- (-4,-15.5);
\draw[thick] (0.5,-12.5) -- (2,-15.5);
\draw[thick] (1.5,-12.5) -- (10,-15.5);

\draw[thick] (6.5,-12.5) -- (3.5,-15.5);
\draw[thick] (8.5,-12.5) -- (17,-15.5);
\draw[thick] (10.5,-12.5) -- (25,-15.5);

\draw[thick] (17,-12.5) -- (12,-15.5);
\draw[thick] (17.7,-12.8) -- (27,-15.5);
\draw[thick] (20.1,-12.7) -- (35,-15.5);

\node at (-39,-17) {$($};\diag{-38.5}{-16.5}{3};\node at (-33.7,-17) {$, \emptyset)$};
\node at (-30,-17) {$($};\diagg{-29.5}{-16.5}{2}{1};\node at (-25.7,-17) {$, \emptyset)$};
\node at (-22,-17) {$($};\diaggg{-21.5}{-16.5}{1}{1}{1};\node at (-18.7,-17) {$, \emptyset)$};

\node at (-15,-17) {$($};\diag{-14.5}{-16.5}{2};\node at (-12,-17) {$,$};\diag{-11.5}{-16.5}{1};\node at (-10,-17) {$)$};
\node at (-7,-17) {$($};\diagg{-6.5}{-16.5}{1}{1};\node at (-5,-17) {$,$};\diag{-4.5}{-16.5}{1};\node at (-3,-17) {$)$};

\node at (1,-17) {$($};\diag{1.5}{-16.5}{1};\node at (3,-17) {$,$};\diag{3.5}{-16.5}{2};\node at (6,-17) {$)$};
\node at (9,-17) {$($};\diag{9.5}{-16.5}{1};\node at (11,-17) {$,$};\diagg{11.5}{-16.5}{1}{1};\node at (13,-17) {$)$};

\node at (17,-17) {$(\emptyset ,$};\diag{18.5}{-16.5}{3};\node at (22,-17) {$)$};
\node at (26,-17) {$(\emptyset ,$};\diagg{27.5}{-16.5}{2}{1};\node at (30,-17) {$)$};
\node at (34,-17) {$(\emptyset ,$};\diaggg{35.5}{-16.5}{1}{1}{1};\node at (37,-17) {$)$};
\end{tikzpicture}
\end{center}

Thus, the representation theory of the algebras $\hat{\mathcal{H}}_n^{(0,k+1)}$ and $H_{k,n}$ and Corollary \ref{quo} justify and motivate the reason why we consider the algebra $\hat{\mathcal{H}}_n^{(0,k+1)}$ in order to provide a canonical basis and a presentation by generators and relations of $H_{k,n}$.

However, for the moment, from Corollary \ref{quo}, we only know that $H_{k,n}$ is isomorphic to a quotient of $\hat{\mathcal{H}}_n^{(0,k+1)}$ if $n\leq k$.
\subsubsection{Via an algebraic construction}
The aim now is to generalize this quotient for arbitrary integers $n\in \mathbb{Z}_{\geq 0}$ and $k\geq 1$.

For this purpose, we are going to construct a surjective algebra morphism from $\hat{\mathcal{H}}_n^{(0,k+1)}$ to $H_{k,n}$.

Fix now $n\in \mathbb{Z}_{\geq 0}$ and $k\geq 1$ without any other restriction.
\begin{Def}
In $H_{k,n}$, we define:
\begin{align}
 	&\mathbf{1} :=
 	\begin{tikzpicture}[scale=0.3,baseline={([yshift=\eseq]current bounding box.center)}]		
		\ellk{0}{0}
		\ellstrand{0}{0}
		\slab{(4,3)}{$1$}
		\strand{4}{2}
		\node at (6,0) {$\dots$};
		\slab{(8,3)}{$n$}
		\strand{8}{2}
	\end{tikzpicture},\\
	&\sigma_i := 
	\begin{tikzpicture}[scale=0.3,baseline={([yshift=\eseq]current bounding box.center)}]		
		\ellk{0}{0}
		\ellstrand{0}{0}
		\slab{(4,3)}{$1$}
		\strand{4}{2}
		\node at (6,0) {$\dots$};
		\slab{(8,3)}{$i-1$}
		\strand{8}{2}
		\slab{(12,3)}{$i$}
		\slab{(16,3)}{$i+1$}
		\ocross{12}{2}
		\slab{(20,3)}{$i+2$}
		\strand{20}{2}
		\node at (22,0) {$\dots$};
		\slab{(24,3)}{$n$}
		\strand{24}{2}
	\end{tikzpicture},~~i=1,\ldots ,n-1, \label{eq:defSi}\\	
	&\sigma_0 :=
	\begin{tikzpicture}[scale=0.3,baseline={([yshift=\eseq]current bounding box.center)}]
		\ellp{0}{0}
		\ellocross{0}{0}
		\slab{(4,3)}{$1$}
		\slab{(8,3)}{$2$}
		\strand{8}{2}
		\node at (10,0) {$\dots$};
		\slab{(12,3)}{$n$}
		\strand{12}{2}
	\end{tikzpicture}\ . \label{eq:defS0} 
\end{align}
Finally, we set: \[t := \mathbf{1} + k\sigma_0.\]

We have the following algebraic correspondance (in $\mathcal{H}_{k,n}$) for $i=1,2,\ldots ,n-1$:
\begin{equation}
\mathbf{1}=\mathbf{H}(P_k),
\end{equation}
\begin{equation}
\sigma_i=\mathbf{H}(P_k\delta_{k+i}P_k)=\mathbf{H}(P_k\delta_{k+i})=\mathbf{H}(\delta_{k+i}P_k),
\end{equation}
\begin{equation}
\sigma_0=\mathbf{H}(P_k\delta_kP_k),
\end{equation}
where $\mathbf{H}$ is the algebra isomorphism defined in (\ref{def H}).

We have used that $\delta_{k+i}$ commutes with $P_k$ when $i\geq 1$.
\end{Def}
\paragraph{A surjective map from $\hat{\mathcal{H}}_n^{\kappa}$ to $H_{k,n}$.}
We set the following elements of $H_{k,n}$:
\begin{align*}
	&h_{i} :=
	\begin{tikzpicture}[scale=0.3,baseline={([yshift=\eseq]current bounding box.center)}]
		\ellU{0}{0}
		\node at (6,0) {$\dots$};
		\slab{(16,3)}{$i+1$}
		\strand{16}{2}
		\node at (18,0) {$\dots$};
		\slab{(20,3)}{$n$}
		\strand{20}{2}
		\slab{(4,3)}{$1$}
		\slab{(8,3)}{$i-1$}
		\slab{(12,3)}{$i$}		
	\end{tikzpicture}, \quad i=1,2,\dots,\min (k,n).
\end{align*}
\begin{Lem} \label{(s_1t)^2}
The following equalities hold in $H_{k,n}$:
\begin{equation}\label{eq:R1}
(t\sigma_1)^2 = t(\sigma_1 t - \sigma_1 + \mathbf{1}), 
\end{equation}
\begin{equation} \label{eq:R2}
(\sigma_1 t)^2 = (t \sigma_1 - \sigma_1 + \mathbf{1}) t. 
\end{equation}
\end{Lem}
\begin{proof}
First, if $k=1$, it is easy to check that $\sigma_0\sigma_1\sigma_0=\sigma_1\sigma_0\sigma_1$.
Moreover, we have that $\sigma_i^2=\mathbf{1}$, $i=0,1$.

From the previous equalities, we have, on the one hand:
\begin{equation*}
(t\sigma_1)^2=((\mathbf{1}+\sigma_0)\sigma_1)^2=(\sigma_1+\sigma_0\sigma_1)^2=\mathbf{1}+\sigma_1\sigma_0\sigma_1+\sigma_0+\sigma_1\sigma_0.
\end{equation*}
On the other hand, we obtain:
\begin{equation*}
t(\sigma_1t-\sigma_1+\mathbf{1})=(\mathbf{1}+\sigma_0)(\sigma_1\sigma_0+\mathbf{1})=\sigma_1\sigma_0+\mathbf{1}+\sigma_1\sigma_0\sigma_1+\sigma_0.
\end{equation*}
Then we obtain the equality \eqref{eq:R1}.

A similar calculation gives the second equality (\ref{eq:R2}).

Now assume that $k\geq 2$.

For (\ref{eq:R1}), we first compute:
\begin{align*}
\sigma_0\sigma_1\sigma_0 &=
\begin{tikzpicture}[scale=0.25,baseline={([yshift=\eseq]current bounding box.center)}]
	\ellp{0}{0}
	\ellocross{0}{0}
	\slab{(4,3)}{$1$}
	\slab{(8,3)}{$2$}
	\strand{8}{2}
	\node at (10,0) {$\dots$};
	\slab{(12,3)}{$n$}
	\strand{12}{2}
	\ellp{0}{0}			
	\begin{scope}[yshift=-4cm]
		\ellk{0}{0}
		\ellstrand{0}{0}
		\ocross{4}{2}
		\node at (10,0) {$\dots$};
		\strand{12}{2}
	\end{scope}			
	\begin{scope}[yshift=-8cm]
		\ellp{0}{0}
		\ellocross{0}{0}
		\strand{8}{2}
		\node at (10,0) {$\dots$};
		\strand{12}{2}
		\ellp{0}{0}
	\end{scope}			
\end{tikzpicture}
=
\begin{tikzpicture}[scale=0.25,baseline={([yshift=\eseq]current bounding box.center)}]	
	\draw[lightgray,fill=lightgray] (-1,-2) rectangle (1,2);
	\fill (0,2) ellipse (1.4cm and 0.2cm);
	\fill (0,-2) ellipse (1.4cm and 0.2cm);
	\draw[thick] (-1,2) -- (-1,-2);
	\slab{(0,0)}{$\scriptstyle{k-1}$}	
	\draw[thick] (4,2)..controls +(0,-2) and +(0,+2) .. (1,-2);
	\draw[thick] (8,2)..controls +(0,-2) and +(0,+2) .. (4,-2);
	\draw[thick] (1,2)..controls +(0,-3) and +(0,+3) .. (8,-2);
	\fill (4,2) circle (0.2);
	\fill (4,-2) circle (0.2);
	\slab{(4,3)}{$1$}
	\slab{(8,3)}{$2$}	
	\fill (4,2) circle (0.2);
	\fill (4,-2) circle (0.2);
	\fill (8,2) circle (0.2);
	\fill (8,-2) circle (0.2);
	\node at (10,0) {$\dots$};
	\slab{(12,3)}{n}
	\strand{12}{2}
	\begin{scope}[yshift=-4cm]
		\ellp{0}{0}
		\ellocross{0}{0}
		\strand{8}{2}
		\node at (10,0) {$\dots$};
		\strand{12}{2}
		\ellp{0}{0}
	\end{scope}		
\end{tikzpicture}\\
&= \frac{1}{k}
\begin{tikzpicture}[scale=0.25,baseline={([yshift=\eseq]current bounding box.center)}]	
	\draw[lightgray,fill=lightgray] (-1,-2) rectangle (1,2);
	\fill (0,2) ellipse (1.4cm and 0.2cm);
	\fill (0,-2) ellipse (1.4cm and 0.2cm);
	\draw[thick] (-1,2) -- (-1,-2);
	\slab{(0,0)}{$\scriptstyle{k-1}$}	
	\strand{4}{2}
	\draw[thick] (8,2)..controls +(0,-2) and +(0,+2) .. (1,-2);
	\draw[thick] (1,2)..controls +(0,-2) and +(0,+2) .. (8,-2);
	\fill (4,2) circle (0.2);
	\fill (4,-2) circle (0.2);
	\slab{(4,3)}{$1$}
	\slab{(8,3)}{$2$}	
	\fill (4,2) circle (0.2);
	\fill (4,-2) circle (0.2);
	\fill (8,2) circle (0.2);
	\fill (8,-2) circle (0.2);
	\node at (10,0) {$\dots$};
	\slab{(12,3)}{n}
	\strand{12}{2}
\end{tikzpicture}
+ \frac{k-1}{k}
	\begin{tikzpicture}[scale=0.25,baseline={([yshift=\eseq]current bounding box.center)}]
	\draw[lightgray,fill=lightgray] (-1,-2) rectangle (1,2);
	\fill (0,2) ellipse (1.4cm and 0.2cm);
	\fill (0,-2) ellipse (1.4cm and 0.2cm);
	\draw[thick] (-1,2) -- (-1,-2);
	\slab{(0,0)}{$\scriptstyle{k-2}$}	
	\draw[thick] (4,2)..controls +(0,-2) and +(0,+2) .. (1,-2);
	\draw[thick] (8,2)..controls +(0,-2) and +(0,+2) .. (1.2,-2);
	\draw[thick] (1.2,2)..controls +(0,-3) and +(0,3) .. (8,-2);	
	\draw[thick] (1,2)..controls +(0,-3) and +(0,+3) .. (4,-2);	
	\fill (4,2) circle (0.2);
	\fill (4,-2) circle (0.2);
	\slab{(4,3)}{$1$}
	\slab{(8,3)}{$2$}	
	\fill (4,2) circle (0.2);
	\fill (4,-2) circle (0.2);
	\fill (8,2) circle (0.2);
	\fill (8,-2) circle (0.2);
	\slab{(12,3)}{3}
	\slab{(16,3)}{n}
	\node at (14,0) {$\dots$};
	\strand{16}{2}	
	\strand{12}{2}	
	\end{tikzpicture}
\\
&= \frac{1}{k} \sigma_1 \sigma_0 \sigma_1 + \frac{k-1}{k} h_2.
\end{align*}

In particular,
\begin{equation}
h_2 = \frac{k}{k-1} \sigma_0 \sigma_1 \sigma_0 - \frac{1}{k-1} \sigma_1 \sigma_0 \sigma_1. \label{eq:h2}
\end{equation}

Since $\sigma_1 h_2 = h_2 \sigma_1 = h_2$, we deduce:
\begin{align*}
(\sigma_0 \sigma_1)^2 &= \frac{1}{k} \sigma_1 \sigma_0 + \frac{k-1}{k} h_2 \\
&= \frac{1}{k} \sigma_1 \sigma_0 + \sigma_0 \sigma_1 \sigma_0 - \frac{1}{k} \sigma_1 \sigma_0 \sigma_1 \quad \text{(from \eqref{eq:h2})}.
\end{align*}

Hence:
\begin{align*}
(t \sigma_1)^2 &= ((\mathbf{1} + k \sigma_0) \sigma_1)^2 \\
&= \mathbf{1} + k \sigma_1 \sigma_0 \sigma_1 + k \sigma_0 + k^2 (\sigma_0 \sigma_1)^2 \\
&= \mathbf{1} + k \sigma_0 + k \sigma_1 \sigma_0 + k^2 \sigma_0 \sigma_1 \sigma_0 \\
&= \mathbf{1}  + (t - \mathbf{1}) + \sigma_1 (t - \mathbf{1}) + (t - \mathbf{1}) \sigma_1 (t - \mathbf{1})\\
&= t(\sigma_1 t - \sigma_1 + \mathbf{1}).
\end{align*}

A similar calculation gives the second equality (\ref{eq:R2}).
\end{proof}
\begin{Pro} \label{quotient}
There exists a surjective (unital) algebra morphism $\varphi$ from $\hat{\mathcal{H}}^{(0,k+1)}_n$ to $H_{k,n}$ defined on the generators by:
\[ \begin{array}{cccc}
\varphi : & \hat{\mathcal{H}}^{(0,k+1)}_n &\longrightarrow & H_{k,n} \\
    & s_i &\longmapsto & \sigma_i,\\
    & x_1 &\longmapsto & t,
\end{array}\]
where $i=1,\ldots ,n-1$.

In particular, the algebra $H_{k,n}$ is isomorphic to a quotient of $\hat{\mathcal{H}}_n^{(0,k+1)}$.
\end{Pro}
\begin{proof}
To show that the defined map is indeed an surjective algebra morphism, we will use the following two lemmas:
\begin{Lem}\label{WD relations}
The defining relations (\ref{eq: DR1})-(\ref{eq: DR4}) of the algebra $\hat{\mathcal{H}}_n^{\kappa}$ are satisfied for the image of the generators under $\varphi$.

\end{Lem}
\begin{proof}
For (\ref{eq: DR1}), we have:
\begin{align*}
\sigma_0^2=
\begin{tikzpicture}[scale=0.25,baseline={([yshift=\eseq]current bounding box.center)}]
		\ellp{0}{0}
		\ellocross{0}{0}
		\slab{(4,3)}{$1$}
		\slab{(8,3)}{$2$}
		\strand{8}{2}
		\node at (10,0) {$\dots$};
		\slab{(12,3)}{$n$}
		\strand{12}{2}
		\ellp{0}{0}
		\begin{scope}[yshift=-4cm]
			\ellp{0}{0}
			\ellocross{0}{0}
			\node at (6,2) {$\dots$};
			\strand{8}{2}
			\node at (10,0) {$\dots$};
			\strand{12}{2}
			\end{scope}			
\end{tikzpicture} &= \frac{1}{k} \begin{tikzpicture}[scale=0.25,baseline={([yshift=\eseq]current bounding box.center)}]		
		\ellk{0}{0}
		\ellstrand{0}{0}
		\slab{(4,3)}{$1$}
		\strand{4}{2}
		\node at (6,0) {$\dots$};
		\slab{(8,3)}{$n$}
		\strand{8}{2}
	\end{tikzpicture} + \frac{k-1}{k}
\begin{tikzpicture}[scale=0.3,baseline={([yshift=\eseq]current bounding box.center)}]
		\ellp{0}{0}
		\ellocross{0}{0}
		\slab{(4,3)}{$1$}
		\slab{(8,3)}{$2$}
		\strand{8}{2}
		\node at (10,0) {$\dots$};
		\slab{(12,3)}{$n$}
		\strand{12}{2}
	\end{tikzpicture}\\
	&=\frac{1}{k} \mathbf{1}+\frac{k-1}{k} \sigma_0.
\end{align*}  

Therefore, $t^2=(\mathbf{1}+k\sigma_0)^2=\mathbf{1}+2k\sigma_0+k\mathbf{1}+k(k-1)\sigma_0=(k+1)t.$

For (\ref{eq: DR2}), given $\ell=2,\ldots, n-1$, we have:
\begin{align*}
\sigma_0\sigma_{\ell}=
\begin{tikzpicture}[scale=0.25,baseline={([yshift=\eseq]current bounding box.center)}]
		\ellp{0}{0}
		\ellocross{0}{0}
		\slab{(4,3)}{$1$}
		\node at (6,0) {$\dots$};
		\slab{(8,3)}{$\ell$}
		\strand{8}{2}
		\slab{(12,3)}{$\ell+1$}
		\strand{12}{2}
		\node at (14,0) {$\dots$};
		\slab{(16,3)}{$n$}
		\strand{16}{2}
		\ellp{0}{0}
		\begin{scope}[yshift=-4cm]
			\ellk{0}{0}
		\ellstrand{0}{0}
		\strand{4}{2}
		\ocross{8}{2}
		\node at (6,0) {$\dots$};	
		\node at (14,0) {$\dots$};
		\strand{16}{2}
			\end{scope}			
\end{tikzpicture} &= \begin{tikzpicture}[scale=0.25,baseline={([yshift=\eseq]current bounding box.center)}]		
		\ellp{0}{0}
		\ellocross{0}{0}
		\slab{(4,3)}{$1$}
		\node at (6,0) {$\dots$};
		\slab{(8,3)}{$\ell$}
		\ocross{8}{2}
		\slab{(12,3)}{$\ell+1$}
		\node at (14,0) {$\dots$};
		\slab{(16,3)}{$n$}
		\strand{16}{2}
		\ellp{0}{0}
	\end{tikzpicture} \\
	&=\begin{tikzpicture}[scale=0.25,baseline={([yshift=\eseq]current bounding box.center)}]
	\ellk{0}{0}
		\ellstrand{0}{0}
		\strand{4}{2}
		\ocross{8}{2}
		\node at (6,0) {$\dots$};	
		\node at (14,0) {$\dots$};
		\strand{16}{2}
		\begin{scope}[yshift=-4cm]
			\ellk{0}{0}
		\ellocross{0}{0}
		\slab{(4,7)}{$1$}
		\node at (6,0) {$\dots$};
		\slab{(8,7)}{$\ell$}
		\strand{8}{2}
		\slab{(12,7)}{$\ell+1$}
		\strand{12}{2}
		\node at (14,0) {$\dots$};
		\slab{(16,7)}{$n$}
		\strand{16}{2}
		\ellp{0}{0}
			\end{scope}			
\end{tikzpicture}=\sigma_{\ell}\sigma_0.
\end{align*} 
Therefore, $\sigma_{\ell}t=t\sigma_{\ell}$ by definition of $t$.

Equality (\ref{eq: DR3}) follows immediatly from lemma \ref{(s_1t)^2}. Indeed, by subtracting the equations (\ref{eq:R1})-(\ref{eq:R2}), the result follows directly.

Finally, (\ref{eq: DR4}) follows from the fact that the subalgebra of $H_{k,n}$ generated by $\sigma_1,\ldots,\sigma_{n-1}$ is isomorphic to $\mathbb{C}\mathfrak{S}_n$.
\end{proof}
For $1 \leq p \leq q \leq n-1$, we denote $\sigma_{q,p} := \sigma_q\sigma_{q-1}\ldots \sigma_p$ and $\sigma^{p,q} := \sigma_p\sigma_{p+1}\ldots \sigma_q$.
\begin{Lem} \label{h generate}
The elements:
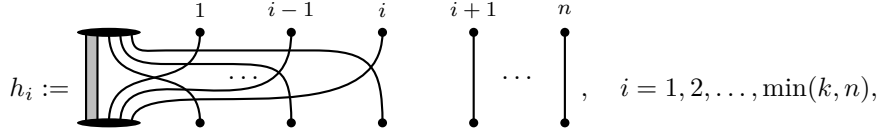
\begin{align*}
	&h_{i} :=
	\begin{tikzpicture}[scale=0.3,baseline={([yshift=\eseq]current bounding box.center)}]
		\ellU{0}{0}
		\node at (6,0) {$\dots$};
		\slab{(16,3)}{$i+1$}
		\strand{16}{2}
		\node at (18,0) {$\dots$};
		\slab{(20,3)}{$n$}
		\strand{20}{2}
		\slab{(4,3)}{$1$}
		\slab{(8,3)}{$i-1$}
		\slab{(12,3)}{$i$}		
	\end{tikzpicture}, \quad i=1,2,\dots,\min (k,n),
\end{align*}
together with the elements $(\sigma_i)_{i=1,\ldots ,n-1}$ generate $H_{k,n}$.
\end{Lem}

\begin{proof}
It suffices to show that any diagram is generated by these elements, since diagrams generate $H_{k,n}$.

For a given diagram, let $i=1,2,\ldots ,\min (k,n)$ be the number of edges connecting the top ellipse (respectively the bottom) to points at the bottom (respectively top).

Then, left (respectively right) multiplication of $h_i$ by the $(\sigma_i)_{i=1,\ldots ,n-1}$ permutes the endpoints at the top (respectively bottom) of these $i$ edges among the $n$ points, hence producing any diagram of this form.

Finally, any diagram having $k$ parallel edges connecting the top and bottom ellipses is generated by the $(\sigma_i)_{i=1,\ldots ,n-1}$, which gives the result.
\end{proof}

In view of Lemma \ref{WD relations}, the map $\varphi$ is indeed an algebra morphism. Let us now show that the map is surjective. To do this, we need to show that the generators of $H_{k,n}$ lie in the image of $\varphi$. Note that for $i=0$, $h_0=1$ and for $i=1$, $h_1=\sigma_0$. By Lemma \ref{h generate}, it remains to show that the elements $(h_i)_{i=1,\ldots,\min (k,n)}$ belong to the image of $\varphi$. We prove this by induction on $i \geq 1$.

For $i=1$, $h_1=\varphi\left(\frac{1}{k}\left(t-1\right)\right)$.

Assume that for some $p\in \{1,\ldots,\min (k,n)-1\}$, the elements $h_i$ belong to the image of $\varphi$ for $i=1,\ldots,p$. We now show that $h_{p+1}$ also lies in the image. We have:
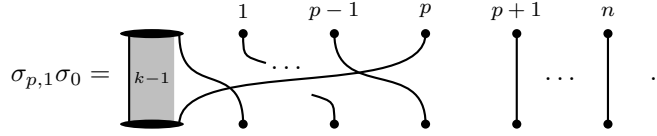
\begin{center}
\begin{tikzpicture}[scale=0.3]
\node at (-4,0) {$\sigma_{p,1}\sigma_0=$};
\ellp{0}{0}
\draw[thick] (12,2)..controls +(0,-3) and +(0,3) .. (1.2,-2);
		\draw[thick] (8,2)..controls +(0,-3) and +(0,3) .. (12,-2);	
		\strand{20}{2}
		\node at (18,0) {$\dots$};
		\fill (4,2) ellipse (0.2cm and 0.2cm);\fill (4,-2) ellipse (0.2cm and 0.2cm);
		\fill (8,2) ellipse (0.2cm and 0.2cm);\fill (8,-2) ellipse (0.2cm and 0.2cm);
		\fill (12,2) ellipse (0.2cm and 0.2cm);\fill (12,-2) ellipse (0.2cm and 0.2cm);
		\draw[thick] (4,2)..controls +(0,-1)  .. (5,0.7);
		\draw[thick] (8,-2)..controls +(0,1)  .. (7,-0.7);	
		\node at (6,0.5) {$\dots$};
		\draw[thick] (1.2,2)..controls +(0,-3) and +(0,3) .. (4,-2);	
		\strand{16}{2}
				\slab{(16,3)}{$p+1$}
		\slab{(20,3)}{$n$}
		\slab{(4,3)}{$1$}
		\slab{(8,3)}{$p-1$}
		\slab{(12,3)}{$p$}	
		\node at (22,0) {$.$};
\end{tikzpicture}
\end{center}

Then we have:
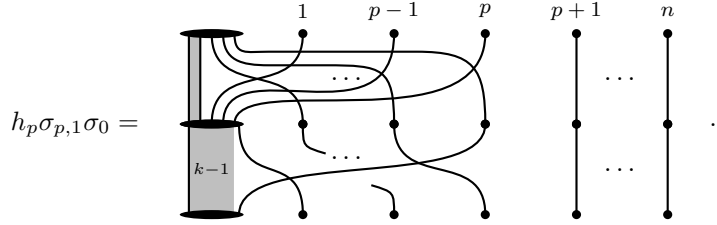
\begin{center}
\begin{tikzpicture}[scale=0.3,baseline={([yshift=\eseq]current bounding box.center)}]
	\node at (-6,-2){$h_p\sigma_{p,1}\sigma_0=$};
		\ellU{0}{0}
		\node at (6,0) {$\dots$};
		\slab{(16,3)}{$p+1$}
		\strand{16}{2}
		\node at (18,0) {$\dots$};
		\slab{(20,3)}{$n$}
		\strand{20}{2}
		\slab{(4,3)}{$1$}
		\slab{(8,3)}{$p-1$}
		\slab{(12,3)}{$p$}			
	\begin{scope}[yshift=-4cm]
			\ellp{0}{0}
\draw[thick] (12,2)..controls +(0,-3) and +(0,3) .. (1.2,-2);
		\draw[thick] (8,2)..controls +(0,-3) and +(0,3) .. (12,-2);	
		\strand{20}{2}
		\node at (18,0) {$\dots$};
		\fill (4,2) ellipse (0.2cm and 0.2cm);\fill (4,-2) ellipse (0.2cm and 0.2cm);
		\fill (8,2) ellipse (0.2cm and 0.2cm);\fill (8,-2) ellipse (0.2cm and 0.2cm);
		\fill (12,2) ellipse (0.2cm and 0.2cm);\fill (12,-2) ellipse (0.2cm and 0.2cm);
		\draw[thick] (4,2)..controls +(0,-1)  .. (5,0.7);
		\draw[thick] (8,-2)..controls +(0,1)  .. (7,-0.7);	
		\node at (6,0.5) {$\dots$};
		\draw[thick] (1.2,2)..controls +(0,-3) and +(0,3) .. (4,-2);	
		\strand{16}{2}
			\end{scope}
		\node at (22,-2) {$.$};	
\end{tikzpicture}
\end{center}
The multiplication of the two diagrams right above provides a linear combination of $k!$ diagrams.

Among them, the element $h_{p+1}$ appears when the only edge in the bottom diagram in the multiplication right above which join the top ellipse to the first bottom dot is linked to one of the $k-p$ parallel vertical edges joining the two different ellipses of the top diagram. This amounts to have an element $\pi\in \mathfrak{S}_k$ such that $\pi^{-1}(k)\in\lbrace 1,\dots ,k-p\rbrace$. The cardinality of this subset of elements of $\mathfrak{S}_k$ is:
\[\lvert \lbrace \pi\in \mathfrak{S}_k\mid \pi^{-1}(k)\in \lbrace 1,\dots ,k-p\rbrace\rbrace\rvert=(k-p)(k-1)!.\]

Thus, the element $h_{p+1}$ appears with non-zero coefficient, namely:
\[
\frac{(k-p)(k-1)!}{k!}=\frac{k-p}{k}.\]

The others ones are the $(\sigma^{j,p}h_p\sigma_{p,1})_{j=1,\ldots,p}$, each appearing with the same non-zero coefficient, namely $k^{-1}$.

By the induction hypothesis, the sum of these latter elements admits a preimage under $\varphi$; let us denote it by $\omega$ $\left(\emph{i.e.}~ \varphi(\omega)=\sum_{j=1}^p\sigma^{j,p}h_p\sigma_{p,1}\right)$. Then we have:
\[
h_{p+1}=\frac{1}{k-p}\varphi\left(\left(h_p\sigma_{p,1}(t-1)-\omega\right)\right),
\]
and therefore $h_{p+1}\in \mathrm{Im}(\varphi)$, an antecedent being given by $\frac{1}{k-p}\left(h_p\sigma_{p,1}(t-1)-\omega\right)\in \hat{\mathcal{H}}_n^\kappa$.

This concludes the proof.
\end{proof}
\begin{Rem}
From the proof of Lemma \ref{WD relations}, we deduce the characteristic equation satisfied by $\sigma_0$ in $H_{k,n}$:
\begin{equation} \label{charac sigma}
\left(\sigma_0-1\right)\left(\sigma_0+\frac{1}{k}\right)=0,
\end{equation}
which algebraically corresponds to (in $\mathcal{H}_{k,n}$):
\begin{equation}
\left(P_k\delta_kP_k-P_k\right)\left(P_k\delta_kP_k+\frac{1}{k}P_k\right)=0.
\end{equation}
It is easy to see that, when $2\leq n\leq k$, there are three one-dimensional representations for $H_{k,n}$. They are given by the following partitions:
\begin{equation} \label{1rep FP}
\lambda=(k+n) :\begin{array}{l} \sigma_0\mapsto 1\,,\\
\sigma_i\mapsto 1\,, \end{array}~~
\lambda=(k,1^n) :\begin{array}{l} \sigma_0\mapsto -\frac{1}{k}\,,\\
\sigma_i\mapsto -1\,,\end{array} ~~
\lambda=(k,n) :\begin{array}{l} \sigma_0\mapsto -\frac{1}{k}\,,\\
\sigma_i\mapsto 1\,,\end{array}
\end{equation}
where we give the associated values of the elements $\sigma_0$ and $\sigma_i$ ($i\geq 1$). These are easily obtained from the description in \cite{CP}.

In particular, the eigenvalues of $\sigma_0$ which appear in the characteristic equation \eqref{charac sigma}  explain its values given in the one-dimensional representations of $H_{k,n}$ defined in \eqref{1rep FP}.

When $n>k$, there are only two remaining one-dimensional representations, the one corresponding to $\lambda=(k,n)$ being removed from the list (it would not make sense for $n>k$).
\end{Rem}

By Proposition \ref{quotient}, we have an algebra isomorphism:
\[
\hat{\mathcal{H}}_n^{\kappa} / Ker(\varphi) \cong H_{k,n}.
\]
The purpose of what follows is to understand the kernel of $\varphi$, to give a precise description of it, and to deduce:
\begin{enumerate}
\item[•] A canonical basis of $H_{k,n}$.
\item[•] A presentation of $H_{k,n}$ by generators and relations.
\end{enumerate}
\subsection{Case where $n\leq k$}

We recall the expressions of the four central primitive idempotents of $\hat{\mathcal{H}}_n^{\kappa}$ in the standard basis $\lbrace s_\omega\rbrace_{\omega\in B_n}$, see (\ref{idempotent formula}).

\begin{Def}
We define the algebra $\mathcal{A}_n^{\kappa}$ as the quotient of $\hat{\mathcal{H}}_n^{\kappa}$ by the relation:
\begin{equation} \label{(x_1s_1)^2 general}
	(x_1s_1)^2=\kappa_1(1+\kappa_1)(s_1-1)+\kappa_1(s_1x_1s_1-s_1x_1)+(1+\kappa_1)(x_1-x_1s_1)+x_1s_1x_1.
\end{equation}

It is understood that $\mathcal{A}_n^{\kappa}=\hat{\mathcal{H}}_n^{\kappa}$ if $n=0,1$.
\end{Def}
\begin{Rem}
If $\kappa_2-\kappa_1\neq 0,1$, using the explicit expression of $F_2^{(-1,\kappa_2)}$ (see Example \ref{Exe F_2}), this is equivalent to imposing the following relation in $\hat{\mathcal{H}}_n^\kappa$:
\begin{equation}
F_2^{(-1,\kappa_2)}=0,
\end{equation}
where $F_2^{(-1,\kappa_2)}$ is seen as an element of $\hat{\mathcal{H}}_n^{\kappa}$.
\end{Rem}

\paragraph{Artin-Wedderburn decomposition for $\mathcal{A}_n^{\kappa}$.}

We assume here that the condition \eqref{SS H_n} is satisfied.

The element $F_m^{(\alpha_1,\alpha_2)}$ in $\hat{\mathcal{H}}_m^{\kappa}$ (see \eqref{PCI}) is the minimal central idempotent corresponding to the one-dimensional representation associated to $(\alpha_1,\alpha_2)$ (see \eqref{1dimrep}). It means that it is non-zero in this one-dimensional representation of $\hat{\mathcal{H}}_n^{\kappa}$ and acts as $0$ in any other irreducible representation. Now if $n\geq m$, it follows that $F_m^{(\alpha_1,\alpha_2)}$ seen as an element of $\hat{\mathcal{H}}_n^{\kappa}$ is non-zero in an irreducible representation if and only if this irreducible representation contains in its restriction to $\hat{\mathcal{H}}_m^{\kappa}$ the given one-dimensional representation. Therefore the quotient by $F_m^{(\alpha_1,\alpha_2)}=0$ removes exactly these irreducible representations.

In the particular case of $\mathcal{A}_n^{\kappa}$, which is the quotient of $\hat{\mathcal{H}}_n^{\kappa}$ by the relation $F_2^{(-1,\kappa_2)}=0$, we recall the indexing of one-dimensional representations set up in (\ref{1dimrep}), and we find that the disappearing representations are those $V_{(\lambda,\mu)}$ with $\mu$ having at least two non-empty rows. 
Then, the algebra $\mathcal{A}_n^{\kappa}$ is semisimple with the following set of irreducible representations:
\[Irr(\mathcal{A}_n^{\kappa})=\{V_{(\lambda,\mu)}\ |\ (\lambda,\mu)\in Par_2(n)\ \text{and}\ \ell(\mu)<2\}\ .\]
The Bratteli diagram for the algebras $\mathcal{A}_n^{\kappa}$ is obtained from the Bratteli diagram for the algebras $\hat{\mathcal{H}}_n^{\kappa}$, where all bipartitions with more than one row in the second component are removed:
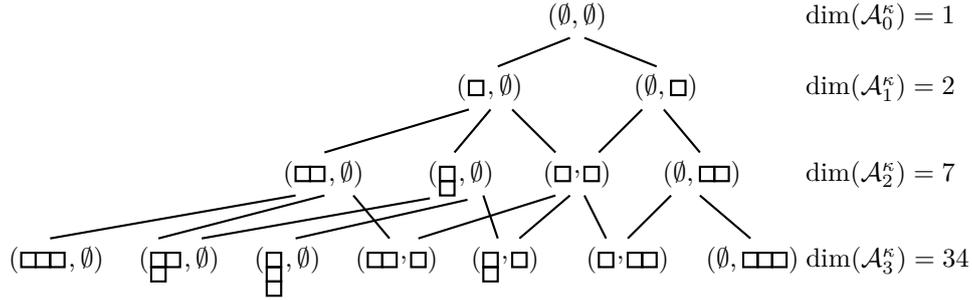
\begin{center}
\begin{tikzpicture}[scale=0.19]
\node at (0,0) {$(\emptyset,\emptyset)$};
\node at (21,0) {$\dim(\mathcal{A}_0^{\kappa})=1$};

\draw[thick] (-0.5,-1.5) -- (-5.5,-3.5);
\draw[thick] (0.5,-1.5) -- (5.5,-3.5);

\node at (-8,-5) {$($};\diag{-7.5}{-4.5}{1};\node at (-5,-5) {$, \emptyset)$};
\node at (5,-5) {$(\emptyset ,$};\diag{6.5}{-4.5}{1};\node at (8,-5) {$)$};
\node at (21,-5) {$\dim(\mathcal{A}_1^{\kappa})=2$};

\draw[thick] (-7.5,-6.5) -- (-17.5,-9.5);
\draw[thick] (-6,-6.5) -- (-8.5,-9.5);
\draw[thick] (-4.5,-6.5) -- (-1.5,-9.5);
\draw[thick] (4.5,-6.5) -- (1.5,-9.5);
\draw[thick] (6,-6.5) -- (8.5,-9.5);
%\draw[thick] (7.5,-6.5) -- (17.5,-9.5);

\node at (-20,-11) {$($};\diag{-19.5}{-10.5}{2};\node at (-16,-11) {$, \emptyset)$};
\node at (-10,-11) {$($};\diagg{-9.5}{-10.5}{1}{1};\node at (-7,-11) {$,\emptyset)$};
\node at (-2,-11) {$($};\diag{-1.5}{-10.5}{1};\node at (0,-11) {$,$};\diag{0.5}{-10.5}{1};\node at (2,-11) {$)$};
\node at (7,-11) {$(\emptyset ,$};\diag{8.5}{-10.5}{2};\node at (11,-11) {$)$};
\node at (21,-11) {$\dim(\mathcal{A}_2^{\kappa})=7$};

\draw[thick] (-19.5,-12.5) -- (-36.5,-15.5);
\draw[thick] (-17.5,-12.5) -- (-29,-15.5);
\draw[thick] (-15.5,-12.5) -- (-13,-15.5);

\draw[thick] (-10.2,-12.7) -- (-26,-15.5);
\draw[thick] (-7.6,-12.8) -- (-19.5,-15.5);
\draw[thick] (-6.5,-12.5) -- (-5.5,-15.5);

\draw[thick] (-1.5,-12.5) -- (-11,-15.5);
\draw[thick] (-0.5,-12.5) -- (-4,-15.5);
\draw[thick] (0.5,-12.5) -- (2,-15.5);

\draw[thick] (6.5,-12.5) -- (3.5,-15.5);
\draw[thick] (8.5,-12.5) -- (12,-15.5);

\node at (-39,-17) {$($};\diag{-38.5}{-16.5}{3};\node at (-34,-17) {$, \emptyset)$};
\node at (-30,-17) {$($};\diagg{-29.5}{-16.5}{2}{1};\node at (-26,-17) {$, \emptyset)$};
\node at (-22,-17) {$($};\diaggg{-21.5}{-16.5}{1}{1}{1};\node at (-19,-17) {$, \emptyset)$};
\node at (-15,-17) {$($};\diag{-14.5}{-16.5}{2};\node at (-12,-17) {$,$};\diag{-11.5}{-16.5}{1};\node at (-10,-17) {$)$};
\node at (-7,-17) {$($};\diagg{-6.5}{-16.5}{1}{1};\node at (-5,-17) {$,$};\diag{-4.5}{-16.5}{1};\node at (-3,-17) {$)$};
\node at (1,-17) {$($};\diag{1.5}{-16.5}{1};\node at (3,-17) {$,$};\diag{3.5}{-16.5}{2};\node at (6,-17) {$)$};
%\node at (9,-17) {$($};\diag{9.5}{-16.5}{1};\node at (11,-17) {$,$};\diagg{11.5}{-16.5}{1}{1};\node at (13,-17) {$)$};
\node at (10,-17) {$(\emptyset ,$};\diag{11.5}{-16.5}{3};\node at (15,-17) {$)$};
\node at (21.5,-17) {$\dim(\mathcal{A}_3^{\kappa})=34$};
\end{tikzpicture}
\end{center}
The dimension of $\mathcal{A}_n^{\kappa}$ can be easily calculated, by summing the squares of the dimensions of the irreducible representations:
\begin{equation}\label{eq: dim A_n}
\dim(\mathcal{A}_n^{\kappa})=\sum_{i=0}^n\sum_{\lambda\vdash n-i}(\dim V_{(\lambda,(i))})^2=\sum_{i=0}^n\binom{n}{i}^2\sum_{\lambda\vdash n-i}d_{\lambda}^2=\sum_{i=0}^n\binom{n}{i}^2(n-i)!\ ,
\end{equation}
where we first split the sum according to the size of the second partition $\mu$, which must be a single line of $i$ boxes, and then we use successively \eqref{dimbipart} and \eqref{dimfactn}. The dimensions for $n=0,1,2,3$ are written in the diagram above.
\paragraph{Avoiding words in $B_n$.}
A word $b = b_1 b_2 \ldots b_n$ of $B_n$ is said to be  $\bar{1}\bar{2}$-avoiding if all barred numbers in $b$ appear in decreasing order, see \cite{JS}. For instance, $35\bar{6}\bar{4}\bar{2}$ is $\bar{1}\bar{2}$-avoiding in $B_6$ while $35\bar{4}1\bar{6}\bar{2}$ is not because of the subsequence  
$\bar{4} \bar{6}$.

We will denote by $B_n(\bar{1}\bar{2})$ the subset of all signed permutations in $B_n$ which are  $\bar{1}\bar{2}$-avoiding. A word $b = b_1b_2\ldots b_n$ corresponding to a permutation in $B_n(\bar{1}\bar{2})$ can be written as follows: choose $i$ numbers in $\lbrace 1, 2, \ldots , n\rbrace$ that will be barred, choose $i$ positions among $n$ to place these barred numbers in decreasing order in $b$, and then permute the remaining $n - i$ numbers in the remaining $n - i$ positions in $b$. It follows that:
\begin{equation}\label{Bn12}
\lvert B_n(\bar{1}\bar{2})\rvert=\sum_{i=0}^{n}(n - i)! \binom{n}{i}^2 .
\end{equation}
\begin{Pro}\label{base A_n} Let $\kappa\in \mathbb{C}^2$.

The set $\{ s_\omega \mid \omega \in B_n(\bar{1}\bar{2}) \}$ is a spanning set of $\mathcal{A}_n^{\kappa}$.
\end{Pro}
\begin{proof}
Since the algebra $\mathcal{A}_n^{\kappa}$ is a quotient of $\hat{\mathcal{H}}_n^{\kappa}$, the set $\{ s_\omega \mid \omega \in B_n \}$ spans $\mathcal{A}_n^{\kappa}$. We will show that this set can be reduced to $\{ s_\omega \mid \omega \in B_n(\bar{1}\bar{2}) \}$.

We proceed by induction on the length of elements in $B_n$.

The only element $\omega$ in $B_n$ such that $\ell(\omega) < 4$ which has more than one barred integer among $\{1,2,\ldots,n\}$ is $\tau_0\tau_1\tau_0$, which corresponds to the word $\bar{2}\bar{1}$.
Indeed, the number of barred integers of an element $\omega$ is equal to $\ell_0(\omega)$. Thus, considering the set \eqref{group set} of the elements of $B_n$, all the other elements such that $\ell(\omega)<4$ have at most one barred number.

Consequently, any such element is $\bar{1}\bar{2}$-avoiding, and so $s_\omega \in \{ s_\delta \mid \delta \in B_n(\bar{1}\bar{2}) \}$. 

Now suppose that any element $s_\omega$ of $\mathcal{A}_n^{\kappa}$ with $\ell(\omega) \leq p$ for some $p \geq 3$ belongs to $\mathrm{Vect} \{ s_\omega \mid \omega \in B_n(\bar{1}\bar{2}) \}$. Consider an element $s_\omega$ with $\ell(\omega) = p+1$ such that $\omega \notin B_n(\bar{1}\bar{2})$. Then $\omega$ necessarily contains $(\tau_0 \tau_1)^2$ in some reduced expression (see \cite[Theorem 2.3]{JS}).

It is well-known that any reduced expression of an element of the group $B_n$ can be obtained from
any other by means of braid relations \eqref{R2 Bn} (see \emph{e.g.} \cite[\textsection{IV}.1.5]{B}).

The use of the braid relations of type $A$ does not provide a different element in $\hat{\mathcal{H}}_n^\kappa$ while the braid relation $(x_1s_1)^2=(s_1x_1)^2$ provides in $\hat{\mathcal{H}}_n^\kappa$ a linear combination of an element that contains $(x_1s_1)^2$ and other elements of length less than or equal to $p$ (see Remark \ref{Rem braid}).

Thus, it remains to study an element of length $p+1$ that contains $(x_1s_1)^2$, that we note $\Lambda$, the other elements appearing being of length less than or equal to $p$, and then lie in $Vect\lbrace s_\omega \mid \omega\in B_n(\bar{1}\bar{2})\rbrace$ by the recurrence hypothesis

But in $\mathcal{A}_n^{\kappa}$, the relation \eqref{(x_1s_1)^2 general} allows us to express $\Lambda$ as a linear combination of terms of length less than or equal to $p$. By the induction hypothesis, it follows that $s_\omega \in \mathrm{Vect} \{ s_\delta \mid \delta \in B_n(\bar{1}\bar{2}) \}$.

Therefore, this family generates the algebra $\mathcal{A}_n^{\kappa}$.
\end{proof}

\begin{Cor}\label{An basis}
Assume that we are in the condition \eqref{SS H_n}.

Then the set $\{ s_\omega \mid \omega \in B_n(\bar{1}\bar{2}) \}$ is a basis of $\mathcal{A}_n^{\kappa}$.
\end{Cor}
\begin{proof}
The set $\{ s_\omega \mid \omega \in B_n(\bar{1}\bar{2}) \}$ is a spanning set of $\mathcal{A}_n^\kappa$ from Proposition \ref{base A_n}. Furthermore, the condition \eqref{SS H_n} allows us to have the dimension of the algebra $\mathcal{A}_n^\kappa$ \eqref{eq: dim A_n}, which is equal to the number of elements of $B_n(\bar{1}\bar{2})$ \eqref{Bn12}.
\end{proof}
\begin{Rem}
More generally, assume that $\kappa_1,\kappa_2$ are indeterminates and that we work over the ring $\mathbb{C}[\kappa_1,\kappa_2]$. Then Proposition \ref{base A_n} is valid without change, and using the semisimplicity over $\mathbb{C}(\kappa_1,\kappa_2)$, we can prove as in Corollary \ref{An basis} that $\mathcal{A}_n^{\kappa}$ is free over $\mathbb{C}[\kappa_1,\kappa_2]$ with basis given by $\lbrace s_\omega \mid \omega \in B_n(\bar{1}\bar{2})\rbrace$.
\end{Rem}

\paragraph{The particular case $\kappa=(0,k+1)$.}
Let $n\in \mathbb{Z}_{\geq 0}$ and $k\geq 1$ such that $n\leq k$.

We are going to work with the specialisation \eqref{specialisation} of the parameters $(\kappa_1,\kappa_2)$, namely:
\[\kappa_1=0,\quad \kappa_2=k+1.\]
This is motivated by Proposition 4.3 since these two values are the eigenvalues of the element $t$.

One sees that the condition $n\leq k$ implies the semisimplicity of the algebra $\hat{\mathcal{H}}_n^{(0,k+1)}$ (see \eqref{SS H_n}).

Recall the application $\varphi$ defined at Proposition \ref{quotient}.
\begin{The} \label{An iso}
Assume $n\leq k$. Then the algebra $H_{k,n}$ is isomorphic to $\mathcal{A}_n^{(0,k+1)}$.

Equivalently, the kernel of the application $\varphi$ is the two-sided ideal generated by the idempotent $F_2^{(-1,k+1)}$.
\end{The}

\begin{proof}
We saw that the following unital algebra homomorphism is surjective:
\[
\begin{array}{cccc}
\varphi : & \hat{\mathcal{H}}^{(0,k+1)}_n &\longrightarrow & H_{k,n} \\
          & s_i &\longmapsto & \sigma_i,\\
          & x_1 &\longmapsto & t,
\end{array}
\]
for $i = 1, \ldots, n-1$.

By Lemma \ref{(s_1t)^2}, we can take the quotient and obtain an algebra morphism, still denoted $\varphi : \mathcal{A}_n^{(0,k+1)} \rightarrow H_{k,n}$. Taking the quotient does not affect surjectivity, so this morphism remains surjective.

Finally, assuming $n \leq k$, the dimension of $\mathcal{A}_n^{(0,k+1)}$ (see \eqref{eq: dim A_n}) is equal to $\dim(H_{k,n})$ (see \eqref{eq: dim H_{k,n}}) and then, in particular, is bounded above by the dimension of $H_{k,n}$, which proves the result. Then, $\varphi$ is injective. 

Hence it is an algebra isomorphism.
\end{proof}

\begin{Cor}\label{basis Hkn}
Assume $n \leq k$.

The set $\{ \varphi(s_\omega) \mid \omega \in B_n(\bar{1}\bar{2}) \}$ is a basis of the algebra $H_{k,n}$.
\end{Cor}
\begin{proof}
	The proof is straightforward from Theorem \ref{An iso} and Corollary \ref{An basis}, using the classical fact that the set formed by the images of elements of a basis under an isomorphism is a basis.
\end{proof}
\begin{Exe}
\begin{enumerate}
\item[•] \underline{$n = 2$}: The basis is:
\[
\{1, \sigma_1, t, \sigma_1 t, t \sigma_1, \sigma_1 t \sigma_1, t \sigma_1 t\},
\]
corresponding respectively to the images of 
\[
\{1, s_1, x_1, s_1 x_1, x_1 s_1, s_1 x_1 s_1, x_1 s_1 x_1\}.
\]
Adding $(x_1 s_1)^2$ to this set yields a basis of $\hat{\mathcal{H}}_2^{\kappa}$, but passing to the quotient makes it a linear combination of the others. Note that identifying $x_1$ with $s_0$, the word $(x_1 s_1)^2$ corresponds to the word $\bar{1}\bar{2}$, which is not $\bar{1}\bar{2}$-avoiding.

\item[•] \underline{$n = 3$}: The basis is:
\[
\begin{aligned}
\{ & 1,\ \sigma_2,\ \sigma_2\sigma_1,\ \sigma_2\sigma_1t,\ \sigma_2\sigma_1t\sigma_1,\ \sigma_2\sigma_1t\sigma_1\sigma_2,\ \sigma_1,\ \sigma_1\sigma_2,\ \sigma_1\sigma_2\sigma_1,\\
& \sigma_1\sigma_2\sigma_1t,\ \sigma_1\sigma_2\sigma_1t\sigma_1,\ \sigma_1\sigma_2\sigma_1t\sigma_1\sigma_2,\ \sigma_1t,\ \sigma_1t\sigma_2,\ \sigma_1t\sigma_2\sigma_1,\\
& \sigma_1t\sigma_2\sigma_1t,\ \sigma_1t\sigma_1,\ \sigma_1t\sigma_1\sigma_2,\ \sigma_1t\sigma_1\sigma_2\sigma_1,\ \sigma_1t\sigma_1\sigma_2\sigma_1t,\ \sigma_1t\sigma_1\sigma_2\sigma_1t\sigma_1,\\
& t,\ t\sigma_2,\ t\sigma_2\sigma_1,\ t\sigma_2\sigma_1t,\ t\sigma_1,\ t\sigma_1\sigma_2,\ t\sigma_1\sigma_2\sigma_1,\ t\sigma_1\sigma_2\sigma_1t,\ t\sigma_1\sigma_2\sigma_1t\sigma_1,\\
& t\sigma_1t,\ t\sigma_1t\sigma_2,\ t\sigma_1t\sigma_2\sigma_1,\ t\sigma_1t\sigma_2\sigma_1t \}.
\end{aligned}
\]

The elements removed when passing to the quotient are those of $\hat{\mathcal{H}}_3^{\kappa}$ of the form:
\[
\begin{aligned}
\{ & s_1 x_1 s_2 s_1 x_1 s_1,\ s_1 x_1 s_2 s_1 x_1 s_1 s_2,\ s_1 x_1 s_1 s_2 s_1 x_1 s_1 s_2,\ x_1 s_2 s_1 x_1 s_1,\\
& x_1 s_2 s_1 x_1 s_1 s_2,\ x_1 s_1 s_2 s_1 x_1 s_1 s_2,\ x_1 s_1 x_1 s_2 s_1 x_1 s_1,\ x_1 s_1 x_1 s_2 s_1 x_1 s_1 s_2,\\
& x_1 s_1 x_1 s_1,\ x_1 s_1 x_1 s_1 s_2,\ x_1 s_1 x_1 s_1 s_2 s_1,\ x_1 s_1 x_1 s_1 s_2 s_1 x_1,\\
& x_1 s_1 x_1 s_1 s_2 s_1 x_1 s_1,\ x_1 s_1 x_1 s_1 s_2 s_1 x_1 s_1 s_2 \},
\end{aligned}
\]
which correspond to the following words:
\[
\begin{array}{ccc}
s_1 x_1 s_2 s_1 x_1 s_1 & \leftrightsquigarrow & \bar{2} \bar{3} 1 \\
s_1 x_1 s_2 s_1 x_1 s_1 s_2 & \leftrightsquigarrow & \bar{2} 1 \bar{3} \\
s_1 x_1 s_1 s_2 s_1 x_1 s_1 s_2 & \leftrightsquigarrow & 1 \bar{2} \bar{3} \\
x_1 s_2 s_1 x_1 s_1 & \leftrightsquigarrow & \bar{1} \bar{3} 2 \\
x_1 s_2 s_1 x_1 s_1 s_2 & \leftrightsquigarrow & \bar{1} 2 \bar{3} \\
x_1 s_1 s_2 s_1 x_1 s_1 s_2 & \leftrightsquigarrow & 2 \bar{1} \bar{3} \\
x_1 s_1 x_1 s_2 s_1 x_1 s_1 & \leftrightsquigarrow & \bar{2} \bar{3} \bar{1} \\
x_1 s_1 x_1 s_2 s_1 x_1 s_1 s_2 & \leftrightsquigarrow & \bar{2} \bar{1} \bar{3} \\
x_1 s_1 x_1 s_1 & \leftrightsquigarrow & \bar{1} \bar{2} 3 \\
x_1 s_1 x_1 s_1 s_2 & \leftrightsquigarrow & \bar{1} 3 \bar{2} \\
x_1 s_1 x_1 s_1 s_2 s_1 & \leftrightsquigarrow & 3 \bar{1} \bar{2} \\
x_1 s_1 x_1 s_1 s_2 s_1 x_1 & \leftrightsquigarrow & \bar{3} \bar{1} \bar{2} \\
x_1 s_1 x_1 s_1 s_2 s_1 x_1 s_1 & \leftrightsquigarrow & \bar{1} \bar{3} \bar{2} \\
x_1 s_1 x_1 s_1 s_2 s_1 x_1 s_1 s_2 & \leftrightsquigarrow & \bar{1} \bar{2} \bar{3}
\end{array}
\]

These are exactly the elements of $B_3 \setminus B_3(\bar{1}\bar{2})$.
\end{enumerate}
\end{Exe}
The Bratteli diagram for the chain $\lbrace \mathcal{A}_n^{(0,k+1)}\rbrace_{n\geq 0}$ for the first four levels, for $n\leq k$, where the bipartitions shown are those in which the second component is a partition of length at most 1, is the following:
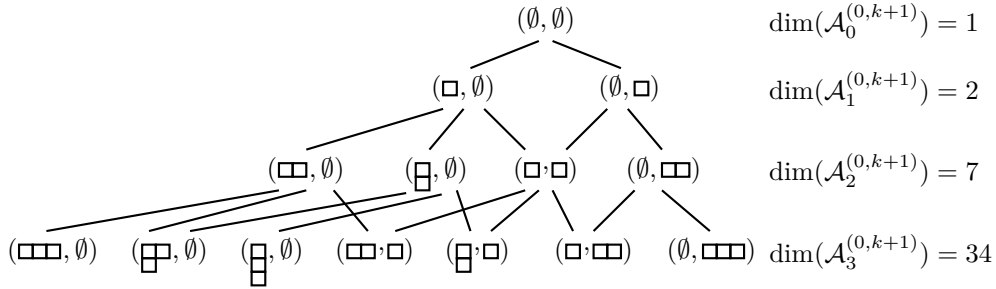
\begin{center}
\begin{tikzpicture}[scale=0.18]
\node at (0,0) {$(\emptyset,\emptyset)$};
\node at (24,0) {$\dim(\mathcal{A}_0^{(0,k+1)})=1$};
\draw[thick] (-0.5,-1.5) -- (-5.5,-3.5);
\draw[thick] (0.5,-1.5) -- (5.5,-3.5);

\node at (-8,-5) {$($};\diag{-7.5}{-4.5}{1};\node at (-5,-5) {$,\emptyset)$};
\node at (5,-5) {$(\emptyset ,$};\diag{6.5}{-4.5}{1};\node at (8,-5) {$)$};
\node at (24,-5) {$\dim(\mathcal{A}_1^{(0,k+1)})=2$};
\draw[thick] (-7.5,-6.5) -- (-17.5,-9.5);
\draw[thick] (-6,-6.5) -- (-8.5,-9.5);
\draw[thick] (-4.5,-6.5) -- (-1.5,-9.5);
\draw[thick] (4.5,-6.5) -- (1.5,-9.5);
\draw[thick] (6,-6.5) -- (8.5,-9.5);
%\draw[thick] (7.5,-6.5) -- (17.5,-9.5);

\node at (-20,-11) {$($};\diag{-19.5}{-10.5}{2};\node at (-16,-11) {$,\emptyset)$};
\node at (-10,-11) {$($};\diagg{-9.5}{-10.5}{1}{1};\node at (-7,-11) {$, \emptyset)$};
\node at (-2,-11) {$($};\diag{-1.5}{-10.5}{1};\node at (0,-11) {$,$};\diag{0.5}{-10.5}{1};\node at (2,-11) {$)$};
\node at (7,-11) {$(\emptyset ,$};\diag{8.5}{-10.5}{2};\node at (11,-11) {$)$};
%\node at (17,-11) {$(\emptyset\,,$};\diagg{18.5}{-10.5}{1}{1};\node at (20,-11) {$)$};
\node at (24,-11) {$\dim(\mathcal{A}_2^{(0,k+1)})=7$};

\draw[thick] (-19.5,-12.5) -- (-36.5,-15.5);
\draw[thick] (-17.5,-12.5) -- (-29,-15.5);
\draw[thick] (-15.5,-12.5) -- (-13,-15.5);

\draw[thick] (-10.2,-12.7) -- (-26,-15.5);
\draw[thick] (-7.6,-12.8) -- (-19.5,-15.5);
\draw[thick] (-6.5,-12.5) -- (-5.5,-15.5);

\draw[thick] (-1.5,-12.5) -- (-11,-15.5);
\draw[thick] (-0.5,-12.5) -- (-4,-15.5);
\draw[thick] (0.5,-12.5) -- (2,-15.5);

\draw[thick] (6.5,-12.5) -- (3.5,-15.5);
\draw[thick] (8.5,-12.5) -- (12,-15.5);

\node at (-39,-17) {$($};\diag{-38.5}{-16.5}{3};\node at (-34,-17) {$,\emptyset)$};
\node at (-30,-17) {$($};\diagg{-29.5}{-16.5}{2}{1};\node at (-26,-17) {$,\emptyset)$};
\node at (-22,-17) {$($};\diaggg{-21.5}{-16.5}{1}{1}{1};\node at (-19,-17) {$,\emptyset)$};
\node at (-15,-17) {$($};\diag{-14.5}{-16.5}{2};\node at (-12,-17) {$,$};\diag{-11.5}{-16.5}{1};\node at (-10,-17) {$)$};
\node at (-7,-17) {$($};\diagg{-6.5}{-16.5}{1}{1};\node at (-5,-17) {$,$};\diag{-4.5}{-16.5}{1};\node at (-3,-17) {$)$};
\node at (1,-17) {$($};\diag{1.5}{-16.5}{1};\node at (3,-17) {$,$};\diag{3.5}{-16.5}{2};\node at (6,-17) {$)$};
%\node at (9,-17) {$($};\diag{9.5}{-16.5}{1};\node at (11,-17) {$,$};\diagg{11.5}{-16.5}{1}{1};\node at (13,-17) {$)$};
\node at (10,-17) {$(\emptyset ,$};\diag{11.5}{-16.5}{3};\node at (15,-17) {$)$};
%\node at (26,-17) {$(\emptyset\,,$};\diagg{27.5}{-16.5}{2}{1};\node at (30,-17) {$)$};
%\node at (34,-17) {$(\emptyset\,,$};\diagg{35.5}{-16.5}{1}{1}{1};\node at (37,-17) {$)$};
\node at (24.5,-17) {$\dim(\mathcal{A}_3^{(0,k+1)})=34$};
\end{tikzpicture}
\end{center}
Note the bijection with the diagram shown in Section \ref{ss22}, using the bijection explained in Section \ref{sssec411}.
\subsection{Case where $n>k$}
Let $n\in \mathbb{Z}_{\geq 0}$ and $k\geq 1$ such that $n>k$.

We remark that the basis \eqref{basis set} of $\hat{\mathcal{H}}_n^\kappa$ is well adapted to the inclusion $B_{n-1}\subset B_n$ and then we identify an element of $B_k$ as an element of $B_n$ via this inclusion.

We set the following element of $\hat{\mathcal{H}}_n^{\kappa}$:
\begin{equation}
\tilde{F}_{k+1}^{(1,\kappa_1)}=\sum_{\omega\in B_{k+1}}
\left(\prod_{i=0}^k(i- \kappa_2)^{1-m_i(\omega)}\right)s_\omega.
\end{equation}

\begin{Def} \label{def: A_n^k}
We define the algebra $\mathcal{A}_n^{\kappa,(k)}$ as the quotient of $\mathcal{A}_n^\kappa$ by the relation:
\begin{equation}\label{rel longest element}
\tilde{F}_{k+1}^{(1,\kappa_1)}=0,
\end{equation}
where $\tilde{F}_{k+1}^{(1,\kappa_1)}$ is seen as an element of $\mathcal{A}_n^{\kappa}$.

It is understood that $\mathcal{A}_n^{\kappa,(k)}=\mathcal{A}_n^{\kappa}$ if $n\leq k$.
\end{Def}
\begin{Rem}
	If $\kappa_2-\kappa_1 \in \mathbb{C}\setminus \lbrace 0,1,\dots ,k\rbrace$, using the explicit expression of $F_{k+1}^{(1,\kappa_1)}$ (see \eqref{idempotent formula}), this is equivalent to imposing the following relation in $\mathcal{A}_n^\kappa$:
\begin{equation}
F_{k+1}^{(1,\kappa_1)} = 0,
\end{equation}
where $F_{k+1}^{(1,\kappa_1)}$ is viewed as an element of $\mathcal{A}_n^{\kappa}$.

Indeed, one sees that $F_{k+1}^{(1,\kappa_1)}=D_{k+1}(1,\kappa_1)\tilde{F}_{k+1}^{(1,\kappa_1)}$ with $D_{k+1}(1,\kappa_1)\neq 0$.
\end{Rem}
For $k=1,\dots ,n-1$, we set the following element of $\hat{\mathcal{H}}_n^\kappa$: 
\begin{equation}\label{longest element}
\Delta_k:=x_1s_1x_1\dots s_k\dots s_1x_1.
\end{equation}
It is understood that $\Delta_0=x_1$.
\begin{Pro}\label{coef} Let $n\in \mathbb{Z}_{\geq 1}$.

The basis element \eqref{longest element} for $k=n-1$ appears with coefficient $n!$ when $\tilde{F}_n^{(1,\kappa_1)}$ is expressed in the basis $\{ s_\omega \mid \omega \in B_n(\bar{1}\bar{2}) \}$ of $\mathcal{A}_n^{\kappa}$.
\end{Pro}

\begin{proof}
We prove the result by induction on $n$. 

For $n=1$, we have $\tilde{F}_1^{(1,\kappa_1)} = x_1-\kappa_2$ and then the coefficient in front of $\Delta_0=x_1$ is $1$.

Assume the result holds for $n \geq 1$. We obtain from \eqref{Rec idempotent} the following recurence formula:
\begin{equation}\label{Rec idempotent 2}
\tilde{F}_{n+1}^{(1,\kappa_1)}=\tilde{F}_n^{(1,\kappa_1)}\left( \sum_{j=0}^n [n,-j]+(n-\kappa_2)\sum_{j=1}^{n+1}[n,j]\right).
\end{equation}

Using the recurrence hypothesis and focusing on terms contributing to the coefficient in front of $\Delta_n$, we have:
\begin{equation}\label{delta rec}
n! \Delta_n \left(1+\sum_{j=1}^n s_1\dots s_j\right).
\end{equation}

Fix $j = 1, \dots, n$. We have:
\[
\Delta_n(s_1\dots s_j) = x_1s_1x_1 \dots s_{n-1}\dots s_1s_n \dots s_2(x_1s_1)^2 s_2 \dots s_j.
\]

Using relation \eqref{(x_1s_1)^2 general} in $\mathcal{A}_n^{\kappa}$, we can rewrite $(x_1s_1)^2$ in $\mathcal{A}_n^{\kappa}$.

The only term not reducing the number of occurrences of $x_1$ in \eqref{(x_1s_1)^2 general} is $x_1s_1x_1$. Each $s_1,s_2, \dots,s_j$ can be pushed  one by one through the remaining expression, yielding a new $(x_1s_1)^2$, which again reduces to $x_1s_1x_1$. Thus, every term in \eqref{delta rec} produces exactly one full expression of the form $\Delta_n$, yielding $(n+1)$ such terms.

Finally, the coefficient in front of $\Delta_n$ is $(n+1)!$ when $\tilde{F}_{n+1}^{(1,\kappa_1)}$ is expressed in the basis $\{ s_\omega \mid \omega \in B_{n+1}(\bar{1}\bar{2}) \}$ of $\mathcal{A}_{n+1}^{\kappa}$.

This concludes the proof.
\end{proof}

Thanks to Proposition \ref{coef} and Corollary \ref{An basis} applied to $n=k+1$, we can rewrite the relation \eqref{rel longest element} as:
\begin{equation}\label{new quo}
\Delta_k=\sum_{\omega\in B_{k+1}(\bar{1}\bar{2})\setminus \lbrace \Delta_k\rbrace}\lambda_\omega s_\omega,
\end{equation}
for $\lambda_\omega \in \mathbb{C}$.

There is a unique element $\omega$ in $B_{k+1}(\bar{1}\bar{2})$ such that $\ell_0(\omega) = k+1$, all others having strictly fewer $x_1$ occurrences. This element is $\Delta_k$ which corresponds to the word $\overline{k+1}\bar{k}\dots \bar{1}$.

In particular, the right hand side of \eqref{new quo} is a linear combination of elements indexed by $\omega\in B_{k+1}(\bar{1}\bar{2})$ such that $\ell_0(\omega)\leq k$.
\paragraph{Artin-Wedderburn decomposition for $\mathcal{A}_n^{\kappa,(k)}$.}
We assume that the condition \eqref{SS H_n} is satisfied, that is:
\[\kappa_1-\kappa_2\in \mathbb{C}\setminus \lbrace 0,\pm 1,\dots ,\pm (n-1)\rbrace.\]

In the particular case of $\mathcal{A}_n^{\kappa,(k)}$, which is the quotient of $\mathcal{A}_n^{\kappa}$ by the relation $F_{k+1}^{(1,\kappa_1)}=0$, we recall the indexing of one-dimensional representations set up in (\ref{1dimrep}), and we find that the disappearing representations are those $V_{(\lambda,\mu)}$ with $\lambda$ with at least $k+1$ columns and $\mu$ having at least two non-empty rows. 
Then, the algebra $\mathcal{A}_n^{\kappa,(k)}$ is semisimple with the following set of irreducible representations:
\[Irr(\mathcal{A}_n^{\kappa,(k)})=\{V_{(\lambda,\mu)}\ |\ (\lambda,\mu)\in Par_2(n)\ \text{such that}\ \lambda_1<k+1\ \text{and}\ \ell(\mu)<2\}\ .\]

For $k=2$, the Bratteli diagram for the chain $\lbrace \mathcal{A}_n^{\kappa,(2)}\rbrace_{n\geq 0}$ for the first four levels, where the bipartitions shown are those in which the second component is a partition of length at most 1 and for $n=3$ only, we delete the partition $((3),\emptyset)$:
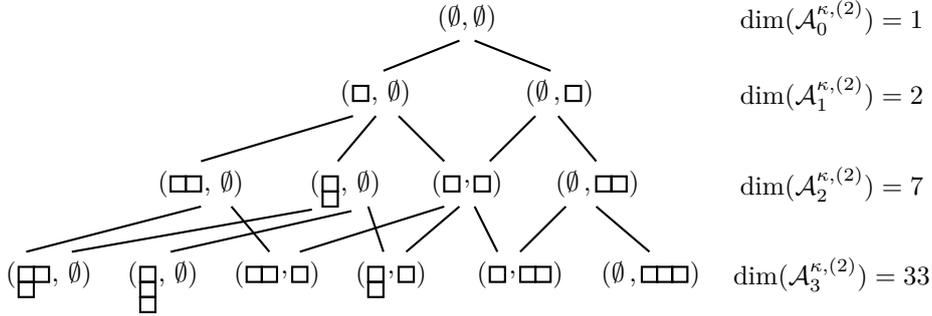
\begin{center}
\begin{tikzpicture}[scale=0.20]
\node at (0,0) {$(\emptyset,\emptyset)$};
\node at (24,0) {$\dim(\mathcal{A}_0^{\kappa,(2)})=1$};

\draw[thick] (-0.5,-1.5) -- (-5.5,-3.5);
\draw[thick] (0.5,-1.5) -- (5.5,-3.5);

\node at (-8,-5) {$($};\diag{-7.5}{-4.5}{1};\node at (-5,-5) {$,\,\emptyset)$};
\node at (5,-5) {$(\emptyset\,,$};\diag{6.5}{-4.5}{1};\node at (8,-5) {$)$};
\node at (24,-5) {$\dim(\mathcal{A}_1^{\kappa,(2)})=2$};
\draw[thick] (-7.5,-6.5) -- (-17.5,-9.5);
\draw[thick] (-6,-6.5) -- (-8.5,-9.5);
\draw[thick] (-4.5,-6.5) -- (-1.5,-9.5);
\draw[thick] (4.5,-6.5) -- (1.5,-9.5);
\draw[thick] (6,-6.5) -- (8.5,-9.5);
%\draw[thick] (7.5,-6.5) -- (17.5,-9.5);

\node at (-20,-11) {$($};\diag{-19.5}{-10.5}{2};\node at (-16,-11) {$,\,\emptyset)$};
\node at (-10,-11) {$($};\diagg{-9.5}{-10.5}{1}{1};\node at (-7,-11) {$,\,\emptyset)$};
\node at (-2,-11) {$($};\diag{-1.5}{-10.5}{1};\node at (0,-11) {$,$};\diag{0.5}{-10.5}{1};\node at (2,-11) {$)$};
\node at (7,-11) {$(\emptyset\,,$};\diag{8.5}{-10.5}{2};\node at (11,-11) {$)$};
%\node at (17,-11) {$(\emptyset\,,$};\diagg{18.5}{-10.5}{1}{1};\node at (20,-11) {$)$};
\node at (24,-11) {$\dim(\mathcal{A}_2^{\kappa,(2)})=7$};

\draw[thick] (-17.5,-12.5) -- (-29,-15.5);
\draw[thick] (-15.5,-12.5) -- (-13,-15.5);

\draw[thick] (-10.2,-12.7) -- (-26,-15.5);
\draw[thick] (-7.6,-12.8) -- (-19.5,-15.5);
\draw[thick] (-6.5,-12.5) -- (-5.5,-15.5);

\draw[thick] (-1.5,-12.5) -- (-11,-15.5);
\draw[thick] (-0.5,-12.5) -- (-4,-15.5);
\draw[thick] (0.5,-12.5) -- (2,-15.5);

\draw[thick] (6.5,-12.5) -- (3.5,-15.5);
\draw[thick] (8.5,-12.5) -- (12,-15.5);

\node at (-30,-17) {$($};\diagg{-29.5}{-16.5}{2}{1};\node at (-26,-17) {$,\,\emptyset)$};
\node at (-22,-17) {$($};\diaggg{-21.5}{-16.5}{1}{1}{1};\node at (-19,-17) {$,\,\emptyset)$};
\node at (-15,-17) {$($};\diag{-14.5}{-16.5}{2};\node at (-12,-17) {$,$};\diag{-11.5}{-16.5}{1};\node at (-10,-17) {$)$};
\node at (-7,-17) {$($};\diagg{-6.5}{-16.5}{1}{1};\node at (-5,-17) {$,$};\diag{-4.5}{-16.5}{1};\node at (-3,-17) {$)$};
\node at (1,-17) {$($};\diag{1.5}{-16.5}{1};\node at (3,-17) {$,$};\diag{3.5}{-16.5}{2};\node at (6,-17) {$)$};
%\node at (9,-17) {$($};\diag{9.5}{-16.5}{1};\node at (11,-17) {$,$};\diagg{11.5}{-16.5}{1}{1};\node at (13,-17) {$)$};
\node at (10,-17) {$(\emptyset\,,$};\diag{11.5}{-16.5}{3};\node at (15,-17) {$)$};
%\node at (26,-17) {$(\emptyset\,,$};\diagg{27.5}{-16.5}{2}{1};\node at (30,-17) {$)$};
%\node at (34,-17) {$(\emptyset\,,$};\diagg{35.5}{-16.5}{1}{1}{1};\node at (37,-17) {$)$};
\node at (24,-17) {$\dim(\mathcal{A}_3^{\kappa,(2)})=33$};
\end{tikzpicture}
\end{center}
\paragraph{Avoiding words in $B_n(\bar{1}\bar{2})$.}
We define the subset $B_n(\bar{1}\bar{2},\overline{k+1} \dots \bar{1})$ of $B_n(\bar{1}\bar{2})$ as the set of $\bar{1}\bar{2}$-avoiding words with at most $k$ barred integers.
For instance, $2\bar{5}\bar{4}1\bar{3}6 \in B_6(\bar{1}\bar{2},\bar{4}\bar{3}\bar{2}\bar{1})$ but is not in $B_6(\bar{1}\bar{2},\bar{3}\bar{2}\bar{1})$.

\begin{Lem}
We have the following identity:
\begin{equation} \label{Bn cardinal}
|B_n(\bar{1}\bar{2},\overline{k+1} \dots \bar{1})| = \sum_{i=0}^{\min(k,n)} \binom{n}{i}^2 (n - i)!.
\end{equation}
\end{Lem}

\begin{proof}
We choose a number $i \in \{1,\dots,n\}$ of barred elements, with $i \leq k$. We then choose $i$ positions among the $n$ available to place the barred elements. Finally, we choose a permutation of the remaining $n - i$ integers to fill the remaining positions. This gives the stated formula.
\end{proof}

\begin{Rem}
There is another useful description of the set $B_n(\bar{1}\bar{2},\overline{k+1} \dots \bar{1})$, namely:
\begin{equation}
B_n(\bar{1}\bar{2},\overline{k+1} \dots \bar{1}) = B_n(\bar{1}\bar{2}) \cap \{\omega \in B_n \mid \ell_0(\omega) < k+1\}.
\end{equation}
In particular, we have $B_n(\bar{1}\bar{2},\overline{k+1} \dots \bar{1}) = B_n(\bar{1}\bar{2})$ if and only if $n \leq k$ (see \cite{SR}).
\end{Rem}
\begin{The} \label{Thm gen}
Let $\kappa\in \mathbb{C}^2$.

The set $\{ s_\omega \mid \omega \in B_n(\bar{1}\bar{2},\overline{k+1} \dots \bar{1}) \}$ spans $\mathcal{A}_n^{\kappa,(k)}$.
\end{The}
\begin{proof}
Recall that $\{ s_\omega \mid \omega \in B_n(\bar{1}\bar{2}) \}$ spans $\mathcal{A}_n^{\kappa}$ (Proposition \ref{base A_n}). Recall also that there is a unique element in this basis such that $\ell_0(\omega) = n$, all others having strictly fewer $x_1$ occurrences. That element is:
\begin{equation} 
\bar{n}\overline{n-1} \dots \bar{1} \leftrightsquigarrow \Delta_{n-1}=x_1s_1x_1s_2s_1x_1 \dots s_{n-1} \dots s_1x_1.
\end{equation}

We now show using Lemma \ref{coef} that any $s_\omega$ with $\omega \in B_n$ such that $\ell_0(\omega) \geq k+1$ can be rewritten in $\mathcal{A}_n^{\kappa,(k)}$ as a linear combination of elements $s_\delta$ with $\ell_0(\delta) \leq k$.

Write such an element in its standard form (see \eqref{Standard form}). Braid relations allow us to isolate a subexpression of the form:
\[
(s_{i_1} \dots s_1 x_1)(s_{i_2} \dots s_1 x_1) \dots (s_{i_{k+1}} \dots s_1 x_1), \quad 1 \leq i_1 < \dots < i_{k+1}.
\]
Using commutativity relations to push $s_i$ to the left, we recover the element (\ref{longest element}) for $n = k+1$. Thanks to Proposition \ref{coef} applied to $n=k+1$, the relation \eqref{new quo} in $\mathcal{A}_n^{\kappa,(k)}$ then allows rewriting this element as a combination of elements $s_\delta$ with $\ell_0(\delta) \leq k$, thus reducing the number of $x_1$ occurrences. By induction, we conclude that $\mathcal{A}_n^{\kappa,(k)}$ is generated by such elements.

Finally, for any $\omega \in B_n \setminus B_n(\bar{1}\bar{2})$ with $\ell_0(\omega) \leq k$, the relation (\ref{(x_1s_1)^2 general}) allows us to express $s_\omega$ (as in Theorem \ref{base A_n}) in terms of $s_\delta$, $\delta \in B_n(\bar{1}\bar{2})$. Moreover, $\ell_0$ does not increase, as the relation involves terms where $x_1$ appears at most twice. The result follows.
\end{proof}
We are now ready to give an algebraic description of the algebra of fused permutations.

Let $n\in \mathbb{Z}_{\geq 0}$ and $k\geq 1$.

We consider now the specialization \eqref{specialisation}, that is:
\[\kappa_1=0,\quad \kappa_2=k+1.\]
\begin{The} \label{Iso}
There exists an algebra isomorphism between $\mathcal{A}_n^{(0,k+1),(k)}$ and $H_{k,n}$ given by:
\begin{equation} \label{Isomorphism}
\begin{array}{ccccc}
\varphi : & \mathcal{A}_n^{(0,k+1),(k)} & \longrightarrow & H_{k,n} \\
& s_i & \longmapsto & \sigma_i, & i=1,\dots,n-1, \\
& x_1 & \longmapsto & t := \mathbf{1} + k\sigma_0.
\end{array}
\end{equation}
\end{The}

\begin{proof}
The morphism $\varphi:\mathcal{A}_n^{(0,k+1)}\longrightarrow H_{k,n}$ defined in the proof of Theorem \ref{An iso} remains a morphism upon passing to the quotient by the two sided-ideal generated by $\tilde{F}_{k+1}^{(1,0)}$. Indeed, if $n\leq k$, there is nothing to verify since $\mathcal{A}_n^{(0,k+1),(k)}=\mathcal{A}_n^{(0,k+1)}$ while in the case where $n>k$, if the image of $\tilde{F}_{k+1}^{(1,0)}$ under $\varphi$ were non-zero in $H_{k,n}$, this in turn would imply the existence of the following one-dimensional representation of $H_{k,n}$ defined on the following generators (see Proposition \ref{quotient}): 
\[\begin{array}{ccc}
t & \longmapsto & 0,\\
\sigma_i & \longmapsto & 1,
\end{array}
\]
for $i=1,\dots,n-1$.
In particular, since $t=1+k\sigma_0$, we would have that $\sigma_0$ is sent via this representation on $-\frac{1}{k}$.
But we already discussed the non-existence of such a one-dimensional representation around \eqref{1rep FP}. Thus, the image of $\tilde{F}_{k+1}^{(1,0)}$ under $\varphi$ is zero in $H_{k,n}$ and then we obtain the morphism property.

Moreover, the morphism $\varphi$ remains surjective upon passing to the quotient.

Thus, it remains to prove that the latter is injective.

Since the family $\{ s_\omega \mid \omega \in B_n(\bar{1}\bar{2}, \overline{k+1} \dots \bar{1}) \}$ generates $\mathcal{A}_n^{(0,k+1),(k)}$ and has cardinality equal to the dimension of the algebra $H_{k,n}$, we deduce that \[\dim(\mathcal{A}_n^{(0,k+1),(k)}) \leq \dim(H_{k,n}),\] hence the map is injective.

Being both surjective and injective, the map is an isomorphism.
\end{proof}
In particular, in Theorem \ref{Iso}, we have proved that the set \[\{ s_\omega \mid \omega \in B_n(\bar{1}\bar{2},\overline{k+1} \dots \bar{1}) \}\] is a basis set of $\mathcal{A}_n^{(0,k+1),(k)}$.
\begin{Cor}
The family $\{ \varphi(s_\omega) \mid \omega \in B_n(\bar{1}\bar{2}, \overline{k+1} \dots \bar{1}) \}$ is a basis of the algebra $H_{k,n}$.
\end{Cor}
\begin{proof}
	The proof is straightforward using the same reasoning as in the proof of Corollary \ref{basis Hkn}.
\end{proof}
\begin{Exe}
\begin{enumerate}
\item[•] \underline{$n = 2$ and  $k = 1$}: The basis is 
\[
\{1, \sigma_1, t, \sigma_1 t, t \sigma_1, \sigma_1 t \sigma_1\},
\]
corresponding respectively to the images of 
\begin{equation}\label{A_2 basis}
\{1, s_1, x_1, s_1 x_1, x_1 s_1, s_1 x_1 s_1\}.
\end{equation}
The only element removed when passing to the quotient $\mathcal{A}_2\twoheadrightarrow \mathcal{A}_2^{(0,2),(1)}$ is $x_1s_1x_1$. It's exactly the element of $B_3(\bar{1}\bar{2})\setminus B_3(\bar{1}\bar{2},\bar{2}\bar{1}) $.

\item[•] \underline{$n=3$ and $k=2$}: The basis is:
\[
\begin{aligned}
\{ & 1,\ \sigma_2,\ \sigma_2\sigma_1,\ \sigma_2\sigma_1t,\ \sigma_2\sigma_1t\sigma_1,\ \sigma_2\sigma_1t\sigma_1\sigma_2,\ \sigma_1,\ \sigma_1\sigma_2,\ \sigma_1\sigma_2\sigma_1,\\
& \sigma_1\sigma_2\sigma_1t,\ \sigma_1\sigma_2\sigma_1t\sigma_1,\ \sigma_1\sigma_2\sigma_1t\sigma_1\sigma_2,\ \sigma_1t,\ \sigma_1t\sigma_2,\ \sigma_1t\sigma_2\sigma_1,\\
& \sigma_1t\sigma_2\sigma_1t,\ \sigma_1t\sigma_1,\ \sigma_1t\sigma_1\sigma_2,\ \sigma_1t\sigma_1\sigma_2\sigma_1,\ \sigma_1t\sigma_1\sigma_2\sigma_1t,\ \sigma_1t\sigma_1\sigma_2\sigma_1t\sigma_1,\\
& t,\ t\sigma_2,\ t\sigma_2\sigma_1,\ t\sigma_2\sigma_1t,\ t\sigma_1,\ t\sigma_1\sigma_2,\ t\sigma_1\sigma_2\sigma_1,\ t\sigma_1\sigma_2\sigma_1t,\ t\sigma_1\sigma_2\sigma_1t\sigma_1,\\
& t\sigma_1t,\ t\sigma_1t\sigma_2,\ t\sigma_1t\sigma_2\sigma_1 \}.
\end{aligned}
\]

The only element removed when passing to the quotient $\mathcal{A}_3^{(0,3)}\twoheadrightarrow \mathcal{A}_3^{(0,3),(2)}$ is $x_1s_1x_1s_2s_1x_1$. It's exactly the element of $B_3(\bar{1}\bar{2})\setminus B_3(\bar{1}\bar{2},\bar{3}\bar{2}\bar{1})$.
\end{enumerate}
\end{Exe}
\paragraph{Dimension of $\mathcal{A}_n^{\kappa,(k)}$.}
Let $\kappa\in \mathbb{C}^2$.

We have from Theorem \ref{Thm gen}, \eqref{eq: dim H_{k,n}} and \eqref{Bn cardinal} the following inequality:
\begin{equation}\label{inequality}
\dim(\mathcal{A}_n^{\kappa,(k)})\leq \dim(H_{k,n}).
\end{equation}

Assume now that $\kappa$ satisfy the condition \eqref{SS H_n}.

Thus, using the semisimplicity of the algebra $\mathcal{A}_n^{\kappa,(k)}$, the inequality \eqref{inequality} can be rewritten as:
\begin{equation}\label{inequality 2}
\sum_{\boldsymbol{\lambda}\in Irr(\mathcal{A}_n^{\kappa,(k)})}d_{\boldsymbol{\lambda}}^2\leq \sum_{i=0}^{\min(k,n)}\binom{n}{i}^2(n-i)!.
\end{equation}

If $n\leq k+1$, the inequality \eqref{inequality 2} is an equality. Indeed, in this case, the algebras $\mathcal{A}_n^{\kappa,(k)}$ and $\mathcal{A}_n^{(0,k+1),(k)}$ are both semisimple with $Irr(\mathcal{A}_n^{\kappa,(k)})=Irr(\mathcal{A}_n^{(0,k+1),(k)})$ and for each $\boldsymbol{\lambda}\in Irr(\mathcal{A}_n^{\kappa,(k)})$, the dimensions of the irreducible representations of both algebras indexed by $\boldsymbol{\lambda}$ are the same (see \eqref{dim W}). From the isomorphism \eqref{AW formula}, we deduce that the two algebras have the same dimension, and thus, from Theorem \ref{Iso}, we obtain the equality of dimensions.

However, in the case where $n>k+1$, the inequality \eqref{inequality 2} can be strict.

Indeed, for $k=1$, the Bratteli diagram for the chain $\lbrace \mathcal{A}_n^{\kappa,(1)}\rbrace_{n\geq 0}$ for the first four levels, where the bipartitions shown are those in which the first component is a partition with at most 1 column and the second component is a partition of length at most 1:

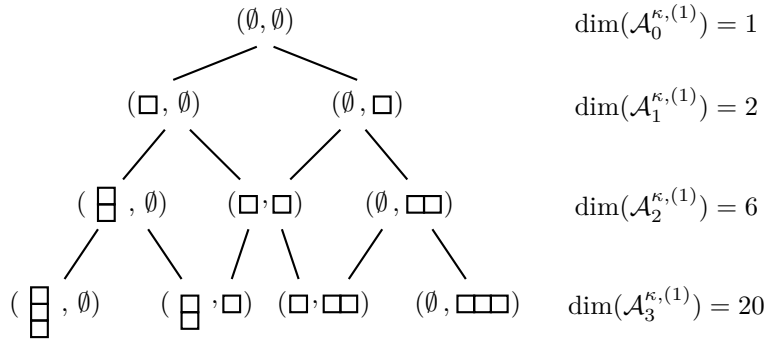
\begin{center}
\begin{tikzpicture}[scale=0.22]
\node at (0,0) {$(\emptyset,\emptyset)$};
\node at (24,0) {$\dim(\mathcal{A}_0^{\kappa,(1)})=1$};

\draw[thick] (-0.5,-1.5) -- (-5.5,-3.5);
\draw[thick] (0.5,-1.5) -- (5.5,-3.5);

\node at (-8,-5) {$($};\diag{-7.5}{-4.5}{1};\node at (-5,-5) {$,\,\emptyset)$};
\node at (5,-5) {$(\emptyset\,,$};\diag{6.5}{-4.5}{1};\node at (8,-5) {$)$};
\node at (24,-5) {$\dim(\mathcal{A}_1^{\kappa,(1)})=2$};

%\draw[thick] (-7.5,-6.5) -- (-17.5,-9.5);
\draw[thick] (-6,-6.5) -- (-8.5,-9.5);
\draw[thick] (-4.5,-6.5) -- (-1.5,-9.5);
\draw[thick] (4.5,-6.5) -- (1.5,-9.5);
\draw[thick] (6,-6.5) -- (8.5,-9.5);
%\draw[thick] (7.5,-6.5) -- (17.5,-9.5);

\node at (-11,-11) {$($};\diagg{-10}{-10}{1}{1};\node at (-7,-11) {$,\,\emptyset)$};
%\node at (-10,-11) {$($};\diagg{-9.5}{-10.5}{1}{1};\node at (-7,-11) {$,\,\emptyset)$};
\node at (-2,-11) {$($};\diag{-1.5}{-10.5}{1};\node at (0,-11) {$,$};\diag{0.5}{-10.5}{1};\node at (2,-11) {$)$};
\node at (7,-11) {$(\emptyset\,,$};\diag{8.5}{-10.5}{2};\node at (11,-11) {$)$};
%\node at (17,-11) {$(\emptyset\,,$};\diagg{18.5}{-10.5}{1}{1};\node at (20,-11) {$)$};
\node at (24,-11) {$\dim(\mathcal{A}_2^{\kappa,(1)})=6$};

\draw[thick] (-10,-12.5) -- (-12,-15.5);
\draw[thick] (-7,-12.5) -- (-5,-15.5);

\draw[thick] (-1,-12.5) -- (-2,-15.5);
\draw[thick] (1,-12.5) -- (2,-15.5);

\draw[thick] (7,-12.5) -- (5,-15.5);
\draw[thick] (10,-12.5) -- (12,-15.5);

\node at (-15,-17) {$($};\diaggg{-14}{-16}{1}{1}{1};\node at (-11,-17) {$,\,\emptyset)$};
%\node at (-30,-17) {$($};\diagg{-29.5}{-16.5}{2}{1};\node at (-26,-17) {$,\,\emptyset)$};
%\node at (-22,-17) {$($};\diaggg{-21.5}{-16.5}{1}{1}{1};\node at (-19,-17) {$,\,\emptyset)$};
\node at (-6,-17) {$($};\diagg{-5}{-16.5}{1}{1};\node at (-3,-17) {$,$};\diag{-2.5}{-16.5}{1};\node at (-1,-17) {$)$};
%\node at (-7,-17) {$($};\diagg{-6.5}{-16.5}{1}{1};\node at (-5,-17) {$,$};\diag{-4.5}{-16.5}{1};\node at (-3,-17) {$)$};
\node at (1,-17) {$($};\diag{1.5}{-16.5}{1};\node at (3,-17) {$,$};\diag{3.5}{-16.5}{2};\node at (6,-17) {$)$};
%\node at (9,-17) {$($};\diag{9.5}{-16.5}{1};\node at (11,-17) {$,$};\diagg{11.5}{-16.5}{1}{1};\node at (13,-17) {$)$};
\node at (10,-17) {$(\emptyset\,,$};\diag{11.5}{-16.5}{3};\node at (15,-17) {$)$};
%\node at (26,-17) {$(\emptyset\,,$};\diagg{27.5}{-16.5}{2}{1};\node at (30,-17) {$)$};
%\node at (34,-17) {$(\emptyset\,,$};\diagg{35.5}{-16.5}{1}{1}{1};\node at (37,-17) {$)$};
\node at (24,-17) {$\dim(\mathcal{A}_3^{\kappa,(1)})=20$};
\end{tikzpicture}
\end{center}
The dimension of $\mathcal{A}_n^{\kappa,(1)}$ is then easily calculated:
\begin{equation}\label{dim A1}
\dim(\mathcal{A}_n^{\kappa,(1)})=\sum_{i=0}^n d_{((1^{n-i}),(i))}^2 =\sum_{i=0}^n\binom{n}{i}^2=\binom{2n}{n}\ .
\end{equation}
For example, the equality \eqref{dim A1} for $n=3$ gives $\dim(\mathcal{A}_3^{\kappa,(1)})=20$. But we know that $H_{1,n}\cong \mathbb{C}\mathfrak{S}_{n+1}$ as algebras. Then for $n=3$, we have that $\dim(\mathbb{C}\mathfrak{S}_4)=24$, hence the strict inequality of dimensions.

In particular, the set $\{ s_\omega \mid \omega \in B_3(\bar{1}\bar{2},\bar{2}\bar{1})\}$, that is the set of elements $s_\omega$ where $x_1$ appears at most 1 time, is not linearly independant in $\mathcal{A}_3^{\kappa,(1)}$ if the latter is semisimple but it is for $\mathcal{A}_3^{(0,2),(1)}$ (see Theorem \ref{Iso}) which is not semisimple.

This same phenomenon occurs for larger values of $k$.

For instance, for $k=2$, the bipartitions which lie in $Irr(\mathcal{A}_4^{\kappa,(2)})$ are exactly those in which the first component is a partition with at most 2 columns and the second component is a partition of length at most 1, namely:
\begin{align*}
Irr(\mathcal{A}_4^{\kappa,(2)})=&\left\lbrace ((2,2),\emptyset),((2,1,1),\emptyset),((1,1,1,1),\emptyset),((2,1),(1)),((1,1,1),(1)),((2),(2)),\right. \\
&\left. ((1,1),(2)),((1),(3)),(\emptyset,(4)) \right\rbrace.
\end{align*}

From the equality \eqref{dimbipart} and the isomorphism \eqref{AW formula} for $\mathfrak{A}=\mathcal{A}_4^{\kappa,(2)}$, we obtain that: \[\dim(\mathcal{A}_4^{\kappa,(2)})=183,\] while we have for the fused permutations algebra the following dimension (see \eqref{eq: dim H_{k,n}}):
\[\dim(H_{2,4})=\sum_{i=0}^2\binom{4}{i}^2(4-i)!=192.\]

Hence we have that the inequality \eqref{inequality 2} is strict in this case.

In particular, the set $\{ s_\omega \mid \omega \in B_4(\bar{1}\bar{2},\bar{3}\bar{2}\bar{1})\}$ is not linearly independant in $\mathcal{A}_4^{\kappa,(2)}$ when the latter is semisimple but it is for $\mathcal{A}_4^{(0,3),(2)}$ which is not semisimple.

\bigskip
Recall that the the basis \eqref{basis set} is well adapted to the inclusion $B_{n-1}\subset B_n$. Then, for $k<n$, we identify an element of $B_k$ as an element of $B_n$ via this inclusion.
\begin{Cor}
The algebra $H_{k,n}$ is generated by the generators $x_1, s_1,\ldots ,s_{n-1}$ subject only to the relations \eqref{eq: DR1}-\eqref{eq: DR4} with the following two additional relations:
\begin{equation}\label{eq: (x_1s_1)^2}
(x_1s_1)^2=x_1s_1x_1-x_1s_1+x_1s_1,
\end{equation}
\begin{equation}\label{eq: k+1}
\Delta_k=\sum_{\omega\in B_{k+1}(\bar{1}\bar{2})\setminus \lbrace \Delta_k\rbrace}\lambda_\omega s_\omega,
\end{equation}
where the $\lambda_\omega$ are the coefficients appearing in \eqref{new quo}.
\end{Cor}
\begin{proof}
By definition, the algebra $\mathcal{A}_n^{(0,k+1),(k)}$ (see Definition \ref{def: A_n^k}) is exactly the algebra mentionned in the statement.  

Hence, the result follows from the isomorphism (\ref{Isomorphism}) between the algebras $\mathcal{A}_n^{(0,k+1),(k)}$ and $H_{k,n}$.
\end{proof}

\end{document}